\documentclass[12pt, reqno]{amsart}
\usepackage[margin=1.0in]{geometry}

\usepackage{graphicx,amsmath,amssymb,amsthm}
\usepackage{enumitem} 

\newtheorem{thm}{Theorem}
\newtheorem{lem}{Lemma}

\theoremstyle{definition}
\newtheorem{defn}{Definition}

\theoremstyle{remark}
\newtheorem{rem}{Remark}
\newtheorem{conjecture}{Conjecture}
\newtheorem{prop}{Proposition}
\newtheorem{cor}{Corollary}
\newtheorem{example}{Example}

\numberwithin{equation}{section}
\numberwithin{thm}{section}
\numberwithin{lem}{section}
\numberwithin{defn}{section}
\numberwithin{rem}{section}
\numberwithin{conjecture}{section}
\numberwithin{prop}{section}
\numberwithin{cor}{section}
\numberwithin{example}{section}

\newtheorem*{example1}{Example \ref{ex1}$^\prime$}
\newtheorem*{example2}{Example \ref{ex2}$^\prime$}

\newcommand{\RR}{\mathbb{R}}

\newcommand{\CC}{\mathbb{C}}

\newcommand{\itemizeEqnVSpacing}{\rule{0pt}{1pt}\vspace*{-12pt}}

\newcommand{\pts}{\mathcal P}
\newcommand{\incidences}{\mathcal I}
\newcommand{\sphrs}{\mathcal S}

\newcommand{\sing}{\operatorname{sing}}
\newcommand{\smooth}{\operatorname{smooth}}
\newcommand{\res}{\operatorname{res}}
\newcommand{\edges}{\mathcal{E}}
\newcommand{\crossing}{\mathcal{C}}
\newcommand{\vertices}{\mathcal{V}}
\newcommand{\BZ}{\mathbf{Z}}
\newcommand{\reali}{\operatorname{Reali}}

\newcommand{\mult}{\operatorname{mult}}
\newcommand{\grass}{\operatorname{Gr}}
\newcommand{\lines}{\mathcal{L}}
\newcommand{\cross}{\mathcal{C}}
\newcommand{\edgemult}{\mathrm{edgemult}}

\newcommand{\bdry}{\partial}

\begin{document}
\title{A Szemer\'edi-Trotter type theorem in $\RR^4$}
\author[J. Zahl]{Joshua Zahl} %
\address{Department of Mathematics, MIT, Cambridge MA 02139-4307, USA}
\email{jzahl@mit.edu}
\date{\today}
\begin{abstract} 
We show that $m$ points and $n$ two-dimensional algebraic surfaces in $\mathbb{R}^4$ can have at most $O(m^{\frac{k}{2k-1}}n^{\frac{2k-2}{2k-1}}+m+n)$ incidences, provided that the algebraic surfaces behave like pseudoflats with $k$ degrees of freedom, and that $m\leq n^{\frac{2k+2}{3k}}$. As a special case, we obtain a Szemer\'edi-Trotter type theorem for 2--planes in $\mathbb{R}^4$, provided $m\leq n$ and the planes intersect transversely. As a further special case, we obtain a Szemer\'edi-Trotter type theorem for complex lines in $\mathbb{C}^2$ with no restrictions on $m$ and $n$ (this theorem was originally proved by T\'oth using a different method). As a third special case, we obtain a Szemer\'edi-Trotter type theorem for complex unit circles in $\mathbb{C}^2$. We obtain our results by combining several tools, including a two-level analogue of the discrete polynomial partitioning theorem and the crossing lemma.
\end{abstract}
\maketitle

\section{Introduction}
In \cite{Erdos}, Erd\H{o}s asked how many incidences could occur between a collection of $m$ points and $n$ lines in the plane. The correct asymptotic bound was found by Szemer\'edi and Trotter in \cite{Szemeredi}. They proved what is now known as the Szemer\'edi-Trotter theorem.
\begin{thm}[Szemer\'edi-Trotter]\label{szemerediTrotterThm}
The number of incidences between $m$ points and $n$ lines in $\RR^2$ is $O(m^{2/3}n^{2/3}+m+n)$. This bound is tight (up to the implicit constant in the $O(\cdot)$ notation).
\end{thm}

This theorem has seen a number of additional proofs, including one by Sz\'ekely \cite{Szekely} which used the crossing lemma (see \cite{Ajtai,Leighton}). In \cite{Pach}, Pach and Sharir built off Sz\'ekely's ideas and proved a Szemer\'edi-Trotter type theorem for curves with $k$ degrees of freedom.
\begin{defn}
Let $\pts\subset\RR^2$ be a set of points and let $\Gamma$ be a set of simple open plane curves. We say the collection $(\pts,\Gamma)$ has $k$ degrees of freedom and multiplicity type $C_0$ if the following conditions hold.
\begin{itemize}
\item For any $k$ distinct points from $\pts$, there are at most $C_0$ curves from $\Gamma$ passing through all of them.
\item Any pair of curves from $\Gamma$ intersect in at most $C_0$ points.
\end{itemize}
\end{defn}

\begin{thm}[Pach-Sharir \cite{Pach}]
Let $\pts$ be a collection of $m$ points and let $\Gamma$ be a collection of $n$ simple open plane curves. Suppose that $(\pts,\Gamma)$ has $k$ degrees of freedom and multiplicity type $C_0$. Then the number of incidences between the points and curves is
\begin{equation*}
O\Big(m^{\frac{k}{2k-1}}n^{\frac{2k-2}{2k-1}}+m+n\Big).
\end{equation*}
The implicit constant depends only on $k$ and $C_0$.
\end{thm}

Similar bounds had previously been known in the special case that the curves in $\Gamma$ were real algebraic curves of bounded degree \cite{Clarkson,Pach2}. In this paper we will discuss a variant of the above theorems that applies to certain families of points and two-dimensional algebraic surfaces in $\RR^4.$ First, we will need several definitions. 
\begin{defn}\label{goodSurfacesDefn}
Let $\sphrs$ be a collection of two-dimensional real algebraic surfaces in $\RR^4$ and let $C_0\geq 1,\ k\geq 1$ be integers. We say that $\sphrs$ is a $C_0$--\emph{good collection of pseudoflats with $k$ degrees of freedom} if the following conditions are satisfied
 \begin{enumerate}[label=(\roman{*}), ref=(\roman{*})]
 \item\label{smoothSurfacesSmooth} Every surface in $\sphrs$ is smooth (i.e.~it is a two-dimensional manifold) and can be defined by a (single) polynomial of degree at most $C_0$.
 \item\label{noCommonComponents} If $S,S^\prime\in\sphrs$ then $|S\cap S^\prime|\leq C_0$.
 \item\label{tangentSpacesTransverse} If $p_1,\ldots,p_k\in\RR^4$ are distinct points, then at most $C_0$ surfaces from $\sphrs$ contain each of the points $p_1,\ldots,p_k$.
\end{enumerate}
\end{defn}

Given a collection of points $\pts\subset\RR^4$ and a collection of surfaces $\sphrs$ in $\RR^4$, we define the set of incidences between $\pts$ and $\sphrs$ to be
\begin{equation*}
\mathcal{I}(\pts,\sphrs)=\{(p,S)\in\pts\times\sphrs\colon p\in S\}.
\end{equation*}
\begin{defn}\label{goodCollectionIncidencesDefn}
Let $\pts$ be a collection of points, let $\sphrs$ be a collection of two-dimensional real algebraic surfaces in $\RR^4$, and let $I\subset \mathcal{I}(\pts,\sphrs)$. We say that $I$ is a \emph{good collection of incidences} if whenever $(p,S),(p,S^\prime)\in I$, we have that $T_p(S)\cap T_p(S^\prime)=p$, i.e. whenever two surfaces are incident to a common point, their tangent planes intersect transversely.
\end{defn}

We are now ready to state our results.
\begin{thm}[Point-surface incidences in $\RR^4$]\label{surfacesInR4Thm}
Let $\pts\subset\RR^4$ be a collection of $m$ points. Let $\sphrs$ be a $C_0$-good collection of pseudoflats with $k$ degrees of freedom. Let $n=|\sphrs|$, and suppose $m\leq n^{\frac{2k+2}{3k}}$. Let $I\subset\incidences(\pts,\sphrs)$ be a good collection of incidences. Then
\begin{equation}\label{surfacesInR4ThmIntm}
|I| \leq C_1\Big(m^{\frac{k}{2k-1}}n^{\frac{2k-2}{2k-1}}+m+n\Big).
\end{equation}
The constant $C_1$ depends only on $C_0$ and $k$.
\end{thm}
\begin{rem}
The requirement that $m\leq n^{\frac{2k+2}{3k}}$ is a limitation arising from our proof techniques. We conjecture that the inequality holds for all $m$ and $n$. For $m\geq n^2$, the inequality is trivial. In section \ref{openProblemsSec}, we will discuss progress towards weakening the above restriction on $m$ and $n$.
\end{rem}

\begin{cor}[The $k=2$ case]\label{szemerediTrotterKIs2}
 Let $\pts\subset\RR^4$ be a collection of $m$ points. Let $\sphrs$ be a $C_0$-good collection of pseudoflats with $2$ degrees of freedom. Let $n=|\sphrs|$, and suppose $m\leq n$. Let $I\subset\incidences(\pts,\sphrs)$ be a good collection of incidences. Then
\begin{equation}\label{surfacesInR4ThmIntm}
|I| = O\big(m^{2/3}n^{2/3}+m+n\big).
\end{equation}
The implicit constant depends only on $C_0$.
\end{cor}

\begin{cor}[$2$--planes in $\RR^4$]\label{szemerediTrotter2Planes}
Let $\pts\subset\RR^4$ be a collection of $m$ points and let $\sphrs$ be a collection of $n$ 2--planes in $\RR^4$ such that any two planes meet in at most one point. Suppose that $m\leq n$. Then the number of point-plane incidences is
\begin{equation}
O\big(m^{2/3}n^{2/3}+m+n\big).
\end{equation}
\end{cor}

We can use Corollary \ref{szemerediTrotter2Planes} to recover the Szemer\'edi-Trotter theorem for complex lines in $\CC^2$, which was originally proved by T\'oth in \cite{Toth}. Note that by point-line duality in $\CC^2$, we can always assume that the number of lines is at least as great as the number of points. Thus we have:
\begin{cor}[complex lines]
Let $\pts$ be a collection of $m$ points and let $\sphrs$ be a collection of $n$ (complex) lines in $\CC^2$. Then the number of point-line incidences is
\begin{equation}
O\big(m^{2/3}n^{2/3}+m+n\big).
\end{equation}
\end{cor}

As another corollary,  we obtain a Szemer\'edi-Trotter type theorem for complex unit circles in the complex plane. If $z=(z_1,z_2)\in\CC^2$, we can define the complex unit circle centered at $z$ to be the set $C_z=\{(w_1,w_2)\in\CC^2\colon (z_1-w_1)^2+(z_2-w_2)^2=1\}$. If we identify $\CC^2$ with $\RR^4$, then $C_z$ becomes a smooth two-dimensional surface defined by a single polynomial of degree 4. If $\sphrs$ is a collection of such surfaces, then $\sphrs$ is a $4$-good collection of pseudoflats with 2 degrees of freedom. If $\pts\subset\RR^4$ is a collection of points, then we can partition $\incidences(\pts,\sphrs)$ into $O(1)$ collections $I_1,\ldots,I_{O(1)}$ so that each is a good collection of incidences (see
\cite[Corollary 2.7]{Solymosi} for details). Finally, if $\pts\subset\CC^2$ is a collection of points and $\sphrs\subset\CC^2$ is a collection of (complex) unit circles, then by point-unit circle duality we can always assume that $|\pts|\leq|\sphrs|$. Thus we can apply Theorem \ref{surfacesInR4Thm} to each of the collections $I_1,\ldots,I_{O(1)}$ to conclude the following:
\begin{cor}[Complex unit circles]
Let $\pts\subset\CC^2$ be a collection of $m$ points, and let $\sphrs$ be a collection of $n$ complex unit circles in $\CC^2$. Then the number of point-circle incidences is $O\big(m^{2/3}n^{2/3}+m+n\big)$.
\end{cor}

\subsection{Previous work}
In \cite{Toth}, T\'oth extended Szemer\'edi and Trotter's original proof to the complex plane. T\'oth's result is specific to complex lines, so it does not (for example) apply to complex unit circles. Solymosi and Tardos \cite{Solymosi3} gave a simpler proof of the same bound in the special case where the point set is a Cartesian product of the form $A\times B\subset\CC^2$.

Edelsbrunner and Sharir \cite{Edelsbrunner} obtained incidence results for certain configurations of points and codimension--one hyperplanes in $\RR^4$, and \L{}aba and Solymosi \cite{Laba} obtained incidence bounds for points and a general class of two-dimensional surfaces in $\RR^3$, provided the points satisfied a certain homogeneity condition.

Elekes and T\'oth \cite{Elekes} and later Solymosi and T\'oth \cite{Solymosi2} obtained incidence results between points and hyperplanes in $\RR^d$, again provided the points satisfied various non-degeneracy and homogeneity conditions.

In \cite{Solymosi}, Solymosi and Tao  used the discrete polynomial partitioning theorem (this is \cite[Theorem 4.1]{Guth}, also Theorem \ref{GuthKatzPartitioning}) to obtain bounds for the number of incidences between points and bounded degree algebraic surfaces satisfying certain non-degeneracy and pseudoflat conditions (i.e.~they behaved similarly to hyperplanes). Aside from an $\epsilon$ loss in the exponent, Solymosi and Tao's result resolved a conjecture of T\'oth on the number of incidences between points and $d$--flats in $\RR^{2d}$ (T\'oth conjectured that Solymosi and Tao's result should hold without the $\epsilon$ loss in the exponent \cite[Conjecture 3]{TothV1}). The discrete polynomial partitioning theorem was also used by the author in \cite{Zahl} to obtain incidence results between points and two-dimensional surfaces in $\RR^3$ (with no homogeneity condition), and by Kaplan et al.~ in \cite{Kaplan2} to obtain similar bounds on the number of incidences between points and spheres in $\RR^3.$
\subsection{Proof sketch}
To keep the proof sketch simple, we shall assume that the surfaces in $\sphrs$ are 2--planes, and that every pair of 2--planes are either disjoint or intersect transversely. The actual proof of the theorem (presented in the following sections) will not make these assumptions. The basic idea is as follows. By the assumption that 2--planes must intersect transversely, there can be at most one 2-plane passing through any pair of points. Thus we can use the Cauchy-Schwarz inequality to obtain a rudimentary bound on the cardinality of any collection of point-surface incidences. We will call this the Cauchy-Schwarz bound.

Using the discrete polynomial partitioning theorem, we find a polynomial $P$ of controlled degree (the degree will be a suitable power of $m$ and $n$) so that $\RR^4\backslash \BZ_{\RR}(P)$ is a union of open \emph{cells}, such that each cell contains roughly the same number of points from $\pts$, and no surface from $\sphrs$ enters too many cells (here $\BZ_{\RR}(P)$ is the set of points in $\RR^4$ at which $P$ vanishes). We can then apply the Cauchy-Schwarz bound within each cell. This allows us to count the incidences occurring between surfaces and points in $\pts\backslash \BZ_{\RR}(P)$. In order to count the remaining incidences, we perform a second level polynomial partitioning decomposition on the variety $\BZ_{\RR}(P)$. This gives us a polynomial $Q$ which cuts $\BZ_{\RR}(P)$ into a collection of three-dimensional cells, which are open in the relative (Euclidean) topology of $\BZ_{\RR}(P)$. We then apply the Cauchy-Schwarz bound to each of these three-dimensional cells. The only incidences left to count are those between surfaces in $\sphrs$ and points in $\pts\cap \BZ_{\RR}(P)\cap \BZ_{\RR}(Q)$.

We can choose $P$ and $Q$ in such a way that $\BZ_{\RR}(P)\cap \BZ_{\RR}(Q)$ is a two-dimensional variety in $\RR^4$. Let $S$ be a 2--plane from $\sphrs.$ Then $S$ will intersect $\BZ_{\RR}(P)\cap \BZ_{\RR}(Q)$ in a union of isolated points (proper intersections) and one-dimensional curves (non-proper intersections); the case where $S$ meets $\BZ_{\RR}(P)\cap \BZ_{\RR}(Q)$ in a two-dimensional variety can be dealt with easily. The number of isolated points in the intersection can be bounded by the degrees of the polynomials $P$ and $Q$ (we are working over $\RR,$ where B\'ezout's theorem need not hold, so we need to be a bit careful). Thus the number of incidences between points $p\in \pts\cap \BZ_{\RR}(P)\cap \BZ_{\RR}(Q)$ and surfaces $S\in\sphrs$ such that $p$ is an isolated point of $S\cap \BZ_{\RR}(P)\cap \BZ_{\RR}(Q)$ can be bounded.

The only remaining task is to bound the number of incidences between points of $\pts\cap \BZ_{\RR}(P)\cap \BZ_{\RR}(Q)$ and one-dimensional curves arising from the intersection of $\BZ_{\RR}(P)\cap \BZ_{\RR}(Q)$ and surfaces $S\in\sphrs$. To simplify the exposition, we will pretend (in this sketch only!) that $\BZ_{\RR}(P)\cap \BZ_{\RR}(Q)$ is a disjoint union of $N$ 2--planes, i.e.~$\BZ_{\RR}(P)\cap \BZ_{\RR}(Q)=\Pi_1\sqcup\ldots\sqcup\Pi_N$. Then for each plane $\Pi_i,$ $\Pi_i\cap S=L_{S,i}$ is a line on $\Pi_i$. It remains to count the number of incidences between $\pts\cap\Pi_i$ and $\{L_{S,i}\}_{S\in\sphrs}.$ The Szemer\'edi-Trotter theorem for lines in $\RR^2$ would give us the bound
\begin{equation}\label{2dST}
I(\pts\cap\Pi_i,\{L_{S,i}\}_{S\in\sphrs})=O\big( |\pts\cap\Pi_i|^{2/3}|\sphrs|^{2/3}+|\pts\cap\Pi_i|+|\sphrs|\big).
\end{equation}
However, if we sum \eqref{2dST} over the $N$ values of $i$, we have only bounded the number of incidences by
\begin{equation}\label{badBound}
O\big(N^{1/3}|\pts|^{2/3}|\sphrs|^{2/3}+|\pts|+|\sphrs|\big).
\end{equation}
Since $N$ can be quite large (for example, if $|\pts|=|\sphrs|$, then $N$ could be as large as $|\pts|^{1/3}$), this is not sufficient. Instead, recall Sz\'ekely's proof in \cite{Szekely} of the Szemer\'edi-Trotter theorem, which uses the crossing lemma (the crossing lemma and all other graph-related results are introduced in Section \ref{graphTheorySection} below). Loosely speaking, we consider the graph drawing $H_i$ on $\Pi_i$ whose vertices are the points of $\pts\cap\Pi_i$, and two vertices are connected by an edge if there is a line from $\{L_{i,S}\}_{S\in\sphrs}$ passing through the two points, and the two points are adjacent on the line (i.e.~there are no points in between them). Then the number of edges of the graph is comparable to the number of incidences between points and lines, and this is bounded by $\crossing(H_i)^{1/3}V_i^{2/3}$, where $\crossing(H_i)$ is the number of times two edges cross in the drawing $H_i$, and $V_i$ is the number of vertices of $H_i$. Thus in place of \eqref{2dST}, we have
\begin{equation}\label{2dSTCrossings}
I(\pts\cap\Pi_i,\{L_{S,i}\}_{S\in\sphrs})=O\big( |\pts\cap\Pi_i|^{2/3}|\crossing(H_i)|^{1/3}+|\pts\cap\Pi_i|+|\sphrs|\big).
\end{equation}
The key insight is that
\begin{equation}\label{boundOnCrossingNum}
\sum_{i}|\crossing(H_i)| \leq |\sphrs|^2.
\end{equation}
Indeed, every pair of 2--planes $S, S^\prime\in\sphrs$ can intersect in at most one point, and since we assumed the planes $\{\Pi_i\}$ composing $\BZ_{\RR}(P)\cap \BZ_{\RR}(Q)$ were disjoint, the intersection point of $S\cap S^\prime$ can occur on $\Pi_i$ for at most one index $i$. Thus we have
\begin{equation}
\begin{split}
\sum_i &I(\pts\cap\Pi_i,\{L_{S,i}\}_{S\in\sphrs})\\
&=O\Big( \sum_i\Big(|\pts\cap\Pi_i|^{2/3}|\crossing(H_i)|^{1/3}+|\pts\cap\Pi_i|+|\sphrs|\Big)\Big)\\
&=O\Big( \Big(\sum_i|\pts\cap\Pi_i|\Big)^{2/3}\Big(\sum_i|\crossing(H_i)|\Big)^{1/3}+\sum_i|\pts\cap\Pi_i|+\sum_i|\sphrs|\Big)\\
&=O\big( |\pts|^{2/3}|\sphrs|^{2/3}+|\pts|+N|\sphrs|\big)\\
&=O\big(m^{2/3}n^{2/3}+m+Nn\big).
\end{split}
\end{equation}

This is a much better bound than \eqref{badBound}, and it gives us the desired bound on the number of incidences between surfaces in $\sphrs$ and points lying on $\BZ_{\RR}(P)\cap \BZ_{\RR}(Q)$.

Unfortunately, the assumption that $\BZ_{\RR}(P)\cap \BZ_{\RR}(Q)$ is a disjoint union of 2--planes need not be true, and dealing with this difficulty will occupy the bulk of the paper. To handle this, we must cut $\BZ_{\RR}(P)\cap \BZ_{\RR}(Q)$ into pieces, each of which behaves like a 2--plane (more accurately, each piece is homeomorphic to an open set in $\RR^2$), and we need to prove a more general form of the (planar) Szemer\'edi-Trotter theorem which gives an incidence bound for an arrangement of points of curves based on the number of curve crossings, rather than the number of curves.

\subsection{Major tools and techniques}
We will give a brief overview of the main tools that will be used in the proof of Theorem \ref{surfacesInR4Thm}. First, we will require some results from real algebraic geometry. Specifically, we will make use of Barone and Basu's refined bounds on the number of sign conditions of a real algebraic variety. This will be discussed further in Section \ref{realAlgGeomSec}. Polynomial partitionings will also play a central role. These were first developed by Guth and Katz, and later extended by Kaplan, Matou\v{s}ek, Safernova, and Sharir, and by the author. These partitioning will be discussed in Section \ref{polySandwichSec}.

We will use some results from algebraic geometry and intersection theory. These will be discussed in Section \ref{forayAlgGeoSec}. We will also require some elementary results from differential geometry. This will be discussed in section \ref{diffGeomSec}. Finally, we will make use of the crossing lemma from topological graph theory. This will be discussed in Section \ref{graphTheorySection}.
\subsection{Notation}
Throughout this paper, $C,C_1,C_2,\ldots$ will denote large constants, and $c,c_1,$ $c_2,\ldots$ will denote small (positive) constants.
\begin{defn}\label{defnOflesssim}
We will say that $A\lesssim B$ or $A=O(B)$ if $A\leq CB$ for some constant $C$ that depends only on $C_0$ and $k$ from the statement of Theorem \ref{surfacesInR4Thm}, and possibly the ambient dimension $d$ (in the statement of Theorem \ref{surfacesInR4Thm} we have $d=4$. However we will sometimes state results in greater generality). Sometimes the constant $C$ will appear to depend on other parameters as well. However, these additional parameters will ultimately only depend on $C_0$ and $k$. If $A\lesssim B$ and $B \lesssim A$, we say $A\approx B$ or  $A=\Theta(B)$ .
\end{defn}

\subsection{Thanks}
The author is very grateful to Saugata Basu, Kiran Kedlaya, Silas Richelson, Terence Tao, and Burt Totaro for helpful discussions. The author would like to especially thank the anonymous referees for their careful reading and numerous suggestions. Referee \#1 in particular went above and beyond the usual refereeing process, and the author is very grateful for the time and effort he or she put in. The author was supported in part by the Department of Defense through the National Defense Science \& Engineering Graduate Fellowship (NDSEG) Program.
\section{Preliminaries: real algebraic geometry and polynomial partitions}\label{preliminariesSection}
\subsection{Real algebraic geometry}\label{realAlgGeomSec}
\subsubsection{Ideals and varieties}
Unless otherwise noted, all polynomials will be (affine) real polynomials, i.e.~elements of $\RR[x_1,\ldots,x_d]$. In the first sections of this paper we will deal mainly with real affine varieties; in later sections we will be concerned with both real and complex varieties. Definitions and standard results about real algebraic varieties can be found in \cite{Basu, Bochnak}.

\begin{defn}
A (real) algebraic variety $Z\subset\RR^d$ is a set of the form $Z=\bigcap_{i=1}^\ell \{P_i=0\},$ where $P_1,\ldots,P_\ell\in\RR[x_1,\ldots,x_d]$ are polynomials. Note that we do not require varieties to be irreducible.
\end{defn}

\begin{defn}
If $J\subset\RR[x_1,\ldots,x_d]$ is an ideal, we define
\begin{equation*}
\BZ_{\RR}(J)=\{x\in\RR^d\colon f(x)=0\ \textrm{for all}\ f\in J\}.
\end{equation*}
By abuse of notation, if $P\in\RR[x_1,\ldots,x_d]$, then we define $\BZ_{\RR}(P) = \BZ_{\RR}((P))$, where $(P)$ is the ideal generated by $P$. Sometimes we will also need to work over $\CC$. We define $\BZ_{\CC}(J)$ analogously, with $\CC^d$ in place of $\RR^d$.
\end{defn}
\begin{defn}
If $Z\subset\RR^d$ is a variety, we define
\begin{equation*}
\mathbf{I}(Z) = \{P\in\RR[x_1,\ldots,x_d]\colon\ P(x)=0\ \textrm{for all}\ x\in Z\}.
\end{equation*}
\end{defn}

\subsubsection{Smooth points and dimension of a real variety}
\begin{defn}
Let $Z\subset\RR^d$ be a real algebraic variety, and let $z\in Z$. We define the dimension $\dim_{\RR,z}(Z)$ of $Z$ at $z$ as in Definition 2.8.11 of \cite{Bochnak}. Informally, if $\dim_{\RR,z}(Z)=e,$ then we can find a homeomorphism from a small (Euclidean) neighborhood of $z\in Z$ to the $e$--dimensional cube $(0,1)^e$. We define
\begin{equation*}
\dim_{\RR}(Z)=\sup_{z\in Z}\dim_{\RR,z}(Z).
\end{equation*}
\end{defn}

To avoid confusion, if $Z\subset\CC^d$ is a (complex) variety, we will denote the (complex) dimension of $Z$ by $\dim_{\CC}(Z)$.
\begin{defn}\label{smoothLocusDefn}
We define the smooth locus $Z_{\smooth}$ as in Section 3.3 of \cite{Bochnak} (to be more precise, we say a point $z$ is in $Z_{\smooth}$ if $z$ is smooth in dimension $e=\dim_{\RR}(Z)$). Informally, if $\mathbf{I}(Z)=(f_1,\ldots,f_\ell)\subset\RR[x_1,\ldots,x_d]$ and $e=\dim_{\RR}(Z)$, then $z\in Z$ is a smooth point of $Z$ if
\begin{equation*}
 \operatorname{rank}\left[\begin{array}{c}\nabla f_1\\\vdots\\\nabla f_n\end{array}\right]=d-e.
\end{equation*}
Here and throughout this paper, if $f\in\RR[x_1,\ldots,x_d]$, then $\nabla f$ is the vector-valued function $(\frac{df}{dx_1},\ldots,\frac{df}{dx_d})$. If $f\in\CC[x_1,\ldots,x_d]$, we can define $\nabla f$ similarly.
\end{defn}

If $z\in Z$ is a smooth point, then $Z$ is a $e$--dimensional real manifold in a small (Euclidean) neighborhood of $z$. However the converse need not hold.

\subsubsection{Real ideals}
Ideals and varieties over $\RR$ can have some rather pathological properties. Luckily, there is a class of ideals over $\RR$ that behave more sanely. Confusingly, these ideals are called \emph{real ideals}.

\begin{defn}\label{defnOfRealIdeal}
 An ideal $J\subset\RR[x_1,\ldots,x_d]$ is \emph{real} if for every sequence $a_1,\ldots,a_\ell\in\RR[x_1,\ldots,x_d]$, $a_1^2+\ldots+a_\ell^2\in J$ implies $a_j\in J$ for each $j=1,\ldots,\ell$.
\end{defn}
The following proposition shows that real principal prime ideals and their corresponding real varieties have some of the nice properties of ideals and varieties defined over $\CC$.
\begin{prop}[see {\cite[\S 4.5]{Bochnak}} ]\label{propertiesOfPrincipleRealIdealProp}
 Let $(P)\subset\RR[x_1,\ldots,x_d]$ be a principal prime ideal. Then the following are equivalent:
 \begin{enumerate}[label=(\roman{*}), ref=(\roman{*})]
 \item $(P)$ is real.
 \item\label{nullz} $(P) = \mathbf I (\BZ(P))$.
 \item\label{smallDim} $\dim_{\RR} (\BZ (P))=d-1$.
 \item \label{notVanishIdentically} $\nabla P$ does not vanish identically on $\BZ(P)$.
 \item\label{noSignChange} The sign of $P$ changes somewhere on $\RR^d$.
 \end{enumerate}
 \end{prop}
\begin{rem}
In  \cite{Bochnak}, Proposition \ref{propertiesOfPrincipleRealIdealProp} is stated in the more general language of real closed fields. However, $\RR$ is an example (indeed, the motivating example) of a real closed field, so the proposition applies to ideals in $\RR[x_1,\ldots,x_d]$.
\end{rem}

While not every polynomial $P\in\RR[x_1,\ldots,x_d]$ is a product of irreducible polynomials that generate real ideals, the following lemma shows that for our applications, we can always modify our polynomials to ensure that this is the case.
\begin{lem}\label{makeAPolyProdRealIdeals}
Let $P\in\RR[x_1,\ldots,x_d]$ be a real polynomial. Then there exists a non-zero real polynomial $\tilde P$ such that $\deg\tilde P\leq\deg P,$ $\BZ_{\RR}(P)\subset\BZ_{\RR}(\tilde P),$ and the irreducible components of $\tilde P$ generate real ideals.
\end{lem}
\begin{proof}
We will prove the statement by induction on $\deg P$. If $\deg P=1$, then $P$ is irreducible and $(P)$ is real. Now suppose the statement has been proved for all polynomials of degree $\leq D$. Let $P$ be a polynomial of degree $D+1$ and factor $P=P_1\cdots P_a.$ If all the irreducible factors of $P$ generate real ideals, we are done. If not, then without loss of generality we can assume that $(P_1)$ does not generate a real ideal. In particular, $\deg P_1\geq 2$. Let $v$ be a generic (with respect to $P$) real vector\footnote{Over $\RR$ one must be very careful with the phrase ``generic,'' but informally, a generic vector is any vector that does not lie in a certain bad set that has smaller dimension than the entire vector space. Often the bad set will not be defined explicitly, but will be determined from the list of properties we wish the generic vector to have. A precise definition of a generic real vector is given in Section \ref{genericPtsSection}}. Then $v\cdot\nabla P_1$ is a non-zero polynomial of degree $\deg(P_1)-1$, and by Proposition \ref{propertiesOfPrincipleRealIdealProp}\ref{notVanishIdentically}, $\BZ_{\RR}(P_1)\subset \BZ_{\RR}(v\cdot\nabla P_1)$. Let $P^\prime = (v\cdot\nabla P_1)P_2\cdots P_a$. Then $\deg P^\prime\leq D$ and $\BZ_{\RR}(P)\subset \BZ_{\RR}(P^\prime)$. We can now apply the induction hypothesis to $P^\prime$ to find a polynomial $\tilde P$ so that $\deg(P)\leq\deg(P^\prime)\leq\deg(\tilde P)$, and $\BZ_{\RR}(P)\subset \BZ_{\RR}(P^\prime)\subset\BZ_{\RR}(\tilde P)$.
\end{proof}

\begin{rem}\label{gradOfPolyRealIdeal}
Examining the above proof, we see that if $P\in\RR[x_1,\ldots,x_d]$ is a square-free polynomial whose irreducible components generate real ideals, and if $v\in\RR^d$ is a generic vector, then $\BZ_{\RR}(P)_{\sing}\subset\BZ_{\RR}(P)\cap\BZ_{\RR}(v\cdot\nabla P)$.

If $\deg P=1$, then $\BZ_{\RR}(P)$ is smooth, so this statement is not very interesting. If $\deg P\geq 2$, then $v\cdot\nabla P$ is not the zero polynomial, and $P$ and $v\cdot\nabla P$ have no common components (over $\RR$). Since $P$ and $v\cdot\nabla P$ are real polynomials, this also implies that $P$ and $v\cdot\nabla P$ have no common components over $\CC$. In particular, this means the variety $\BZ_{\CC}(P)\cap\BZ_{\CC}(v\cdot\nabla P)$ has codimension two. Complex algebraic varieties will be discussed further in Section \ref{realCplxVarietiesSec}.
\end{rem}

\subsubsection{Sign conditions}\label{signConditionsSection}
Several of the results we will cite refer to \emph{strict sign conditions} or \emph{realizations of realizable strict sign conditions}. While we will not use sign conditions in our proof directly, it is useful to understand how they relate to the objects we will be studying.

\begin{defn}Let $\mathcal Q\subset\RR[x_1,\ldots,x_d]$ be a collection of non-zero real polynomials. A \emph{strict sign condition} on $\mathcal Q$ is a map $\sigma\colon \mathcal Q\to\{\pm 1\}$. If $Q\in\mathcal Q$, we will denote the evaluation of $\sigma$ at $Q$  by $\sigma_Q$.

If $Z\subset\RR^d$ is a variety and $\sigma$ is a strict sign condition on $\mathcal Q$, then we can define the \emph{realization of $\sigma$ on Z} by
\begin{equation}\label{aSignConditionOnZ}
\reali(\sigma,\mathcal Q,Z)=\{x\in Z\colon Q(x)\sigma_Q>0\ \textrm{for all}\ Q\in\mathcal Q\}.
\end{equation}
We define
\begin{equation}\label{setOfRealizableSignConditions}
\Sigma_{\mathcal Q,Z}=\{\sigma\colon \reali(\sigma,\mathcal Q,Z)\neq\emptyset\},
\end{equation}
and
\begin{equation}\label{RealizationsOfSignCondition}
\reali(\mathcal Q,Z)=\{\reali(\sigma,\mathcal Q,Z) \colon \sigma\in\Sigma_{\mathcal Q,Z}\}.
\end{equation}
We call $\reali(\mathcal Q,Z)$ the collection of \emph{realizations of realizable strict sign conditions of $\mathcal Q$ on $Z$.} Note that if some $Q\in\mathcal Q$ vanishes identically on $Z$ then $\Sigma_{\mathcal Q,Z}=\emptyset$ and thus $\reali(\mathcal Q,Z)=\emptyset$.
\end{defn}

The key observation is that if $\mathcal{Q}$ is a collection of non-zero real polynomials, $Q_0=\prod_{Q\in \mathcal{Q}}Q$, and if $Z\subset\RR^d$ is a variety, then every connected component of $Z\backslash\BZ_{\RR}(Q)$ is contained in some set from $\reali(\mathcal Q,Z)$.

\begin{rem}\label{signConditionRemark}
The above observation has two implications. First, the number of connected components of $Z\backslash\BZ_{\RR}(Q)$ bounds the number of sets in $\reali(\mathcal Q,Z)$. Second, suppose that $\pts\subset Z$ is a collection of points, and at most $C$ points from $\pts$ lie in any set from $\reali(\mathcal Q,Z)$. Then at most $C$ points lie in any connected component of $Z\backslash\BZ_{\RR}(Q)$.
\end{rem}

\subsection{The topology of real varieties: Milnor-Thom type theorems}
In the proof below, we will find a polynomial whose zero-set partitions Euclidean space into open cells, and we will apply a rudimentary incidence bound to bound the number of incidences inside each cell. To apply this rudimentary bound, we will need to control how many surfaces from $\sphrs$ enter each cell. The theorems in this section will give us the tools to do this.

\begin{thm}[Barone-Basu {\cite[Theorem 5]{Barone2}}, special case ]
Let $Q_1,\ldots,Q_\ell\in\RR[x_1,\ldots,x_d]$. Let $D_i = \deg(Q_i)$. For $i=1,\ldots,\ell$, let $\mathcal{Q}_i=\{Q_1,\ldots,Q_i\}$, and let $V_i=\bigcap_{j=1}^i \BZ_{\RR}(Q_j)$. Suppose that $\dim_{\RR}(V_i)\leq e_i$ for each index $i$ (by convention, $V_0=\RR^d,$ and $e_0 = d$). Let $P\in\RR[x_1,\ldots,x_d]$, and let $D=\deg P$.

Suppose that
\begin{equation}
2\leq D_1\leq D_2 \leq \frac{1}{d+1}D_3 \leq \frac{1}{(d+1)^2}D_4\leq\cdots\leq \frac{1}{(d+1)^{\ell-2}}D_\ell\leq D,
\end{equation}
and that $\ell\leq d$. Then the number of (Euclidean) connected components of the set
\begin{equation*}
\{x\in V_{\ell}:P(x)> 0\}
\end{equation*}
is bounded by

\begin{equation}\label{BaroneBasuEqnBd}
CD^{e_{\ell}}\prod_{j=1}^{\ell}D_j^{e_{j-1}-e_j},
\end{equation}
where the constant $C$ depends only on $d$.
\end{thm}
\begin{rem}
The bound \eqref{BaroneBasuEqnBd} is a special case of the bound from \cite[Theorem 5]{Barone2}. In \cite{Barone2},  \eqref{BaroneBasuEqnBd} appears as the inequality preceding Remark 1.13.

In \cite{Barone2}, Barone and Basu consider a family of polynomials, while in our formulation the family is just the singleton $\{P\}$. Furthermore, Barone and Basu bound the number of connected components of \emph{all} sign conditions of this family of polynomials on the variety $V_{\ell}$, while we are only interested in the sign condition $P>0$. Finally, Barone and Basu state their result for semialgebraically connected components over a real closed field. Since we are only interested in results over $\RR$, we can restrict our attention to Euclidean connected components.
\end{rem}
We will always be interested in the case $d=4$. We shall record three special cases that will be of particular interest to us
\begin{cor}\label{surfacesEntering4DCellCor}
Let $f,P\in \RR[x_1,\ldots,x_4]$. Suppose that

 \begin{enumerate}[label=(\roman{*}), ref=(\roman{*})]
 \item\label{dimZf2} $\dim_{\RR}(\BZ_{\RR}(f))=2$.
 \item\label{dimZfP1} $\dim_{\RR}(\BZ_{\RR}(f)\cap \BZ_{\RR}(P))=1$.
\end{enumerate}
Then
\begin{itemize}
 \item The number of connected components of $\{x\in \BZ_{\RR}(f)\colon P(x)>0\}$ is $O\big( (\deg P)^2\big)$.
 \item The number of connected components of $\BZ_{\RR}(f)\cap \BZ_{\RR}(P)$ is $O\big((\deg P)^2\big)$.
\end{itemize}
The implicit constants depend only on $\deg f$.
\end{cor}
\begin{cor}\label{surfacesEntering3DCellCor}
Let $f,P,Q\in \RR[x_1,\ldots,x_4]$. Suppose that $f$ and $P$ satisfy Properties \ref{dimZf2} and \ref{dimZfP1} from Corollary \ref{surfacesEntering4DCellCor}, and that $\deg P\leq C\deg Q$. Then
\begin{itemize}
 \item The number of connected components of
 \begin{equation*}
    \{x\in \BZ_{\RR}(f)\cap \BZ_{\RR}(P)\colon Q(x)>0\}
 \end{equation*}
is $O\big((\deg P)(\deg Q)\big).$
 \item The number of isolated points of
 \begin{equation*}
    \BZ_{\RR}(f)\cap \BZ_{\RR}(P)\cap \BZ_{\RR}(Q)
 \end{equation*}
is $O\big((\deg P)(\deg Q)\big)$.
\end{itemize}
Again, the implicit constant depends only on $\deg f$ and the constant $C$.
\end{cor}

\begin{rem}
Corollary \ref{surfacesEntering4DCellCor} (resp. Corollary \ref{surfacesEntering3DCellCor}) is only meaningful if the degree of $P$ (resp. $P$ and $Q$) is much larger than the degree of $f$. In practice, the degree of $f$ will be bounded by quantities that depend only on the constants $C_0$ and $k$ from the statement of Theorem \ref{surfacesInR4Thm}, while the degrees of $P$ and $Q$ will grow as the quantities $m$ and $n$ from the statement of Theorem \ref{surfacesInR4Thm} become larger.
\end{rem}

\subsection{Polynomial partitioning type theorems}\label{polySandwichSec}
In \cite{Guth}, Guth and Katz proved the following theorem.
\begin{thm}[Discrete polynomial partitioning theorem]\label{GuthKatzPartitioning}
Let $\pts$ be a set of $m$ points in $\RR^d$ and let $D\geq 1$ be an integer. Then there is a polynomial $P$ of degree at most $D$ with the following property: $\RR^d \backslash \BZ_{\RR}(P)$ is the union of $O(D^d)$ open connected sets (cells), and each cell contains $\leq m/D^d$ points of $\pts$.
\end{thm}

After applying Lemma \ref{makeAPolyProdRealIdeals}, we can ensure that the irreducible components of $P$ generate real ideals:
\begin{cor}\label{hamSandwichThmStronger}
Let $\pts$ be a set of $m$ points in $\RR^d$ and let $D\geq 1$ be an integer. Then there is a polynomial $P$ of degree at most $D$ with the following property: $\RR^d \backslash \BZ_{\RR}(P)$ is the union of $O(D^d)$ open connected sets (cells), and each cell contains $\leq m/D^d$ points of $\pts$. Furthermore, each irreducible component of $P$ generates a real ideal.
\end{cor}
\begin{example}\label{ex1}
Consider the following collection of 72 points:
\begin{equation}
\begin{split}
\pts = \bigcup_{j=1}^3 \{(\pm j ,\pm j ,\pm j,\pm j)\}\cup \bigcup_{j=1}^3 \{(0 ,\pm j ,\pm j,\pm j)\},
\end{split}
\end{equation}
and let $D=2$.
Then the degree--four polynomial
\begin{equation*}
P(x_1,x_2,x_3,x_4)= x_1x_2x_3x_4
\end{equation*}
cuts $\RR^4$ into 16 open cells $\Omega_1,\ldots,\Omega_{16},$ (the cells are unbounded, but this is fine) plus the set
\begin{equation*}
\BZ_{\RR}(P) = \bigcup_{i=1}^4\{x_i=0\}.
\end{equation*}
We can verify that the polynomials $x_1,\ldots,x_4$ generate real ideals, so $P$ is a product of irreducible polynomials, each of which generates a real ideal. We have $|\Omega_i\cap \pts|=3\leq |\pts|/D^4$ for each $i=1,\ldots,16$. Thus $P$ satisfies the requirements of Corollary \ref{hamSandwichThmStronger} (Corollary \ref{hamSandwichThmStronger} only specifies the degree of $P$ up to an implicit constant, so we cannot verify that the degree is correct). Finally, note that we have $|Z\cap\pts|=16.$
\end{example}

\begin{example}\label{ex2}
Let $\pts\subset\RR^4$ be a large collection of points that lie in general position on the 2--plane $\{x_1=x_2=0\}$, and let $D$ be much smaller than $|\pts|^{1/4}.$ Then we can verify that the polynomial $P(x_1,x_2,x_3,x_4)=x_1$ satisfies the requirements of Corollary \ref{hamSandwichThmStronger}; $\BZ_{\RR}(P)$ cuts $\RR^4$ into the two cells $\Omega_1=\{x_1>0\}$ and $\Omega_2=\{x_1<0\}$. We have $\Omega_1\cap\pts=\emptyset$ and $\Omega_2\cap\pts=\emptyset.$ This phenomenon is unavoidable: any polynomial $P$ satisfying the requirements of Corollary \ref{hamSandwichThmStronger} must contain a factor that vanishes on the 2--plane $\{x_1=x_2=0\}$ (provided the points of $\pts$ are in general position). Thus we must have $\pts\subset \BZ_{\RR}(P),$ so each of the cells of the decomposition $\RR\backslash\BZ_{\RR}(P)$ will be empty.

This example is interesting for the following reason. Let $(\pts_1,\mathcal{L}_1)$ be a collection of $m$ points and $n$ lines in $\RR^2$ that determine $\Theta(m^{2/3}n^{2/3}+m+n)$ incidences. Consider $(\pts_1,\mathcal{L}_1)$ as a collection of complex points and complex lines in $\CC^2$. Now, identify $\CC^2$ with $\RR^4$, and let $(\pts,\sphrs)$ be the corresponding collection of points and 2--planes in $\RR^4$. Then all of the points of $\pts$ will lie on a common 2--plane, so the situation will resemble this example.
\end{example}

Theorem \ref{GuthKatzPartitioning} will be used to obtain the first level decomposition of the point set $\pts$. However, as seen in the above examples, many points may lie on the boundary $\BZ_{\RR}(P),$ and we will need to bound the number of incidences between surfaces in $\sphrs$ and points on $\BZ_{\RR}(P)$. To do this, we shall perform a second discrete polynomial partitioning decomposition on the algebraic variety $\BZ_{\RR}(P)$.
\begin{thm}[Polynomial partitioning decomposition on a hypersurface]\label{variantHamSandwichThm}
Let $\pts$ be a collection of points in $\RR^d$ lying on the set $\BZ_{\RR}(P),$ where $P$ is an irreducible polynomial of degree $D$ that generates a real ideal. Let $E\geq cD$. Then there exists a polynomial $Q\in\RR[x_1,\ldots,x_d]$ with the following properties:
\begin{enumerate}[label=(\roman{*}), ref=(\roman{*})]
 \item\label{variantHamSandwichThmItm2} $\deg Q \leq C E$.
 \item\label{variantHamSandwichThmItm3} $Q$ does not vanish identically on $\BZ_{\RR}(P)$. In particular, $\dim_{\RR}(\BZ_{\RR}(P)\cap \BZ_{\RR}(Q))\leq d-2.$
 \item\label{variantHamSandwichThmItm4} The set $\BZ_{\RR}(P)\backslash \BZ_{\RR}(Q)$ is a union of $O(DE^{d-1})$ connected components (cells). Each cell contains at most $\frac{C|\pts|}{DE^{d-1}}$ points from $\pts$.
 \item\label{variantHamSandwichThmItm5} Each irreducible component of $Q$ generates a real ideal.
\end{enumerate}
The constant $C$ depends only on $c$ and the dimension $d$.
\end{thm}

Theorem \ref{variantHamSandwichThm} is proved in \cite[\S A.3]{Zahl}. The theorem is stated in terms of realizations of strict sign conditions (discussed in section \ref{signConditionsSection}) rather than cells. The version stated in \cite[\S A.3]{Zahl} bounds the number of points that can lie in any realization of a realizable strict sign condition on $\BZ_{\RR}(P)$. However, as noted in Remark \ref{signConditionRemark}, the version stated above follows immediately.

We can continue Examples \ref{ex1} and \ref{ex2}.
\begin{example1}
Let $\pts,$ $D$, $P,$ and $Z$ be as in Example \ref{ex1} above. Then $P_1(x_1,x_2,x_3,x_4)=x_1$ is the only irreducible component of $P$ whose zero-set contains points from $\pts$. 
Let $E=2$ and let $\mathcal Q=\{x_2,x_3,x_4\}.$ Then $\BZ_{\RR}(P_1)\backslash \BZ_{\RR}(Q)$ consists of the 8 octants of $\RR^3$, where we identify $\RR^3$ with the hyperplane $\{x_1=0\}$ in $\RR^4$. Each of these components contains 2 points from $\pts\cap \BZ_{\RR}(P_1)$, and
\begin{equation*}
\pts\cap \BZ_{\RR}(P_1)\cap \bigcup_{j=2}^4\{x_j=0\}=\emptyset,
\end{equation*}
i.e.~every point of $\pts$ either lies in some cell of $\RR^4\backslash \BZ_{\RR}(P_1)$ or some connected component of $\BZ_{\RR}(P_1)\backslash \BZ_{\RR}(Q)$.
\end{example1}

\begin{example2}
Let $\pts,$ $D$, and $P$ be as in Example \ref{ex2}, and let $E$ be much smaller than $|\pts|^{1/3}.$ Let $\mathcal Q = \{x_2\}$. Then $\BZ_{\RR}(P)\backslash \BZ_{\RR}(Q)$ consists of the sets $\{ x_1=0,x_2>0\}$ and $\{ x_1=0,x_2<0\}.$ Neither of these sets contain any points from $\pts\cap \BZ_{\RR}(P)$; indeed, $\pts\subset \{x_1=x_2=0\}=\BZ_{\RR}(P)\cap \BZ_{\RR}(Q)$. Thus $\mathcal Q$ satisfies the requirements of Theorem \ref{variantHamSandwichThm}, but none of the points of $\pts$ lie in any cell of $\RR^4\backslash \BZ_{\RR}(P)$ nor in any connected component of $\BZ_{\RR}(P)\backslash \BZ_{\RR}(Q)$. Sections \ref{forayAlgGeoSec}--\ref{boundingI6Sec} will be devoted to dealing with this type of situation.
\end{example2}
\section{Proof of Theorem \ref{surfacesInR4Thm} step 1: cell partitionings}
\subsection{Initial reductions}\label{initialRedSec}
Let $\pts,\sphrs,$ be as in the statement of Theorem \ref{surfacesInR4Thm}. First, it suffices to prove Theorem \ref{surfacesInR4Thm} in the special case where all the surfaces $S\in\sphrs$ are irreducible. If the surfaces are reducible, then each $S\in\sphrs$ can be written $S=S_1\cup S_2\cup\ldots\cup S_{C(S)}$, where $C(S)\leq C_0$ and each $S_i$ is a smooth irreducible two-dimensional surface. Since $S$ is smooth, the surfaces $\{S_1,\ldots,S_{C(S)}\}$ are disjoint.

Now, for each $i=1,\ldots,C_0$, let $\sphrs_i = \{S_i\colon S\in\sphrs\ \textrm{and}\ C(S)\leq i \}.$  Then $\sphrs_i$ is a $C_0$--good collection of pseudoflats, and $I_i=I\cap \mathcal{I}(\pts,\sphrs_i)$ is a good collection of incidences. We can then consider each collection $(\pts,\sphrs_i,I_i)$ in turn.

Henceforth, we shall assume that all surfaces in $\sphrs$ are irreducible. We will prove Theorem \ref{surfacesInR4Thm} by induction on $m+n$. In contrast to the proof of Solymosi and Tao in \cite{Solymosi}, the use of induction will not introduce an $\epsilon$ loss in the exponent. The induction is merely used to streamline the argument by controlling a few minor terms in one of the bounds in Section \ref{secondHamSection}. These terms can also be controlled through a lengthier argument that does not involve induction. An analogue of this lengthier argument appears around Equation (2.9) in \cite{Zahl}.

The base case where $m+n$ is small is obvious, provided we choose the constant $C_1$ from Theorem \ref{surfacesInR4Thm} to be larger than $mn$.

We will frequently make use of the following classical theorem of K\H{o}vari, S\'os, and Tur\'an from \cite{Turan}:
\begin{thm}\label{turanBoundLemma}
Let $s,t$ be fixed positive integers, and let $G$ be a bipartite graph with one vertex set of size $a$ and one vertex set of size $b$. Suppose that $G$ contains no induced subgraph isomorphic to $K_{s,t}$. Then $G$ has at most $O(ba^{1-1/s} + a)$ edges. Symmetrically, $G$ has at most $O(ab^{1-1/t} + b)$ edges. Here the implicit constants depend only on $s$ and $t$.
\end{thm}

From this theorem, we have that
\begin{align}
|\mathcal{I}(\pts,\sphrs)|&\lesssim mn^{1-1/k}+n,\label{R4TuranBound1}\\
|\mathcal{I}(\pts,\sphrs)|&\lesssim m^{1/2} n+m.\label{R4TuranBound2}
\end{align}

In particular, we can assume
\begin{equation}\label{r4mnNotTooDifferent}
\begin{split}
n &< c_1m^k,\\
m &< c_1n^2,
\end{split}
\end{equation}
where $c_1$ is a small constant that we are free to determine later; we can make $c_1$ smaller by making the constant $C_1$ from Theorem \ref{surfacesInR4Thm} larger. If \eqref{r4mnNotTooDifferent} failed, then Theorem \ref{surfacesInR4Thm} would follow immediately from \eqref{R4TuranBound1} or \eqref{R4TuranBound2}. The bounds \eqref{r4mnNotTooDifferent} imply that
\begin{equation}\label{mainTermDominates}
\begin{split}
n \leq c_2m^{\frac{k}{2k-1}}n^{\frac{2k-2}{2k-1}},\\
m \leq c_2m^{\frac{k}{2k-1}}n^{\frac{2k-2}{2k-1}},\\
\end{split}
\end{equation}
where $c_2$ can be made arbitrarily small by making the constant $c_1$ from \eqref{r4mnNotTooDifferent} sufficiently small. These inequalities will be useful for closing the induction.

\subsection{First polynomial partition}
Let
\begin{equation}\label{defnOfD}
D=m^{\frac{k}{4k-2}}n^{-\frac{1}{4k-2}}.
\end{equation}
By \eqref{r4mnNotTooDifferent}, $D$ satisfies the inequalities
\begin{equation}\label{CleqDleqn12}
C_2<D<c_3 m^{1/4},
\end{equation}
where we can make the constant $C_2$ arbitrarily large and $c_3$ arbitrarily small by making the constant $c_1$ in \eqref{r4mnNotTooDifferent} smaller.

Let $P$ be a polynomial of degree at most $D$ such that $\BZ_{\RR}(P)$ cuts $\RR^4$ into $O(D^4)$ cells $\{\Omega_i\}$, each containing $O(m/D^4)$ points, as given by Corollary \ref{hamSandwichThmStronger}. We can assume that $P$ is square-free and its irreducible components generate real ideals. Let $n_i$ be the number of surfaces in $\sphrs$ that meet the $i$--th cell.

\begin{lem}\label{ThomTypeBoundLem1}
\begin{equation}
\sum n_i \lesssim D^2 n,
\end{equation}
where the sum is taken over all cells in the decomposition.
\end{lem}
\begin{proof}
We will show that each surface in $\sphrs$ enters $O(D^2)$ cells. Let $S\in\sphrs,$ let $f_S$ be a polynomial so that $S=\BZ_{\RR}(f_S)$, and let $P$ be the partitioning polynomial described above. We can assume that $\dim_{\RR}(S\cap \BZ_{\RR}(P^2))\leq 1$, since otherwise $S$ enters no cells. The polynomials $f_S$ and $P^2$ satisfy the requirements of Corollary \ref{surfacesEntering4DCellCor}, so the number of connected components of $S\cap\{P^2>0\}$ is $O(D^2)$. Thus $S$ enters $O(D^2)$ connected components of $\RR^4\backslash\BZ_{\RR}(P)$, i.e.~$S$ enters $O(D^2)$ cells.
\end{proof}
Applying the Theorem \ref{turanBoundLemma} inside each cell, we obtain
\begin{equation}\label{incidencesInsideFirstPartition}
\begin{split}
|I\cap\mathcal{I}(\pts\backslash \BZ_{\RR}(P),\sphrs)|
&\lesssim \sum_i|\pts\cap\Omega_i|n_i^{1-1/k}+\sum_i n_i\\
&\lesssim \sum_i \frac{m}{D^4} n_i^{1-1/k}+D^2n\\
&\lesssim \frac{m}{D^4}\Big(\sum_i 1\Big)^{1/k}\Big(\sum_i n_i\Big)^{1-1/k}+D^2n\\
&\lesssim \frac{m}{D^4} D^{4/k}\Big(D^2n\Big)^{1-1/k}+D^2n\\
&\lesssim m^{\frac{k}{2k-1}}n^{\frac{2k-2}{2k-1}}.
\end{split}
\end{equation}
Here we used H\"older's inequality on the third line, Lemma \ref{ThomTypeBoundLem1} on the fourth line (plus the fact that there are $O(D^4)$ cells), and the definition of $D$ from \eqref{defnOfD} on the final line.

Recall from Definition \ref{defnOflesssim} that the implicit constants above depend only on $C_0$ and $k$ from the statement of Theorem \ref{surfacesInR4Thm}. Thus, if we select $C_1$ in the statement of Theorem \ref{surfacesInR4Thm} sufficiently large (depending on $C_0$ and $k$), we have
\begin{equation}\label{boundInsideFirstLevelCells}
|I\cap\mathcal{I}(\pts\backslash \BZ_{\RR}(P),\sphrs)| \leq \frac{C_1}{100}(m^{\frac{k}{2k-1}}n^{\frac{2k-2}{2k-1}}+m+n).
\end{equation}

\subsection{Boundary incidences of the first partition}
Write
\begin{equation}\label{decompOfSphrs}
\sphrs=\sphrs_1\sqcup\sphrs_2,
\end{equation}
where $\sphrs_1$ is the set of surfaces that are contained in $\BZ_{\RR}(P)$, and $\sphrs_2$ is the set of surfaces that properly intersect $\BZ_{\RR}(P)$ (Since each surface is irreducible, the latter type of intersection must have dimension at most 1).
\begin{lem}\label{controlOfSmoothPtsLem}
\begin{equation}\label{ptsCapZSmoothSphrs1}
|I\cap\mathcal{I}(\pts\cap \BZ_{\RR}(P)_{\smooth},\sphrs_1)|\leq  m.
\end{equation}
\end{lem}
\begin{proof}
Let $p\in\pts\cap \BZ_{\RR}(P)_{\smooth},$ and let $H = T_p(\BZ_{\RR}(P)).$ Suppose there exist surfaces $S_1,S_2\in\sphrs_1$ with $(p,S_1),\ (p,S_2)\in I$. By Property \ref{smoothSurfacesSmooth} from Definition \ref{goodSurfacesDefn}, $p$ is a smooth point of $S$ and of $S^{\prime}$. Since $S\subset \BZ_{\RR}(P)$, we have $T_p(S)\subset T_p(\BZ_{\RR}(P))=\Pi$. Similarly, $T_p(S^\prime)\subset \Pi$. On the other hand, from the definition of a good collection of incidences (Definition \ref{goodCollectionIncidencesDefn}), we have that $T_p(S)\cap T_p(S^\prime)=p.$ Thus we have two affine 2--planes, $T_p(S)$ and $T_p(S^\prime)$ which meet only at the point $p$, but both are contained in the affine 3--plane $\Pi$. This cannot occur. Thus for each point $p\in\pts\cap \BZ_{\RR}(P)_{\smooth},$ there exists at most one surface $S\in\sphrs_1$ with $(p,S)\in I(\pts\cap \BZ_{\RR}(P)_{\smooth},\sphrs_1)$.
\end{proof}
It remains to consider incidences between surfaces and points lying on $\BZ_{\RR}(P)_{\sing}$. By \eqref{CleqDleqn12}, we can assume $\deg P\geq 2$. Let $v$ be a generic (with respect to $P$)  vector and let $R=v\cdot\nabla P$. By Remark \ref{gradOfPolyRealIdeal}, $R$ is not the zero polynomial, $\BZ_{\RR}(P)_{\sing}\subset \BZ_{\RR}(P)\cap \BZ_{\RR}(R)$, and $\dim_{\CC}(\BZ_{\CC}(P)\cap \BZ_{\CC}(R))=2$.

Let $\sphrs_1^\prime \subset\sphrs_1$ be those surfaces contained in $\BZ_{\RR}(P)_{\sing}.$ If $S\in\sphrs_1^\prime$ and $S^*$ is the complexification of $S$ (the smallest complex variety in $\CC^4$ that contains $S$), then $S^*\subset \BZ_{\CC}(P)\cap \BZ_{\CC}(R).$ Since $\dim_{\CC}(\BZ_{\CC}(P)\cap \BZ_{\CC}(R))=2,$ we must have that $S^*$ is a union of irreducible components of $\BZ_{\CC}(P)\cap \BZ_{\CC}(R)$. Furthermore, if $S_1,\ldots,S_\ell\in \sphrs_1^\prime$, then $S_1^*\cup\ldots\cup S_\ell^*$ must contain at least $\ell$ irreducible components. But by B\'ezout's theorem (discussed further in Section \ref{degreeBezSec} below, or \cite[chapter 18]{harris}), $\BZ_{\CC}(P)\cap \BZ_{\CC}(R)$ can contain at most $(\deg P)(\deg R)\lesssim D^2$ irreducible components. We conclude that
\begin{equation}
|\sphrs_1^\prime|\lesssim D^2.
\end{equation}

Applying Theorem \ref{turanBoundLemma} and \eqref{CleqDleqn12}, we have
\begin{equation*}
\begin{split}
|I\cap\mathcal{I}(\pts\cap \BZ_{\RR}(P)_{\sing},\sphrs_1^\prime)|& \lesssim D^2 m^{1/2} + m\\
&\lesssim m.
\end{split}
\end{equation*}
Thus if we choose the constant $C_1$ sufficiently large depending on $C_0$ and $k$ from the statement of Theorem \ref{surfacesInR4Thm}, we have
\begin{equation}\label{ptsCapZSingSphrs1Prime}
 |I\cap\mathcal{I}(\pts\cap \BZ_{\RR}(P)_{\sing},\sphrs_1^\prime)|\leq\frac{C_1}{100}\Big(m^{\frac{k}{2k-1}}n^{\frac{2k-2}{2k-1}}+m+n\Big).
\end{equation}

Let $\sphrs_2^\prime\subset\sphrs_1$ be those surfaces (contained in $\BZ_{\RR}(P)$) that are not contained in $\BZ_{\RR}(R)$. We must now bound $|I\cap\mathcal{I}(\pts\cap \BZ_{\RR}(P),\sphrs_2)|$ and $|I\cap\mathcal{I}(\pts\cap \BZ_{\RR}(R),\sphrs_2^\prime)|$. By Lemma \ref{makeAPolyProdRealIdeals}, we can assume that the irreducible components of $R$ generate real ideals. But note that $\BZ_{\RR}(P)$ and $\BZ_{\RR}(R)$ are both the zero-set of polynomials of degree $O(D)$, and thus the two collections of incidences can be dealt with in the same fashion. In the arguments below, we will prove that
\begin{equation}\label{IZS2}
 |I\cap\mathcal{I}(\pts\cap \BZ_{\RR}(P),\sphrs_2)|\leq\frac{C_1}{10}\Big(m^{\frac{k}{2k-1}}n^{\frac{2k-2}{2k-1}}+m+n\Big).
\end{equation}
An identical argument shows that
\begin{equation}\label{IZRS2P}
 |I\cap\mathcal{I}(\pts\cap \BZ_{\RR}(R),\sphrs_2^\prime)|\leq\frac{C_1}{10}\Big(m^{\frac{k}{2k-1}}n^{\frac{2k-2}{2k-1}}+m+n\Big).
\end{equation}
Once we have established these inequalities, we can combine the bounds \eqref{boundInsideFirstLevelCells}, \eqref{ptsCapZSmoothSphrs1}, \eqref{ptsCapZSingSphrs1Prime}, \eqref{IZS2}, and \eqref{IZRS2P} to close the induction.

\subsection{Second polynomial partitioning decomposition}\label{secondHamSection}
We shall now establish inequality \eqref{IZS2}. Factor $P$ into its irreducible components, $P=P_1\cdots P_\ell,$ and let $D_i=\deg(P_i)$. Let
\begin{equation*}
\pts_i =\big(\pts\cap\BZ_{\RR}(P_i)\big)\backslash\bigcup_{j<i}\pts_j,
\end{equation*}
so $\pts_1,\ldots,\pts_\ell$ are disjoint and $\bigcup\pts_i = \pts\cap\BZ_{\RR}(P)$.

Let
\begin{equation}\label{ptsCuttof}
\begin{split}
\mathcal A_0&=\{i\colon |\pts_i|^{k}\leq  c_4nD_i^{4k-2}\},\\
\mathcal A_1&=\{1,\ldots,\ell\}\backslash \mathcal A_0.
\end{split}
\end{equation}
The (small) constant $c_4$ will be chosen later.
\subsubsection{Incidences on varieties in $\mathcal A_0$}
We have
\begin{equation}
\begin{split}
\big|\bigcup_{j\in \mathcal{A}_0}\pts_j\big| &\leq c_4^{1/k} \sum_{j\in A_0}n^{1/k}D_j^{\frac{4k-2}{k}}\\
&\leq c_4^{1/k} n^{1/k}D^{\frac{4k-2}{k}}\\
&\leq c_4^{1/k}m,
\end{split}
\end{equation}
We will select $c_4$ so that $c_4 <\!\!< 1$. By the induction hypothesis (discussed in Section \ref{initialRedSec}), we conclude that
\begin{equation}\label{A0IncBd}
\Big|I\cap\mathcal{I}\Big( \bigcup_{j\in A_0}\pts_j, \sphrs\Big)\Big| \leq C_1 \Big(c_4^{\frac{1}{2k-1}}m^{\frac{k}{2k-1}}n^{\frac{2k-2}{2k-1}} + C_3c_4^{1/k}m + n\Big).
\end{equation}
Select the constant $c_1$ from \eqref{r4mnNotTooDifferent} sufficiently small so that
\begin{equation*}
n\leq \frac{1}{200}m^{\frac{k}{2k-1}}n^{\frac{2k-2}{2k-1}},
\end{equation*}
and select the constant $c_4$ from \eqref{ptsCuttof} sufficiently small. Then from \eqref{mainTermDominates} we obtain
\begin{equation}\label{contribFromInduction}
\Big|I\cap\mathcal{I}\Big( \bigcup_{j\in A_0}\pts_j, \sphrs\Big)\Big| \leq \frac{C_1}{100}\Big(m^{\frac{k}{2k-1}}n^{\frac{2k-2}{2k-1}}+m+n\Big).
\end{equation}

\subsubsection{Incidences on varieties in $\mathcal A_1$}
For each $i\in \mathcal A_1$, define
\begin{equation}\label{defnEj}
E_i = |\pts_i|^{\frac{k}{3k-2}}n^{-\frac{1}{3k-2}}D_i^{-\frac{k}{3k-2}}.
\end{equation}
Note that with this choice of $E_i$, we have
\begin{equation}\label{E_iNotTooSmall}
E_i \geq c_4^{\frac{1}{3k-2}} D_i,
\end{equation}
where $c_4$ is the constant from \eqref{ptsCuttof}. 

By \eqref{defnOfD}, \eqref{defnEj}, and H\"older's inequality,
\begin{equation}\label{sumDjEj}
n\sum_{j\in \mathcal A_1} D_jE_j \leq m^{\frac{k}{2k-1}}n^{\frac{2k-2}{2k-1}}.
\end{equation}
This fact will be used frequently.

Apply Theorem \ref{variantHamSandwichThm} to the surface $\BZ_{\RR}(P_i)$ and the point set $\pts_i$ with the parameter $E_i$ (here we make use of \eqref{E_iNotTooSmall}), and let $Q_i$ be the resulting polynomial. $Q_i$ has degree $O(E_i)$, where the implicit constant depends only on $c_4$, and the set $\BZ_{\RR}(P_i)\backslash\BZ_{\RR}(Q_i)$ is a union of $O(D_iE_i^3)$ cells; each cell contains $O\big( |\pts_i|/(D_iE_i^3)\big)$ points from $\pts_i$. Again, the implicit constant depends only on $c_4$ from \eqref{ptsCuttof}, which in turn ultimately depends only on $C_0$ and $k$ from the statement of Theorem \ref{surfacesInR4Thm}.

Let $n_{i,j}$ be the number of surfaces in $\sphrs_2$ that meet the $j$--th cell from $\BZ_{\RR}(P_i)\backslash\BZ_{\RR}(Q_i)$.

\begin{lem}\label{corOfThomTypeBoundLem2}
For each index $i$, we have
\begin{equation}
\sum_j n_{i,j}\lesssim nD_iE_i,
\end{equation}
where the sum is taken over all cells in $\BZ_{\RR}(P_i)\backslash\BZ_{\RR}(Q_i)$.
\end{lem}
\begin{proof}
We will prove that each surface in $\sphrs_2$ enters $O(D_iE_i)$ cells. Let $S\in\sphrs,$ let $f_S$ be a polynomial so that $S=\BZ_{\RR}(f_S)$. Since $S\in\sphrs_2$, we have that $\dim_{\RR}(S\cap \BZ_{\RR}(P_i))\leq 1$. Thus we can apply Corollary \ref{surfacesEntering3DCellCor} to conclude that the number of connected components of $S\cap \BZ_{\RR}(P_i)\cap\{Q_i^2>0\}$ is $O(D_iE_i)$. This implies that $S$ enters $O(D_iE_i)$ cells, where the implicit constants depend only on $C_0$ from the statement of Theorem \ref{surfacesInR4Thm}.
\end{proof}

We shall now bound the number of incidences that occur in the cells $\Omega_{i,j}$. Recall that at the moment, $i$ is fixed. We have

\begin{equation}
\begin{split}
|I\cap\mathcal{I}(\pts_i\backslash \BZ_{\RR}(Q_i), \sphrs_2)|&\lesssim  \sum_{j}|\pts_i\cap\Omega_{i,j}|n_{i,j}^{1-1/k}+\sum_j n_{i,j}\\
&\lesssim \Big(\sum_j \Big(\frac{|\pts_i|}{D_iE_i^3}\Big)^k\Big)^{1/k}\Big(\sum_j n_{i,j}\Big)^{1-1/k}+D_iE_in\\
&\lesssim \Big(D_iE_i^3 |\pts_i|^k D_i^{-k}E_i^{-3k}\Big)^{1/k}\Big(D_iE_i n\Big)^{1-1/k}+D_iE_in\\
&\lesssim\frac{|\pts_i|n^{1-1/k}}{E_i^{2-2/k}} + D_iE_in.
\end{split}
\end{equation}

Summing over all indices $i\in\mathcal A_1$ and using \eqref{defnOfD}, \eqref{defnEj}, and H\"older's inequality, we obtain
\begin{equation}\label{secondLevelInsideCells}
\begin{split}
\sum_{i\in \mathcal A_1}&|I\cap\mathcal{I}(\pts_i\backslash \BZ_{\RR}(Q_i), \sphrs_2)| \\
&\lesssim \sum_{i\in\mathcal A_1} \frac{|\pts_i|n^{1-1/k}}{E_i^{2-2/k}} + \sum_{i\in\mathcal A_1}D_iE_in\\
&\lesssim \sum_{i\in \mathcal A_1} n^{\frac{3k-3}{3k-2}}|\pts_i|^{\frac{k}{3k-2}}D_i^{\frac{2k-2}{3k-2}} + m^{\frac{k}{2k-1}}n^{\frac{2k-2}{2k-1}}\\
&\lesssim n^{\frac{3k-3}{3k-2}}m^{\frac{k}{3k-2}}D^{\frac{2k-2}{3k-2}} + m^{\frac{k}{2k-1}}n^{\frac{2k-2}{2k-1}}\\
&\lesssim m^{\frac{k}{2k-1}}n^{\frac{2k-2}{2k-1}}.
\end{split}
\end{equation}

Combining \eqref{contribFromInduction} and \eqref{secondLevelInsideCells}, and selecting $C_1$ sufficiently large depending on $C_0$ and $k$ from the statement of Theorem \ref{surfacesInR4Thm}, we obtain
\begin{equation}\label{intermediateIncidenceBd}
\begin{split}
 |I\cap \incidences(\pts\cap Z,\sphrs_2)| \leq \frac{C_1}{50} &\Big(m^{\frac{k}{2k-1}}n^{\frac{2k-2}{2k-1}}+m+n\Big) \\
 &+\sum_{i\in\mathcal A_1} \big|I\cap\mathcal{I}(\pts_i \cap \BZ_{\RR}(Q_i),\sphrs_2)\big|.
\end{split}
 \end{equation}

It remains to bound the second term in \eqref{intermediateIncidenceBd}.
%
%
%
%
%
%
%
%
%
\section{A foray into algebraic geometry}\label{forayAlgGeoSec}
\subsection{(Some more) real algebraic geometry}
\subsubsection{Semialgebraic sets}
A semialgebraic set is a finite union of sets of the form
\begin{equation*}
\bigcap_{j=1}^\ell \BZ_{\RR}(R_j) \cap \bigcap_{j=1}^{\ell^\prime} \{x\in\RR^d\colon R_j^\prime(x)>0\},
\end{equation*}
where $R_1,\ldots,R_{\ell}$ and $R^\prime_1,\ldots,R^\prime_{\ell^\prime}$ are real polynomials.

Later in our arguments we will need to consider (real) algebraic curves with finitely many points deleted. These objects are semialgebraic sets.

\subsubsection{Real and complex varieties}\label{realCplxVarietiesSec}
If $Z\subset\CC^d$ is a complex variety, let
\begin{equation*}
Z(\RR)=\{(x_1,\ldots,x_d)\in Z\colon \operatorname{Im}(x_i)=0\ \textrm{for each}\ i=1,\ldots,d\}.
\end{equation*}
Thus $Z(\RR)$ is the set of real points of $Z$. Generally, we will be interested in varieties $Z\subset\CC^d$ that can be defined by real polynomials. Conversely, if $Z\subset\RR^d$ is a real variety, let $Z^*$ be the smallest complex variety containing $Z$, i.e. $Z^*$ is the closure of $Z$ (after $Z$ has been embedded into $\CC^d$) in the Zariski topology on $\CC^d$. Observe that if $Z(\RR)$ is Zariski dense in $Z$, then $Z(\RR)^*=Z$; we will only use this observation in the special case where $Z$ is irreducible.

\subsubsection{The Zariski tangent space of a variety}\label{ZariskiTangentSpaceSec}
\begin{defn}
Let $Z\subset\CC^d$ be a variety. We define the \emph{Zariski tangent space} of $Z$ at the point $z$ to be
\begin{equation}
T_z(Z) = \{v\in \CC^d\colon \nabla f(v)=0\ \textrm{for all}\ f\in I(Z)\}.
\end{equation}
\end{defn}
If $I(Z)=(f_1,\ldots,f_\ell)$, then we can replace the condition ``$f(v)=0$ for all $f\in I(Z)$'' with the equivalent condition ``$f_1(v)=0,\ldots,f_\ell(z)=0$.''

If $\dim(T_z(Z))=\dim(Z)$, we say that $z$ is a smooth point of $Z$. Otherwise, it is a singular point. $z\in Z$ is a smooth point if and only if $Z$ is a $\dim(Z)$--dimensional complex manifold in a (Euclidean) neighborhood of $z$ see \cite[Chapter 1]{mumford} for further details.

\subsubsection{Points where real and complex dimension don't agree}
We will be interested in a variant of the following question. Let $Z\subset\CC^d$ be a variety that can be written as an intersection of $d-\dim_{\CC}(Z)$ hypersurfaces, each defined by a real polynomial, and let $z\in Z(\RR)$. Suppose that $\dim_{z,\RR}(Z(\RR))<\dim_{\CC} Z.$ Must $z$ be a singular point of $Z$? In this section, we will show that at least in some special cases, the answer is yes. The main tool will be a similar result about curves, which is proved in \cite[Section 6]{ElKahoui}.

\begin{lem}\label{isolatedPtOfCurveLem}
Let $\zeta \subset\CC^3$ be a space curve (a one-dimensional complex variety). Suppose that $\zeta=\BZ_{\CC}(P_1)\cap\BZ_{\CC}(P_2)$, where $P_1,P_2$ are real polynomials. Let $O\in O(3;\RR)$ be a generic (with respect to $P_1$ and $P_2$) rotation (see Section \ref{genericPtsSection} below for the definition of a generic rotation.) and let $\zeta^\prime = O(\zeta)$. Let $\pi\colon\CC^3\to\CC^2$ be the projection in the $x_3$--direction. If $z\in \zeta^\prime(\RR)$ is an isolated point, then $\pi(z)$ is an isolated point of $(\pi(\zeta^\prime))(\RR)$.
\end{lem}
\begin{proof}
The main tool we will use is Lemma 6.2 from \cite{ElKahoui}. Let $\zeta\subset\CC^3$ be a space curve. We say that $\zeta$ is in \emph{generic position} with respect to the projection to the $(x_1,x_2)$--plane if it satisfies the conditions from \cite[Definition 4.1]{ElKahoui}. Rather than state the definition of generic position here (it is quite technical), we will only state the properties we need.

First, by \cite[Section 5.4]{ElKahoui}, any curve $\gamma\subset\CC^3$ may be put in generic position after applying a generic orthogonal transformation\footnote{\cite[Section 5.4]{ElKahoui} actually considers a generic affine transformation rather than a generic orthogonal transformation, but the same argument applies.} $O\in O(3;\RR)$. Informally, a curve is in generic position if no coincidences happen when the curve $\zeta$ is projected onto the $x_1,x_2,$ or $x_3$ axes (for example, it would be bad if two distinct singular points of $\zeta$ projected to the same point).

In \cite{ElKahoui}, El Kahoui also defines what he calls an \emph{event point} for the (real) curve $\zeta(\RR)$. This includes objects such a critical points of $\zeta$, etc. Again, we do not need a precise definition; the only property we will use is that the set of event points is finite, and thus they will not be relevant to our argument. 

Let $\zeta\subset\CC^3$ be a space curve in general position that is defined by real polynomials, and let $\pi\colon\CC^3\to\CC^2$ be the projection onto the $(x_1,x_2)$--plane. Define $\alpha_\zeta = (\pi(\zeta))(\RR)$ (while in general the projection of a space curve to the plane need not be a plane curve, after applying a generic orthogonal transformation we can ensure that this is the case).

Lemma 6.2(i) from \cite{ElKahoui} relates the properties of $\zeta(\RR)$ and $\alpha_\zeta.$ In the terminology used here, \cite[Lemma 6.2(i)]{ElKahoui} says the following: if $I\subset \RR$ is an interval that does not contain the $x$--coordinate of any event point, and if $\beta\subset\alpha_\zeta$ is a simple open smooth real curve (in this case not an algebraic curve, but a smooth subset of an algebraic curve that is homeomorphic to $(0,1)$ ), then there is a simple open smooth real curve $\beta^\prime\subset\zeta(\RR)$ whose projection to the $(x_1,x_2)$--plane is $\beta$.

We can now prove Lemma \ref{isolatedPtOfCurveLem}. Let $z\in \zeta^\prime(\RR)$ be an isolated point. Suppose that $x=\pi(z)$ is not isolated. Since $O$ was a generic rotation, we can assume that $\pi^{-1}(x)=\{z\}$. For any $\epsilon>0$, we can find a simple open smooth real curve $\beta\subset \pi(\zeta^\prime)(\RR)$ such that $\operatorname{dist}(x,\beta)<\epsilon$ and the projection of $\beta$ to the $x_1$--axis does not contain any event points. We can now apply lemma 6.2(i) from \cite{ElKahoui} to conclude that there is a curve $\beta^\prime\subset\zeta^\prime(\RR)$ whose projection to the $(x_1,x_2)$--plane is $\beta$.

This means that for every $\epsilon>0$ there is a curve $\beta^\prime\subset \zeta(\RR)$ whose projection is $\epsilon$--close to $\pi(z)$. Since $\zeta(\RR)$ is closed (in the Euclidean topology) and $z$ is an isolated point of $\zeta(\RR)$, we conclude that the pre-image $\pi^{-1}(x)$ contains at least two points. But we assumed that this was not the case. This contradiction establishes the lemma.
\end{proof}

\begin{cor}\label{realComplexDimDontAgreeCor}
Let $P_1,P_2\in\RR[x_1,\ldots,x_4]$, and let $Z=\BZ_{\CC}(P_1)\cap\BZ_{\CC}(P_2)$. Suppose that $\dim_{\CC}(Z)=2$. If $z\in Z(\RR)$ satisfies $\dim_{z,\RR}(Z(\RR))\leq 1$, then $z$ is a singular point of $Z$.
\end{cor}
\begin{proof}
Suppose $z$ is a smooth point of $Z$; we will obtain a contradiction. Let $H\subset\CC^4$ be a generic real 3--plane passing through $z$, i.e.~$H$ is the zero set of a linear polynomial in $\RR[x_1,\ldots,x_4]$. Then $H\cap Z$ is a complex one-dimensional variety (i.e.~a curve), $z$ is a smooth point of $H\cap Z$, and if we identify $H$ with $\CC^3$, we can write $H\cap Z = \BZ_{\CC}(P_1^\prime)\cap\BZ_{\CC}(P_2^\prime),$ where $P_1^\prime,P_2^\prime\in\RR[x_1,x_2,x_3]$. Furthermore, since $\dim_{z,\RR}(Z(\RR))\leq 1$, we have $\dim_{z,\RR}((H\cap Z)(\RR))=0,$ i.e.~$z$ is an isolated point of $(H\cap Z)(\RR)$. Thus we can apply Lemma \ref{isolatedPtOfCurveLem} to conclude that $z$ is a singular point of $H\cap Z$. This contradicts the assumption that $z$ was a smooth point of $H\cap Z$. We conclude that $z$ is a singular point of $Z$.
\end{proof}
\subsection{Generic points}\label{genericPtsSection}
Often in our arguments we will consider properties that hold at most places on an algebraic variety. In this section we will make the notion of ``most places'' precise. Specifically, we will introduce the notion of a generic point. We will begin with a motivating example. 
\begin{example}
Let $Z\subset\CC^4$ be an irreducible hypersurface and let $\pts,$ $\sphrs$ be the set of points and pseudoflats from Theorem \ref{surfacesInR4Thm}. Then a generic point of $Z$ does not lie in $\pts$, and does not intersect any any pseudoflat from $\sphrs$. 
\end{example}

\begin{defn}
Let $Z$ be an irreducible complex variety and let $\mathcal M$ be a finite collection of polynomials, none of which vanish on $Z$. We say that a point $z\in Z$ is \emph{generic} with respect to $\mathcal M$ if none of the polynomials in $\mathcal{M}$ vanish on $z$. In particular, for $Z$ and $\mathcal{M}$ fixed, the set of generic points is Zariski dense in $Z$. In practice, the collection $\mathcal{M}$ of polynomials will be aparent from context, so we will abuse notation and make statements such as ``a generic point of $Z$ has the following properties.'' Here the set of polynomials $\mathcal M$ should be inferred from the properties we have specified.

In general, the set of polynomials $\mathcal{M}$ will depend on the variety $Z$, the set of points and pseudoflats from Theorem \ref{surfacesInR4Thm} as well as any intermediate objects that have already been constructed, and whatever property is currently under consideration.

If $Z(\RR)$ is Zariski dense in $Z$, then we define a generic real point of $Z(\RR)$ to be a point $z\in Z(\RR)$ for which no polynomial in $\mathcal{M}$ vanishes. In particular, if $Z(\RR)$ is dense in $Z$, then the set of generic real points is non-empty.
\end{defn}

The set $\CC^4$ will be of particular interest, and we will consider it as both a vector space and a complex variety. In our arguments below, we will refer to generic vectors in $\CC^4$ or $\RR^4$. This means that the vector is generic with respect to all of the objects defined previously---this includes the points $\pts$, surfaces $\sphrs$, the partitioning polynomials $P$ and $\{Q_i\}$, and any previously defined vectors, etc.

We will also be interested in several other generic objects:
\begin{itemize}
 \item Generic $k$--planes. These are generic elements of the Grassmannian $\grass(k,d;\CC)$ or $\grass(k,d;\RR)$. They will be discussed further in Section \ref{GaussMapSection} below.
 \item Generic (real) rotations. These are generic elements of the orthogonal group $O(d;\RR)$ (this group has the structure of a real variety).
 \item Generic projections. These are projections of the form $O^{-1}\circ \pi\circ O,$ where $\pi\colon\CC^d\to\CC^{d^\prime}$ is the projection to the first $d^\prime$ coordinates, and $O$ is a generic rotation.
\end{itemize}
\subsection{Resultants and Projections}
Given two polynomials $f,g\in\CC[x_1,\ldots,$ $x_d]$, we define the resultant $\res(f,g)\in\CC[x_1,\ldots,x_{d-1}]$ to be the resultant of $f$ and $g$ in the $x_d$--variable, i.e.~we consider $f$ and $g$ to be polynomials in $x_d$ with coefficients in the ring $\CC[x_1,\ldots,x_{d-1}]$, and we take the (classical) resultant of these two polynomials. If $f$ and $g$ have real coefficients, then $\res(f,g)$ also has real coefficients.

If $f\in \CC[x_1,\ldots,x_{d}]$, we say that $f$ is $x_{d}$--\emph{monic} if the coefficient of $x_d^{\deg f}$ is non-zero. If $f$ is $x_d$--monic, $f$ and $g$ intersect properly (i.e. if $\dim_{\CC}(\BZ_{\CC}(f)\cap\BZ_{\CC}(g))=d-2$), and if $\pi_d\colon\CC^{d}\to\CC^{d-1}$ is the projection to the first $(d-1)$-coordinates, then
\begin{equation*}
\pi_d(\BZ_{\CC}(f)\cap\BZ_{\CC}(g))\subset \BZ_{\CC}(\res(f,g)).
\end{equation*}
See for example Section 2C from \cite{mumford}. In particular, if $f$ and $g$ have real coefficients, $f$ is $x_d$--monic, and if $\dim_{\CC}(\BZ_{\CC}(f)\cap\BZ_{\CC}(g))=d-2$, then
\begin{equation}\label{projectionDimensionResultant}
\pi_d(\BZ_{\RR}(f)\cap\BZ_{\RR}(g))\subset \BZ_{\RR}(\res(f,g)),
\end{equation}
and $\BZ_{\RR}(\res(f,g))$ is a variety of dimension at most $d-2$.

Note, however, that if we only require that $\dim_{\RR}(\BZ_{\RR}(f)\cap\BZ_{\RR}(g))=d-2$, then $\BZ_{\RR}(\res(f,g))$ may be all of $\RR^{d-1}$. For example, if $f=g$ and $\dim_{\RR}(\BZ_{\RR}(f))=d-2$, then $\dim_{\RR}(\BZ_{\RR}(f)\cap\BZ_{\RR}(g))=d-2$, but $\res(f,g)$ is the zero polynomial.

While not every polynomial $f\in\CC[x_1,\ldots,x_d]$ is $x_d$--monic, we can usually fix this problem by pre-composing $f$ with a generic orthogonal transformation. More precisely, if $f\in\CC[x_1,\ldots,x_d]$ and if $O\in O(d;\CC)$ is a generic rotation, then $f\circ O$ is $x_d$--monic. The same statement holds if $f\in\RR[x_1,\ldots,x_d]$ and $O$ is a generic real rotation.

\subsubsection{Projections and degree} If $Z\subset\CC^d$ is an irreducible variety of dimension $\leq d-2$, and $\pi\colon \CC^{d}\to\CC^{d-1}$ is a generic projection, then $\deg(\overline{\pi(Z)})=\deg(Z)$, where $\overline{\phantom{1}\cdot\phantom{1}}$ denotes closure in the Zariski topology. This follows from the definition of degree given in Section \ref{degreeBezSec} below.
\subsection{Singular points of transverse intersections}\label{multVarietyPtSec}
Let $f,g\in\CC[x_1,x_2,x_3]$ be square-free polynomials. If $z\in \BZ_{\CC}(f)\cap\BZ_{\CC}(g)$, we define the \emph{intersection multiplicity} of $\BZ_{\CC}(f)\cap\BZ_{\CC}(g)$ at $z$ to be the intersection multiplicity of the plane curves $\big(\BZ_{\CC}(f)\cap H\big) \cap \big(\BZ_{\CC}(g)\cap H\big)$ in $H$, where $H$ is a generic plane passing through $z$. The intersection multiplicity of plane curves in $\CC^2$ is a classical subject and has many equivalent definitions. See Section 5.1 of \cite{mumford} for further discussion. We will need the following properties of intersection multiplicity:

\begin{itemize}
\item If $z$ is a smooth point of $\BZ_{\CC}(f)\cap\BZ_{\CC}(g)$ and if $\binom{\nabla f(z)}{ \nabla g(z)}$ has rank 2, then the intersection multiplicity of $\BZ_{\CC}(f)\cap\BZ_{\CC}(g)$ at $z$ is 1; this is because $\BZ_{\CC}(f),\BZ_{\CC}(g),$ and $H$ form a transverse complete intersection at $z$.
\item $f$ and $g$ are square-free, $z\in\BZ_{\CC}(f)\cap\BZ_{\CC}(g),$ and if $z$ is a singular point of $\BZ_{\CC}(f)$, then the intersection multiplicity of $\BZ_{\CC}(f)\cap\BZ_{\CC}(g)$ at $z$ is strictly greater than 1.
\end{itemize}

\begin{lem}\label{singPtCor}
Let $Y,Z$ be two-dimensional varieties in $\CC^3$ and let $\zeta\subset Y\cap Z$ be an irreducible component. Suppose that $Y$ is smooth, and suppose that $Y$ and $Z$ intersect transversely on $\zeta$ (i.e. $Z$ is smooth at a generic point of $\gamma$, and $Y$ and $Z$ intersect transversely at a generic point of $\zeta$). Then if $z\in \zeta$ is a singular point of $Z$, $z$ must also be a singular point of $Y\cap Z$.
\end{lem}

\begin{proof}
First, if $z$ lies on more than one component of $Y\cap Z$, then $z$ is a singular point of $Y\cap Z$, so we are done. Thus we may assume that $z$ only lies on the component $\zeta$. Let $Z=\BZ_{\CC}(f),\ Y=\BZ_{\CC}(g)$ with $f$ and $g$ square-free. Since $Y$ and $Z$ intersect transversely along $\zeta$, each generic point $x\in\zeta\subset Y\cap Z$ has multiplicity 1. However, since $z$ is a singular point of $Z$, $\nabla f(z)=0$, so  $\binom{\nabla f(z)}{\nabla g(z)}$ has rank $\leq 1$. Thus $z$ is a singular point of $\zeta$.
\end{proof}

\subsection{Degree and B\'ezout's theorem}\label{degreeBezSec}
\begin{defn}\label{defnOfDegree}
Let $Z\subset\CC^d$ be a pure-dimensional variety (i.e.~all of its irreducible components have the same dimension). We define the degree of $Z$ to be $|Z\cap H|$, where $H$ is a generic linear space of dimension $d-\dim_{\CC}(Z)$. This definition is independent of the choice of (generic) hyperplane; see \cite[chapter 18]{harris} for further details.  In particular, if $Z=\BZ_{\CC}(f)$, then $\deg Z\leq\deg f$. If $Z\subset\CC^d$ is a hypersurface (a $(d-1)$--dimensional variety), we can write $Z=\BZ_{\CC}(f)$ for some $f\in\CC[x_1,\ldots,x_d]$ with $\deg f=\deg(Z)$.
\end{defn}
We will make frequent use of B\'ezout's theorem, which gives us quantitative control on the complexity of the intersection of two varieties. There are many variants of this theorem. We will state a version below that is sufficient for our needs. 

\begin{prop}[B\'ezout's theorem for properly intersecting varieties]\label{Bezout}
Let $Y,Z\subset\CC^d$ be pure-dimensional varieties, and suppose that
\begin{equation}\label{propIntersecting}
\dim_{\CC}(Y\cap Z)=\dim Y + \dim Z - d.
\end{equation}
Then
\begin{equation}\label{fultonBezoutBd}
\deg(Y\cap Z)\leq \deg(Y)\deg(Z).
\end{equation}
In particular, if \eqref{propIntersecting} holds and if $\dim Y + \dim Z=d$, then $Y\cap Z$ is a finite set, and it has cardinality at most $\deg(Y)\deg(Z)$.
\end{prop}
\begin{rem}
The proposition above is Example 12.3.1 from \cite{Fulton}, which is itself a special case of Theorem 12.3. In Example 12.3.1, the LHS of \eqref{fultonBezoutBd} is replaced by the sum of the degrees of the irreducible components of $Y\cap Z$. However, our definition of degree (Definition \ref{defnOfDegree}) allows for a variety to have several irreducible components, so \eqref{fultonBezoutBd} coincidences with the statement in \cite{Fulton}.
\end{rem}

We will also need a version of B\'ezout's theorem when the varieties do not intersect properly. For simplicity, we will only state a special case
\begin{prop}[B\'ezout's theorem for non-properly intersecting varieties; special case]\label{BezoutNonProper}
Let $Y,Z\subset\CC^d$ be pure-dimensional varieties. Then the number of isolated points of $Y\cap Z$ is at most $\deg(Y)\deg(Z)$.
\end{prop}
This is another special case of Example 12.3.1 from \cite{Fulton}. In \cite{Fulton}, Fulton defines a \emph{distinguished component} of the intersection $Y\cap Z$, and then proceeds to bound the number of distinguished components. Isolated points of $Y\cap Z$ are distinguished components of the intersection, so the bound applies here.

Finally, we will need a version of B\'ezout's theorem with multiplicities for plane curves. This is also a corollary of Example 12.3.1 from \cite{Fulton}. We will first introduce the notion of multiplicity of a plane curve at a point and multiplicity of an intersection of plane curves.
\begin{defn}\label{BezoutMultPlane}
Let $\zeta\subset\CC^2$ be an algebraic curve and let $z\in \CC^2$. We define the \emph{multiplicity of $\zeta$ at $z$}, $\mult_z(\zeta),$ to be the order of vanishing of $f$ at $z$, where $f$ is the unique (up to scalar multiples) square-free polynomial such that $\zeta=\BZ_{\CC}(f)$.
\end{defn}

\begin{defn}
Let $\zeta,\zeta^\prime\subset\CC^2$ be algebraic plane curves that have no common components, and let $z\in\zeta\cap\zeta^\prime.$ Then there is a number $m_z=\mult_z(\zeta\cap\zeta^\prime)$ with the following property. For all sufficiently small open Euclidean neighborhoods $U$ of $z$, $z$ is the unique point in $U\cap\zeta\cap\zeta^\prime$. For each such neighborhood $U$, there is a number $\epsilon>0$ so that if $v$ is a generic vector in $\CC^2$ with $|v|\leq \epsilon$, then $U\cap(\zeta+v)\cap\zeta^\prime$ is a union of $m_z$ points. In short, if we shift $\zeta$ by a small generic vector $v$, then the point $z\in\zeta\cap\zeta^\prime$ splits into $\mult_z(\zeta\cap\zeta^\prime)$ distinct points.
\end{defn}

\begin{prop}[B\'ezout's theorem with multiplicity for plane curves]\label{BezoutMultPlane}
Let $\zeta,\zeta^\prime$ be plane curves with no common components. Then
\begin{equation*}
\sum_{z\in\zeta\cap\zeta^\prime}\mult_z(\zeta\cap\zeta^\prime)\leq(\deg\zeta)(\deg\zeta^\prime).
\end{equation*}
\end{prop}
\subsection{Controlling the singular locus of a surface}
\begin{lem}\label{trappingSingLocusVariety}
Let $P,Q\in\CC[x_1,\ldots,x_4]$ be polynomials, and suppose that $\BZ_{\CC}(P)\cap\BZ_{\CC}(Q)$ is a complete intersection. Then there is a curve $\gamma$ of degree  $O\big((\deg P)^2( \deg Q)^2\big)$ so that $(\BZ_{\CC}(P)\cap\BZ_{\CC}(Q))_{\sing}\subset\gamma$.
\end{lem}
\begin{proof}
This is a special case of the general fact that if $Z\subset\CC^d$ is an irreducible variety, then there is a variety of degree $O((\deg Z)^2)$ and dimension $< \dim Z$ that contains $Z_{\sing}$. However, there do not appear to be any easy references to this fact in the literature, so we will briefly sketch the proof of Lemma \ref{trappingSingLocusVariety} here.

After a generic change of coordinates, we can assume that $P$ and $Q$ are $x_4$--monic, and thus $\pi(\BZ_{\CC}(P)\cap\BZ_{\CC}(Q))\subset\BZ_{\CC}(\res(P,Q))$. Recall from Section \ref{ZariskiTangentSpaceSec} that the singular points of $\BZ_{\CC}(P)\cap\BZ_{\CC}(Q)$ are precisely those points at which $\BZ_{\CC}(P)\cap\BZ_{\CC}(Q)$ fails to be a complex manifold. Thus $\pi((\BZ_{\CC}(P)\cap\BZ_{\CC}(Q))_{\sing})\subset(\BZ_{\CC}(\res(P,Q)))_{\sing}$. But $\BZ_{\CC}(\res(P,Q))$ is a surface in $\CC^3$ of degree at most $(\deg P)(\deg Q)$, and thus we can write $\BZ_{\CC}(P)\cap\BZ_{\CC}(Q)=\BZ_{\CC}(f)$, where $f\in\CC[x_1,x_2,x_3]$ is a square-free polynomial of degree at most $(\deg P)(\deg Q)$.

We can now find a polynomial $g\in\CC[x_1,x_2,x_3]$ so that $\BZ_{\CC}(f)\cap\BZ_{\CC}(g)$ is a complete intersection, and $\BZ_{\CC}(f)_{\sing}\subset\BZ_{\CC}(f)\cap\BZ_{\CC}(g)$. Briefly, we do this as follows. Let $v\in\CC^3$ be a generic vector, and let $g= v\cdot\nabla f$. Then $(\BZ_{\CC}(f))_{\sing}\subset\BZ_{\CC}(f)\cap\BZ_{\CC}(g)$. Furthermore, $\BZ_{\CC}(f)\cap\BZ_{\CC}(g)$ is a complete intersection; if this were not the case, then $g$ must vanish identically on some irreducible component of $f$. But since $v$ was chosen generically, this implies that $\nabla f$ vanishes identically on some irreducible component of $f$, and this contradicts the assumption that $f$ was square-free.

Let $g^\prime(x_1,x_2,x_3,x_4)=g(x_1,x_2,x_3)$ and let $\gamma= \BZ_{\CC}(g^\prime)\cap \BZ_{\CC}(P)\cap\BZ_{\CC}(Q)$. Then $(\BZ_{\CC}(P)\subset\BZ_{\CC}(Q))_{\sing}\subset\gamma$, and $\gamma$ is a curve of degree $O\big((\deg P)^2( \deg Q)^2\big)$.
\end{proof}
\subsection{Branches of algebraic curves}\label{branchesOfCurvesSection}
Frequently, we will need to bound the number of point-surface incidences $I\subset\mathcal{I}(\pts,\sphrs)$ when the points lie on a one-dimensional algebraic curve $\zeta$, and the surfaces meet that curve in a one-dimensional intersection (which need not be all of $\zeta$, since generally $\zeta$ will not be irreducible). The idea is that if a point $p\in\pts$ lies in $\zeta_{\smooth}$, then there can be at most one surface $S\in\sphrs$ that is incident to $p$ and for which $S\cap\zeta$ contains an irreducible component of $\zeta$ containing $p.$ However, if $p\in \zeta_{\sing}$, then potentially many surfaces $S\in\sphrs$ can have this property. We need to bound how many surfaces there can be. This is controlled by the number of branches of $\zeta$ at the point $p$. We recall \cite[Lemma 3.3]{milnor}:
\begin{lem}
Let $z$ be a non-isolated point of a real or complex one-dimensional variety $V$. Then a suitably chosen (Euclidean) neighborhood of $z$ in $V$ is the union of finitely many \emph{branches} which intersect only at $z$. Each branch is homeomorphic to a (Euclidean) open interval of real numbers (if $V$ is a real variety) or a (Euclidean) open disk of complex numbers (if $V$ is  a complex variety).
\end{lem}

\begin{defn}\label{defnOfGz}
For $z\in\zeta$, let $G_z(\zeta)$ be the number of branches of $\zeta$ through $z$. For example, if $z$ is a smooth point of $\zeta$, then $G_z(\zeta)=1$.
\end{defn}

\begin{lem}
Let $\zeta\subset\CC^d$ be an algebraic curve. Suppose $z\in \zeta(\RR)$ is a non-isolated point. Then the number of real branches of $\zeta(\RR)$ through $z$ is at most the number of complex branches of $\zeta$ through $z$.
\end{lem}
See e.g.~\cite[p29]{milnor}.

\begin{lem}\label{branchesOfACurveLem}
Let $\zeta\subset\CC^d$ be an algebraic curve. Then
\begin{equation}
 \sum_{z\in\zeta_{\sing}}G_z(\zeta)\leq(\deg\zeta)^2.
\end{equation}
\end{lem}
\begin{proof}
The main observation is that if $\pi\colon\CC^d\to\CC^2$ is a generic projection, then $G_{z}(\zeta)\leq G_{\pi(z)}(\overline{\pi(\zeta)})$. Thus it suffices to prove the result for plane curves. However, if $\zeta$ is a plane curve then $G_z(\zeta)\leq\mult_z(\zeta)$, where $\mult_z(\zeta)$ is given by Definition \ref{BezoutMultPlane}. We have that
\begin{equation*}
\sum_{z\in\zeta_{\sing}}\mult_z(\zeta)\leq (\deg\zeta)^2.
\end{equation*}
See i.e.~\cite[(7) on page 54]{SR} for a discussion of this formula. Equation (7) on page 54 of \cite{SR} defines the genus of an irreducible complex plane curve $\zeta$ to be
\begin{equation*}
\frac{1}{2}(\deg\zeta)(\deg\zeta-1)-\frac{1}{2}\sum \mult_z(\zeta)(\mult_z(\zeta)-1),
\end{equation*}
where the sum is taken over all multiple points of the curve. Since the genus is non-negative, this implies
\begin{equation*}
\sum \mult_z(\zeta)(\mult_z(\zeta)-1)\leq(\deg\zeta)(\deg\zeta-1),
\end{equation*}
so in particular
\begin{equation*}
\sum \mult_z(\zeta)\leq(\deg\zeta)^2.
\end{equation*}
It remains to extend this result to reducible curves. But this follows from B\'ezout's theorem for plane curves (Proposition \ref{BezoutMultPlane}). Factor $\zeta = \zeta_1\cup\ldots\cup\zeta_\ell$ into irreducible components. If $z\in\zeta$, then $\mult_z(\zeta)=\sum_{i=1}^\ell\mult_z(\zeta_i)$. We have
\begin{equation*}
\begin{split}
\sum_{z\in\zeta\colon\mult_z(\zeta)\geq 2}\mult_z(\zeta)&\leq \sum_{i< i^\prime}\sum_{z\in \zeta_i\cap\zeta_{i^\prime}}\mult_z(\zeta_i\cap\zeta_{i^\prime})+\sum_i\sum_{z\in\zeta_i\colon\mult_z(\zeta_i)\geq 2}\mult_z(\zeta_i)\\
&\leq \sum_{i< i^\prime}(\deg\zeta_i)(\deg\zeta_{i^\prime})+\sum_{i}(\deg\zeta_i)^2\\
&\leq(\deg\zeta)^2.\qedhere
\end{split}
\end{equation*}

\end{proof}

Later in our proof we will be given a collection of Euclidean connected components of real algebraic curves and a collection of bad points on these curves. We will need to remove these bad points to obtain a (possibly larger) collection of curves. The following observation bounds the number of additional connected components that are created in this process. In essence, it says that when you remove a point from an interval, you are left with two connected components. 

\begin{lem}[Real branches and connected components]\label{removePtIncreaseConnectedComponents}
Let $\zeta\subset \CC^d$ be an algebraic curve and let $\alpha\subset\zeta(\RR)$ be a semialgebraic set. Suppose that $z\in\alpha$ and $\dim_{\RR,z}(\alpha)=1$. Then
\begin{equation}
b_0(\alpha\backslash z) \leq b_0(\alpha) + 2G_z(\alpha)\leq b_0(\alpha) +2 G_z(\zeta),
\end{equation}
where $b_0(X)$ is the number of Euclidean connected components of the set $X$.
\end{lem}

\subsection{Incidences on algebraic curves}
The following lemma will be used frequently to bound the number of incidences occurring on various bad sets.
\begin{lem}[Incidences on a curve]\label{incidencesOnACurveLem}
Let $\zeta\subset\CC^4$ be an algebraic curve. Let $\pts\subset\RR^4$ be a collection of points, and suppose $\pts\subset \zeta(\RR)$. Let $\mathcal{S}$ be a $C_0$--good collection of pseudoflats (in $\RR^4$), and let $I\subset\mathcal{I}(\pts,\mathcal{S})$ be a good collection of incidences. Let
\begin{equation*}
\begin{split}
I^\prime=\{(p,S)\in I\colon &p\ \textrm{lies on a one-dimensional component}\\& \textrm{of}\ S^*\cap \zeta,\ \textrm{and}\ p\ \textrm{is a smooth point of this component}\}.
\end{split}
\end{equation*}
Then
\begin{equation}
|I^\prime|\leq |\pts| + (\deg\zeta)^2.
\end{equation}
\end{lem}
\begin{proof}
If $p\in\zeta$ is a smooth point, then $p$ can be incident to at most one surface in $\mathcal{S}$. Otherwise, there can be at most $G_\zeta(p)$ surfaces $S\in\mathcal{S}$ with $(p,S)\in I$. The result now follows from Lemma \ref{branchesOfACurveLem}.
\end{proof}
We are now ready to return to the task of bounding incidences on the surfaces $\BZ_{\RR}(P_i)\cap \BZ_{\RR}(Q_i)$.
\section{Proof of Theorem \ref{surfacesInR4Thm} step two: incidences on a surface in $\RR^4$}\label{proofOfCrossingNumberIneqLemmaSection}
Let
\begin{equation*}
I_0 = I\cap \bigcup_i \mathcal{I}(\pts_i \cap \BZ_{\RR}(Q_i),\sphrs_2).
\end{equation*}

The goal of the next sections is to bound $|I_0|$. The basic idea is that problems occur when surfaces $S\in\sphrs_2$ intersect $\BZ_{\RR}(P_i)\cap\BZ_{\RR}(Q_i)$ in one-dimensional curves, and many points lie on these curves. We will first deal with the incidences where this does not occur. Let
\begin{equation*}
\begin{split}
 I_0^\prime &= \{(p,S)\in I_0\colon p\ \textrm{is an isolated point of}\\
 &\qquad\qquad\qquad\qquad S\cap\BZ_{\RR}(P_i)\cap\BZ_{\RR}(Q_i)\ \textrm{for some index}\ i\},\\
 I_0^* &= I_0\backslash I_0^\prime.
\end{split}
\end{equation*}

By Corollary \ref{surfacesEntering3DCellCor},
\begin{equation}
\begin{split}
|I_0^\prime|&\lesssim n \sum_{i=1}^\ell D_iE_i\\
&\lesssim m^{\frac{k}{2k-1}}n^{\frac{2k-2}{2k-1}}.
\end{split}
\end{equation}
The real difficulty will be to bound $|I_0^*|$.

For each $i=1,\ldots,\ell$, let $V_i=\BZ_{\CC}(P_i)\cap\BZ_{\CC}(Q_i)$, and let $V=\bigcup_i V_i$. For each $S\in\sphrs_2$, $S^*\cap V$ is a union of isolated points and irreducible one-dimensional varieties (scheme-theoretically, $S^*\cap V$ may contain curves with embedded points, but we are only looking at the intersection set-theoretically).

If $\gamma$ is an irreducible one-dimensional variety from the above decomposition, we define $i(\gamma)$ be the smallest index $i$ so that $\gamma\subset V_i$. For each $i=1,\ldots,\ell$, define
\begin{equation*}
 \begin{split}
 \Gamma_{S,i}&=\{\gamma\ \textrm{an irreducible component of}\ S^*\cap V,\ i(\gamma)=i\},\\
 \Gamma_{S,i}^*&=\{\gamma\in\Gamma_{S,i}\colon \dim_{\RR}(\gamma(\RR))=1\}.
 \end{split}
\end{equation*}

Recall that we have partitioned the points $\pts\cap V$ in a similar fashion into the sets $\{\pts_i\}$. Thus, if $(p,S)\in I_0^*$ and $p\in \pts_i$, then at least one of the following two things must happen:
\begin{itemize}
 \item There exists $\gamma\in\Gamma_{S,i}^*$ so that $p\in\gamma$.
 \item There exists an index $j>i$ and some $\gamma\in\Gamma_{S,j}^*$ so that $p\in\gamma$. In addition, $\gamma\cap V_i$ is a discrete set.
\end{itemize}

We will now describe several different types of incidences, and bound each type in turn

\subsection{Different types of curves and incidences}\label{diffTypeCurvesIncidencesSec}
Let
\begin{equation*}
 \begin{split}
  \Gamma_{S,i}^{(1)}&=\{\gamma\in\Gamma_{S,i}^*\colon \gamma\subset (V_i)_{\sing}\},\\
  \Gamma_{S,i}^{(2)}&=\{\gamma\in\Gamma_{S,i}^*\backslash \Gamma_{S,i}^{(1)} \colon T_x(V_i)=T_x(S)\ \textrm{for a generic point}\ x\ \textrm{of}\ \gamma\},\\
  \Gamma_{S,i}^{(3)}&=\Gamma_{S,i}^*\backslash(\Gamma_{S,i}^{(1)}\cup \Gamma_{S,i}^{(2)}).
 \end{split}
\end{equation*}

We will now define several types of incidences. Let

\begin{equation}\label{defnOfAlphaS}
\alpha_{S}=\bigcup_{i}\bigcup_{\gamma\in\Gamma_{S,i}^*}\gamma.
\end{equation}
Note that if $(p,S)\in I_0^*$ and $p\in (\alpha_{S})_{\smooth} $, then there is a unique index $j$ and a unique curve $\gamma\in \Gamma_{S,j}^*$ that contains $p$. If it is clear from context, we will simply call this curve $\gamma$. Define
\begin{equation*}
 \begin{split}
  I_1 &= \{(p,S)\in I_0^*\colon p\in (\alpha_{S})_{\sing} \},\\
  I_2 &= \{(p,S)\in I_0^*\backslash I_1 \colon p\in\pts_i, \gamma\in \Gamma_{S,j}^*,\ \textrm{for some}\ j>i\},\\
  I_3 &= \{(p,S)\in I_0^*\backslash I_1 \colon p\in\pts_i, \gamma\in \Gamma_{S,i}^{(1)}\},\\
  I_4 &= \{(p,S)\in I_0^*\backslash I_1 \colon p\in\pts_i, \gamma\in \Gamma_{S,i}^{(2)}\},\\
  I_5 &= \{(p,S)\in I_0^*\backslash I_1 \colon p\in\pts_i, \gamma\in \Gamma_{S,i}^{(3)},\ p\in(V_i)_{\sing}\},\\
  I_6 &= \{(p,S)\in I_0^*\backslash I_1 \colon p\in\pts_i, \gamma\in \Gamma_{S,i}^{(3)},\ p\in(V_i)_{\smooth}\}.
 \end{split}
\end{equation*}
The above definitions make reference to an index $i$. What we mean by this is that the condition must hold for \emph{some} index $i$.

We will now bound the incidences $I_1,\ldots,I_5$. Bounding $I_6$ will require significant new tools, so this will be done in Section \ref{STOnAVarietySec}.

\subsection{Bounding $I_1,I_2,I_3$: counting singular points on algebraic curves}\label{I1I2I3Sec}

We will begin with $I_1$. For each $S\in\sphrs_2$, $\alpha_{S}$ is an algebraic curve of degree $O(\sum_{i=1}^\ell D_i)=O(D)$. Thus it has at most $O(D^2)$ singular points, so
\begin{equation}\label{I1Bound}
\begin{split}
|I_1|&\lesssim nD^2\\
& \lesssim m^{\frac{k}{2k-1}}n^{\frac{2k-2}{2k-1}}.
\end{split}
\end{equation}

Next, we will see that $I_2$ is empty. Fix $S\in\sphrs_2$. Let $p\in\pts_i$ and suppose the following conditions hold
\begin{itemize}
\item $p$ lies on only one irreducible component (i.e.~one curve) $\gamma$ of $S^*\cap V.$
\item $p$ is a smooth point of $\gamma$.
\item $\gamma\in \Gamma_{S,j}^*$ for some $j>i$.
\end{itemize}
Then $p$ is an isolated point of $S^*\cap V_j$, so $(p,S)\in I_0^\prime.$ In particular, this implies that $(p,S)\notin I_2$. We conclude that
\begin{equation}\label{I2Bound}
 |I_2|=0.
\end{equation}

We will now bound $I_3$. For each index $i$, use Lemma \ref{trappingSingLocusVariety} to find a curve $\zeta_i$ of degree $O((D_iE_i)^2)$ that contains $(V_i)_{\sing}$. Apply Lemma \ref{incidencesOnACurveLem} to bound:
\begin{equation}\label{boundI3Eqn}
\begin{split}
|I_3|&\leq \sum_{i}\Big(|\pts_i| + \sum_i (\deg\zeta_i)^2\Big)\\
&\lesssim m + \sum_i (D_iE_i)^4\\
&\leq m + \Big(\sum_i D_iE_i\Big)^4\\
&\lesssim m+m^{\frac{k}{2k-1}}n^{\frac{2k-2}{2k-1}}.
\end{split}
\end{equation}
On the second-last line we used the observation that $D_iE_i\geq 0$ for each index $i$. On the last line we used the assumption that $m\leq n^{\frac{2k+2}{3k}}$.
Thus we have
\begin{equation}\label{I3Bound}
|I_3|\lesssim m^{\frac{k}{2k-1}}n^{\frac{2k-2}{2k-1}}+m.
\end{equation}
\begin{rem}\label{whereWeUseAssumtionRemark}
\eqref{boundI3Eqn} and \eqref{controlOfI4Eqn} are the only two places where we use the assumption that $m\leq n^{\frac{2k+2}{3k}}$. Thus, if these two arguments could be avoided, we could remove the restriction that $m\leq n^{\frac{2k+2}{3k}}$ in the statement of Theorem \ref{surfacesInR4Thm}. This will be discussed further in Section \ref{removingRestrictionSection}.
\end{rem}
\subsection{Bounding $I_4$: Tangential surface intersections}\label{I4Sec}
\begin{lem}\label{findTildeQ}
Let $W\subset\CC^d$ be an irreducible variety, and let $R\in\CC[x_1,\ldots,x_d]$ be a non-zero polynomial. Suppose that $R$ vanishes on $W$. Then there exists a polynomial $\tilde R$ with the following properties
\begin{enumerate}
 \item $\deg \tilde R\leq\deg R$.
 \item $\tilde R$ vanishes on $W$.
 \item $\nabla\tilde R$ does not vanish identically on $W$.
\end{enumerate}
\end{lem}
\begin{proof}[Proof sketch]
The proof follows similar ideas to the proof of Lemma \ref{makeAPolyProdRealIdeals}, so for brevity we will only sketch it here. For each variety $W$, we will prove the result by induction on $\deg R$. If $\deg R=1$, then $\nabla R$ is non-zero everywhere, so we are done. Now suppose the result has been proved for all polynomials of degree at most $D$, and let $R$ be a polynomial of degree $D+1$. Suppose $\nabla\tilde R$ vanishes identically on $W$. Let $v$ be a generic vector. Then $R^\prime=v\cdot\nabla R$ is not the zero-polynomial, and $R^\prime$ vanishes identically on $W$. We can thus apply the induction hypothesis to $R^\prime$.
\end{proof}
\begin{lem}
For each index $i=1,\ldots,\ell$, we have the bound

\begin{equation}\label{numIncidencesCurveTangY}
|\{(p,S)\in I_4\colon p\in\pts_i\}|\lesssim  nD_iE_i+D_i^4E_i^2 + |\pts_i|.
\end{equation}
\end{lem}
\begin{proof}
Let $\{V_{i,j}\}$ be the irreducible components of $V_i.$ For each index $j$, let $\pts_{i,j}$ be the points of $\pts_i$ lying in $V_{i,j}$ that have not already been placed in some previous $\pts_{i,j^\prime}$ with $j^\prime<j$. Similarly, for each curve $\gamma\in\Gamma_{S,i}^{(2)}$, let $j(\gamma)$ be the smallest index so that $\gamma\subset V_{i,j}$. Define
\begin{equation*}
 \Gamma_{S,i,j}^{(2)}=\{\gamma\in\Gamma_{S,i}^{(2)}\colon j(\gamma)=j \}.
\end{equation*}

We will divide the incidences of $I_4$ into several types. Define
\begin{equation*}
I_4^{\prime}=\{(p,S)\in I_4\colon p\in (V_i)_{\smooth}\}.
\end{equation*}
From tangent space considerations (see Lemma \ref{controlOfSmoothPtsLem}) we have $|I_4^{\prime}|\leq|\pts_i|$. If $p\in\pts_i$, and $(p,S)\in I_4\backslash I_4^\prime$, then $p$ is incident to precisely one curve $\gamma\in \Gamma_{S,i}^{(2)}$. Define

\begin{equation*}
\begin{split}
I_4^{\prime\prime}&=\{(p,S)\in I_4\colon p\in (V_i)_{\sing},\ p\in\pts_{i,j},\ \gamma\in\Gamma_{S,i,j}^{(2)}\},\\
I_4^{\prime\prime\prime}&=\{(p,S)\in I_4\colon p\in (V_i)_{\sing},\ p\in\pts_{i,j},\ \gamma\in\Gamma_{S,i,j^\prime}^{(2)}\ \textrm{for some}\ j^\prime>j\}.
\end{split}
\end{equation*}

We will first consider $I_4^{\prime\prime\prime}$.  If $(p,\gamma)\in I_4^{\prime\prime\prime}$, then $p$ is an isolated point of $S^*\cap V_{i,j}$ (if not then the incidence $p$ would have been counted in $I_1$). Applying Proposition \ref{BezoutNonProper}, we have
\begin{equation}
\begin{split}
|I_4^{\prime\prime\prime}|&\leq\sum_{S\in\sphrs_2}\sum_{j}(\textrm{number of isolated points of}\ S^*\cap V_{i,j})\\
&\leq nD_iE_i.
\end{split}
\end{equation}

It remains to count $I_4^{\prime\prime}.$ Fix $j$. Let $\tilde P_{i,j}$ be the polynomial obtained by applying Lemma \ref{findTildeQ} to the polynomial $P_i$ and the variety $V_{i,j}$. Let $\zeta_{i,j}=\BZ_{\CC}(\tilde P_{i,j})\cap V_{i,j}.$ If $p\in V_{i,j}\backslash\zeta_{i,j}$ and if $(p,S)\in I_4^{\prime\prime}$, then $T_pS$ must lie in $T_p(\BZ_{\CC}(\tilde P_{i,j}))$, which is a three-dimensional vector space. Thus, for each point $p\in V_{i,j}\backslash\zeta_{i,j}$, there can be at most one $S\in\sphrs_2$ with $(p,S)\in I_4^{\prime\prime}$, so the total number of incidences of this type is at most $|\pts_i|$.

Finally, $\zeta_{i,j}$ is a curve of degree $O(D_i^2E_i)$, so by Lemma \ref{incidencesOnACurveLem}, the total number of incidences $(p,S)\in I_4^{\prime\prime}$ with $p\in\zeta_{i,j}$ is  $O(D_i^4E_i^2+|\pts_i|)$. We conclude that $|I_4^{\prime\prime}|\lesssim D_i^4E_i^2+|\pts_i|$.
\end{proof}
Summing the bound \eqref{numIncidencesCurveTangY} over all indices $i$ and using the assumption (from the statement of Theorem \ref{surfacesInR4Thm}) that $m\leq n^{\frac{2k+2}{3k}}$, we obtain the bound
\begin{equation}\label{controlOfI4Eqn}
|I_4| \lesssim m^{\frac{k}{2k-1}}n^{\frac{2k-2}{2k-1}}+m.
\end{equation}

\subsection{Bounding $I_5$: Transverse surface intersections}\label{I5Sec}
\begin{lem}
Let $Y\subset\CC^4$ be a smooth bounded-degree two-dimensional variety, and let $W\subset \CC^4$ be a two-dimensional variety. Let $\Gamma$ be the set of irreducible one-dimensional components of $Y\cap W$. Let
\begin{equation*}
\Gamma^\prime\subset\{\gamma\in\Gamma\colon z\in W_{\smooth}\ \textrm{and}\ T_z(Y)\neq T_z(W)\ \textrm{for generic}\ z\in\gamma \}.
\end{equation*}

Then
\begin{equation}\label{gammaWSing}
 \sum_{\gamma\in\Gamma^\prime} |\gamma\cap W_{\sing}| \leq C\Big( \deg W+ \Big(  \sum_{\gamma\in\Gamma}(\deg\gamma)\Big)^2\Big),
\end{equation}
where the constant $C$ depends only on $\deg Y$.
\end{lem}
\begin{proof}
Let $v\in\CC^4$ be a generic vector and let $\pi_v\colon\CC^4\to\CC^3$ be the projection in the direction $v$. Let
\begin{equation*}
Y^{\dagger}=\pi_v^{-1}(\pi_v(Y)).
\end{equation*}
Then since $v$ was chosen generically with respect to $Y$, $Y^{\dagger}$ is a smooth three-dimensional variety of degree $\deg(Y)$. Furthermore, $Y^{\dagger}\cap W$ is a one-dimensional curve, and each $\gamma\in\Gamma$ is a (irreducible) component of $Y^{\dagger}\cap W$. Let $\zeta$ be the union of all irreducible components of $Y^\dagger\cap W$ that are not contained in $Y$. We have $\deg\zeta\leq\deg(Y^{\dagger}\cap W)\leq C\deg W$, where the constant $C$ depends only on $\deg Y$.

Let $\gamma\in\Gamma^\prime$, and let $z\in \gamma\cap W_{\sing}$. Then $\pi_v(z)\in \pi_v(Y^{\dagger})=\pi_v(Y)$ and $z\in \pi_v(W_{\sing})\subset \overline{\pi_v(W)}_{\sing}$. Note as well that $\pi_v(Y)$ and $\pi_v(W)$ intersect transversely at a generic point of $\pi_v(\gamma)$. Thus by Lemma \ref{singPtCor}, we have that $\pi(z)$ is a singular point of $\pi_v(Y)\cap\pi_v(W)$. But this implies that $z$ is a singular point of $Y^\dagger\cap W$. In particular, at least one of the following three things must occur:
\begin{itemize}
 \item \itemizeEqnVSpacing \begin{equation}\label{zInAlphaSZeta}
z\in \zeta\cap Y,
\end{equation}
or
\item \itemizeEqnVSpacing \begin{equation}\label{zInAlphaS}
z\in\gamma_{\sing},
\end{equation}
or
\item \itemizeEqnVSpacing \begin{equation}\label{zHitsSomeOtherGamma}
z\in \bigcup_{\gamma^\prime\in\Gamma,\ \gamma^\prime\neq\gamma}\gamma\cap\gamma^\prime.
\end{equation}
\end{itemize}

Now, considering all $\gamma\in\Gamma^\prime$, we conclude there can be $O(W)$ points of the form \eqref{zInAlphaSZeta}, at most $\sum_{\gamma\in\Gamma^\prime}(\deg\gamma)^2$ points of the form \eqref{zInAlphaS}, and at most $\Big(  \sum_{\gamma\in\Gamma}(\deg\gamma)\Big)^2$ points of the form \eqref{zHitsSomeOtherGamma}. This establishes \eqref{gammaWSing}.
\end{proof}
Applying Lemma \ref{gammaWSing} to the sets $\{V_i\}$ and bounded-degree smooth surfaces $\{S^*\colon S\in\sphrs_2\},$ with $\Gamma^\prime = \Gamma_{S,i}^{(3)},$ we conclude that
\begin{equation}
\begin{split}
|I_5|&\lesssim n\sum_i D_iE_i + n\sum_i D_i^2\\
&\lesssim m^{\frac{k}{2k-1}}n^{\frac{2k-2}{2k-1}}.
\end{split}
\end{equation}

It remains to bound $|I_6|$. Doing so will require several new tools, which we will discuss in the next section.
\section{Interlude: The gauss map of a Variety}\label{diffGeomSec}
\subsection{Grassmannians and the Gauss map}\label{GaussMapSection}
Let $F=\RR$ or $\CC$. If $v_1,\ldots,v_k\in F^d$ are vectors, let $\langle v_1,\ldots,v_k\rangle\subset F^d$ be the vector space spanned by $v_1,\ldots,v_k$. In practice, the vectors $v_1,\ldots,v_k$ will be linearly independent. Let $\grass(k,d;F)$ be the Grassmannian of $k$--dimensional vector subspaces of $F^d$. We will identify elements of $\grass(k,d;F)$ with planes in $F^d$ passing through the origin, and we will usually use the variable $\Pi$ for these planes.

$\grass(k,d;\CC)$ has the structure of a projective variety (see i.e.~\cite[chapter 6]{harris}). $\grass(k,d;F)$ is a smooth variety, and also a smooth (real or complex) manifold. If $\Pi\in\grass(k,d;F)$, then $T_\Pi\grass(k,d;F)$ is the tangent plane to $\grass(k,d;F)$ at $\Pi$, and $T\grass(k,d;F)$ is the tangent bundle.

If $\Pi,\Pi^\prime$ are vector spaces, let $\Pi+\Pi^\prime$ denote the sum of the two vector spaces. We have
\begin{equation*}
\dim_F(\Pi+\Pi^\prime)=\dim(\Pi)+\dim_F(\Pi^\prime)-\dim_F(\Pi\cap \Pi^\prime).
\end{equation*}

\begin{defn}\label{defnOfAv}
If $\Pi^\prime\in \grass(d-k,d;F)$, let
\begin{equation}
A_{\Pi^\prime}=\{\Pi\in\grass(k,d;F)\colon\dim_{F}(\Pi+\Pi^\prime)<d\}.
\end{equation}
$A_{\Pi^\prime}$ is a codimension-one sub-variety of $\grass(k,d;F)$.
\end{defn}
\subsubsection{Orientation}\label{orientationSection} When working over $\RR$, we will frequently consider pairs of vector spaces $(\Pi,\Pi^\prime)$ with $\Pi\in\grass(k,d;\RR)$ and $\Pi^\prime\in\grass(d-k,d,\RR)$. If $\Pi=\langle v_1,\ldots,v_k\rangle$ and $\Pi^\prime =\langle v_{k+1},\ldots,v_d\rangle$, where $v_1,\ldots,v_k$ and $v_{k+1},\ldots,v_{d}$ are orthogonal unit vectors, then $\dim_{\RR}(\Pi+\Pi^\prime)<d$ if and only if $\det(v_1,\ldots,v_d)=0$. If $\dim_{\RR}(\Pi+\Pi^\prime)=d,$ we wish to make sense of the expression $\det(\Pi,\Pi^\prime)$.

A reasonable first definition would be to define $\det(\Pi,\Pi^\prime)=\det(v_1,\ldots,v_d)$. While the magnitude $|\det(\Pi,\Pi^\prime)|$ is well-defined, the sign is not---if we permute two vectors from $\{v_1,\ldots,v_k\}$ or $\{v_{k+1},\ldots,v_d\}$, then the sign of the above determinant changes.

Ideally, we would like to find continuous functions $v_1(\Pi),\ldots,v_k(\Pi)$ and $v_{k+1}(\Pi^\prime),\ldots,v_{d}(\Pi^\prime)$ so that we can define
\begin{equation}\label{provisionalDet}
\det(\Pi,\Pi^\prime)=\det\big(v_1(\Pi),\ldots,v_k(\Pi),v_{k+1}(\Pi^\prime),v_d(\Pi^\prime)\big).
\end{equation}
While we cannot do this globally, we can do so locally.
\begin{lem}\label{ctsVecLem}
Fix $\Pi_0\in\grass(k,d;\RR)$ and $\Pi_0^\prime\in\grass(d-k,d,\RR)$. Then there exist small neighborhoods $U\subset \grass(k,d;\RR)$ and $U^\prime\subset \grass(d-k,d;\RR)$ of $\Pi$ and $\Pi^\prime$, respectively, and functions $v_1,\ldots,v_k\colon U\to\RR^d$, $v_{k+1},\ldots,v_d\colon U^\prime\to\RR^d$ so that $v_1(\Pi),\ldots,v_k(\Pi)$ and $v_{k+1}(\Pi^\prime),\ldots,v_{d}(\Pi^\prime)$ are orthogonal unit vectors, and
\begin{equation*}
\Pi=\langle v_1(\Pi),\ldots,v_k(\Pi)\rangle,\ \ \Pi^\prime = \langle v_{k+1}(\Pi^\prime),\ldots,v_{d}(\Pi^\prime)\rangle.                                                                                                                                                                                                                                                                                                                                                                                                                                                                                                                                                                                                                                                                                                                                                                                                                                                                                                                              \end{equation*}
\end{lem}
We can now make sense of the expression \eqref{provisionalDet}. This observation will be used in the next lemma, which is an analogue of the
intermediate value theorem. In this lemma it is essential that we work over $\RR$.
\begin{lem}[Intermediate value theorem]\label{intermediateValueThmLem}
Let $\Pi_0\in \grass(k,d;\RR)$, $\Pi_0^\prime\in \grass(d-k,d;\RR)$. Then we can find a neighborhood $U\subset\grass(k,d;\RR)$ of $\Pi_0$ so that the following holds. If $\eta\colon [0,1]\to U$ is continuous, and if
\begin{equation*}
\Big(\det(\eta(0),\Pi^\prime)\Big)\Big(\det(\eta(1),\Pi^\prime)\Big)<0,
\end{equation*}
then there exists $t\in(0,1)$ so that $\eta(t)\in A_{\Pi^\prime}$.
\end{lem}
\begin{proof}
Let $U$ be the neighborhood of $\Pi_0$ from Lemma \ref{ctsVecLem}, and let $v_1(\Pi),\ldots,$ $v_k(\Pi)$ and $v_{k+1}(\Pi^\prime),\ldots,v_{d}(\Pi^\prime)$ be the corresponding unit vectors. Then the function
\begin{equation}
f(t)=\det\big(v_1(\eta(t)),\ldots,v_k(\eta(t)),v_{k+1}(\Pi_0^\prime),\ldots,v_{d}(\Pi_0^\prime)\big)
\end{equation}
is continuous, and $f(0)f(1)<0$. So by the intermediate value theorem, we can find $t\in(0,1)$ with $f(t)=0$. But this implies that $\eta(t) \in A_{\Pi_0^\prime}$.
\end{proof}
\subsubsection{The Gauss map and Gauss image of a variety}
In this section we will define the Gauss map and discuss a few of its properties. Further information can be found in \cite[Chapter 15]{harris}.
\begin{defn}
Let $Z\subset\CC^d$ be a $k$--dimensional variety. For each point $z\in Z_{\smooth},$ we define the Gauss map
\begin{equation}\label{GaussMapDefn}
G(z;Z)= T_z(Z)\in\grass(k,d;\CC),
\end{equation}
and the extended Gauss map
\begin{equation}\label{GaussMapDefn}
G^\dagger(z;Z)= (z,T_z(Z))\in\CC^d \times \grass(k,d;\CC).
\end{equation}

Following notation from \cite{harris}, we define $\mathcal{F}(Z)$ to be the Zariski closure of $G(Z_{\smooth};Z)$, and we define $\mathcal{F}^{\dagger}(Z)$ to be the Zariski closure of $G^\dagger(Z_{\smooth};Z)$.

\end{defn}

\subsubsection{Some transversality  arguments}
Let $Z\subset\CC^d$ be a $k$--dimensional variety, and let $\zeta\subset Z$ be an irreducible curve, with $\zeta\not\subset Z_{\sing}$. We will be interested in the tangent planes to $Z$ at points $z\in\zeta$. To make this more precise, define
\begin{equation*}
\mathcal{G}_{Z,\zeta} = \overline{G(\zeta\cap Z_{\smooth};Z)},
\end{equation*}
The idea is that $\mathcal{G}_{Z,\zeta}$ is the closure of the set of $k$--planes tangent to $Z$ at some point $z\in\zeta\cap Z_{\smooth}$. $\mathcal{G}_{Z,\zeta}$ is a variety of dimension at most one.

\begin{lem}\label{fixedPivLem}
Let $\Pi\in\grass(k,d;\CC),$ and let $v\in T_\Pi(\grass(k,d;\CC))$ be generic (with respect to $\Pi$). Then the variety
\begin{equation*}
X_{\Pi,v}:=\overline{\{\Pi^\prime\in\grass(d-k,d;\CC)\colon \Pi\in (A_{\Pi^\prime})_{\smooth},\ v\in T_{\Pi}(A_{\Pi^\prime})\}}
\end{equation*}
is a subvariety of $\grass(d-k,k;\CC)$ of codimension $\geq 2.$
\end{lem}
\begin{proof}
This is simply the observation that the requirements $\Pi\in A_{\Pi^\prime}$ and $v\in T_{\Pi}(A_{\Pi^\prime})$ are independent constraints on $\Pi^\prime$.

More precisely, note that $\{\Pi^\prime\in\grass(d-k,d;\CC)\colon \Pi\in A_{\Pi^\prime}\}$ is a codimension-one subvariety of $\grass(d-k,k;\CC)$ (indeed, it is isomorphic to $A_{\Pi}$) that contains $X_{\Pi,v}$. Suppose there is some irreducible component $Z\subset \{\Pi^\prime\in\grass(d-k,d;\CC)\colon \Pi\in A_{\Pi^\prime}\}$ that is contained in $X_{\Pi,v}$. In particular, this would imply that the set $\{\Pi^\prime\in Z\colon \Pi\in (A_{\Pi^\prime})_{\smooth}\}$ is Zariski-dense in $Z$.

If $Z$ is contained in $X_{\Pi,v}$, this implies that for a generic choice of $v\in T_\Pi(\grass(k,d;\CC))$, we have $v\in T_{\Pi}(A_{\Pi^\prime})$ for a dense set of $\Pi^\prime\in Z$. This implies that for a dense set of $\Pi^\prime\in Z$, $ v\in T_{\Pi}(A_{\Pi^\prime})$ for generic (and thus every) $v\in T_\Pi(\grass(k,d;\CC))$. But, if $\Pi^\prime$ is a smooth point of $Z$ then the tangent plane $T_{\Pi}(Z)$ has codimension-one in $T_{\Pi}A_{\Pi^\prime}$.

We conclude that no irreducible components of $\{\Pi^\prime\in\grass(d-k,d;\CC)\colon \Pi\in A_{\Pi^\prime}\}$ can be contained in $X_{\Pi,v}$. Thus the codimension of $X_{\Pi,v}$ in $\grass(d-k,k;\CC)$ is at least two.
\end{proof}

\begin{lem}\label{genericCurveGrassmanIntersect}
Let $Z\subset\CC^d$ be a $k$--dimensional variety, and let $\zeta\subset Z$ be an irreducible curve, with $\zeta\not\subset Z_{\sing}$. If we select a generic (with respect to $Z$ and $\zeta$) element $\Pi^\prime\in\grass(d-k,d,\CC),$ then for all pairs $(z,\Pi)$ with $z\in\zeta_{\smooth}\cap Z_{\smooth},$ $(z,\Pi)\in \mathcal{F}^{\dagger}(Z)$, and $\Pi\in A_{\Pi^\prime}$, we have that $\Pi$ is a smooth point of $A_{\Pi^\prime}$, and $\mathcal{G}_{Z,\zeta}$ is transverse to $A_{\Pi^\prime}$ at $\Pi$.
\end{lem}
\begin{proof}
The idea is the following. For each point $z\in \zeta\cap Z_{\smooth}$ there is a unique tangent plane $T_z(Z)\in\grass(k,d;\CC)$. The set of all such points forms a curve $\alpha$ in $\grass(k,d;\CC)$ (technically, we need to take the closure of this curve in the Zariski topology). Now, given a $(d-k)$--plane $\Pi^\prime\in\grass(d-k,d;\CC)$, we can ask: does the variety $A_{\Pi^\prime}$ meet the curve $\alpha$ tangentially? We will show that for generic $\Pi^\prime\in\grass(d-k,d;\CC)$, the answer is no. The reason is that for each point $\Pi\in\alpha$, the set of planes $\Pi^\prime\in\grass(d-k,d;\CC)$ that hit the point $\Pi$ and that are tangent to $\alpha$ at $\Pi$ are contained in a codimension-two sub-variety of $\grass(d-k,d;\CC)$. Since $\alpha$ is one-dimensional, the set of planes $\Pi^\prime$ that are tangent at \emph{some} point $\Pi\in\alpha$ is contained in a codimension--one variety. This means that generically, this doesn't happen.

Now for the details. Let
\begin{equation}
\begin{split}
M = \{((\Pi,v),&\Pi^\prime)\in (T\grass(k,d;\CC))\times \grass(d-k,d;\CC)\colon \\
&\Pi\in \mathcal{G}_{Z,\zeta},\ v\in T_{\Pi}(\mathcal{G}_{Z,\zeta}), \Pi\in (A_{\Pi^\prime})_{\smooth},\ v\in T_{\Pi}(A_{\Pi^\prime})\}.
\end{split}
\end{equation}
We wish to show that $\dim \overline{M}\leq \dim(\grass(d-k,d;\CC))-1.$ By Lemma \ref{fixedPivLem}, for each $(\Pi,v)\in (T\grass(k,d;\CC))$ with $\Pi\in\mathcal{G}_{Z,\zeta}$ and $v\in T_{\Pi}(\mathcal{G}_{Z,\zeta}),$ the set of $\Pi^\prime$ so that $((\Pi,v),\Pi^\prime)\in M$ is contained in a variety of dimension $\dim(\grass(d-k,d;\CC))-2$, i.e. each fiber of the projection map $M\to T\grass(k,d;\CC)$ has dimension at most $\dim(\grass(d-k,d;\CC))-2$. But since $\mathcal{G}_{Z,\zeta}$ has dimension at most one, the image of the map $M\to T\grass(k,d;\CC)$ has dimension at most 1, so $M$ has dimension at most $\dim(\grass(d-k,d;\CC))-1.$

Thus if we consider the projection $M\to \grass(d-k,d;\CC),$ the image of this projection is contained in a proper sub-variety of $\grass(d-k,d;\CC)$, i.e.~a generic element $\Pi^\prime\in \grass(d-k,d;\CC)$ is not contained in the image of the projection.

All that remains is to show that for generic $\Pi^\prime$, $\Pi$ is a smooth point of $A_{\Pi^\prime}$ for all $(z,\Pi)$ with $z\in\zeta_{\smooth}\cap Z_{\smooth},\ (z,\Pi)\in \mathcal{F}^\dagger(Z)$, and $\Pi\in\mathcal{A}_{\Pi^\prime}$. But if we define
\begin{equation}
\begin{split}
M_1 = \{((\Pi,v),&\Pi^\prime)\in (T\grass(k,d;\CC))\times \grass(d-k,d;\CC)\colon \\
&\Pi\in \mathcal{G}_{Z,\zeta},\ v\in T_{\Pi}(\mathcal{G}_{Z,\zeta}), \Pi\in (A_{\Pi^\prime})_{\sing}\},
\end{split}
\end{equation}
then since $\dim (A_{\Pi^\prime})_{\sing}\leq\dim(d-k,d;\CC)-2,$ a similar argument shows that a generic element $\Pi^\prime\in \grass(d-k,d;\CC)$ is not contained in the image of the projection $M_1\to\grass(d-k,d;\CC)$.
\end{proof}

\begin{lem}\label{changeSignAcrossCurve}
Let $Z\subset\CC^d$ be a $k$--dimensional variety defined by real polynomials, and let $\zeta\subset Z$ be an irreducible curve defined by real polynomials. Suppose that $Z(\RR)$ is $k$--dimensional, $\zeta(\RR)$ is one-dimensional, and that $\zeta$ is not contained in $Z_{\sing}$. Let $\Pi^\prime\in\grass(d-k,d,\RR)$ be chosen generically with respect to $Z$ and $\zeta$.

Let $z\in \zeta(\RR)_{\smooth}\cap Z(\RR)_{\smooth}$ and suppose that $T_z(Z)\in A_{\Pi^\prime}$. Then for all $\epsilon>0$ we can find an interval $I\subset\zeta$ centered at $z$ with the following two properties. 1) $I$ has arclength $\leq\epsilon$. 2) If $z_1,z_2$ are the two endpoints of $I$, and if $\Pi_1 = T_{z_1}(Z(\RR)),\ \Pi_2 = T_{z_2}(Z(\RR)),$ then
\begin{equation}\label{prodDetLZero}
\Big(\det(\Pi_1,\Pi^\prime)\Big)\Big(\det(\Pi_2,\Pi^\prime)\Big)<0.
\end{equation}
\end{lem}
\begin{rem} See Section \ref{orientationSection} for a discussion of the definition of $\det$. In this context, the determinant is only defined up to a choice of sign. However, the statement that two determinants have opposite sign is well-defined regardless of the orientation chosen (again, in a small neighborhood of a point $\Pi\in\grass(k,d;\RR)$).
\end{rem}
\begin{proof}[Proof of Lemma \ref{changeSignAcrossCurve}]
In brief, the proof consists of the following observation. If $U\subset\RR^\ell$ is an open set and $A\subset U$ is a smooth manifold such that $U\backslash A$ contains two connected components, and if $\alpha\subset U$ is a smooth curve that meets $A$ transversely at the point $x\in \alpha \cap A$, then $\alpha$ must enter both connected components of $U\backslash A$. Furthermore, we can find points $x_1,x_2\in\alpha$ arbitrarily chose to $x$ such that $x_1,x_2$ lie in separate connected components of $U\backslash A$.

Now for the proof. Let $\Pi_0=T_z(Z)$. By Lemma \ref{genericCurveGrassmanIntersect}, $\Pi^\prime$ is a smooth point of $A_{(\Pi^\prime)^*}$, and  $\mathcal{G}(\zeta,Z)$ is transverse to $A_{(\Pi^\prime)^*}$ at $\Pi^\prime$ (recall that $\Pi^\prime$ is a real $(d-k)$--plane, and $(\Pi^\prime)^*$ is its complexification). Since $\zeta(\RR)$ is one-dimensional and $\Pi^\prime$ was chosen generically, we can assume that $z$ is a smooth point of $\zeta$, and thus $\zeta(\RR)$ is a one-dimensional smooth manifold in a (Euclidean) neighborhood of $z$.

We conclude that $\mathcal{G}(\zeta(\RR),Z(\RR))$ is a smooth curve in $\grass(k,d;\RR)$ in a neighborhood of the (real) $k$--plane $\Pi_0(\RR)\in\grass(k,d;\RR)$, and this curve is transverse to $A_{\Pi^\prime(\RR)}$ at the point $\Pi_0(\RR)$. Use Lemma \ref{ctsVecLem} to choose a small neighborhood $U$ of $\Pi_0(\RR),$ and functions $v_1(\Pi),\ldots,v_k(\Pi)\colon U\to\RR^d$ so that $\Pi=\operatorname{span}\{v_1(\Pi),\ldots,v_k(\Pi)\}$. This allows us to define $\det(\Pi,\Pi^\prime)$ for all $\Pi\in U$.

Now, after possibly shrinking $U$, $A_{\Pi^\prime}\cap U$ is a smooth manifold, and $A_{\Pi^\prime}\cap U$ cuts $\grass(k,d;\RR)\cap U$ into two connected regions; one where $\det(\Pi,\Pi^\prime)>0$ and one where $\det(\Pi,\Pi^\prime)<0.$ Select two points $z_1,z_2\in \gamma(\RR)$ so that $T_{z_1}(\gamma(\RR),Z(\RR))\in U,\ T_{z_2}(\gamma(\RR),Z(\RR))\in U,$ and with $T_{z_1}(\gamma(\RR),Z(\RR))$ and $T_{z_2}(\gamma(\RR),Z(\RR))$ in opposite regions. Since $\mathcal{G}(\zeta(\RR),Z(\RR))$ is transverse to $A_{\Pi^\prime}\subset\grass(k,d;\RR)$ at $\Pi_0(\RR)$, \eqref{prodDetLZero} holds with this choice of $z_1,z_2$.
\end{proof}

\subsubsection{Complex varieties and some perturbative arguments}
\begin{defn}
If $k<d$ and $\pi_v\colon F^d\to F^{d-1}$ is a projection in the direction $v$, let
\begin{equation*}
\tilde\pi\colon\{\Pi\in \grass(k,d;F)\colon v\notin \Pi\} \to\grass(k,d-1;F)
\end{equation*}
be the associated map on the Grassmannian.
\end{defn}

We end with the following observation. Let $Z\subset\CC^4$ be a two-dimensional variety and let $\pi_v\colon\CC^4\to\CC^3$ be a projection. If $z\in Z_{\smooth},$ $\Pi\in G(z;Z),$ and $v\notin \Pi,$ then $\tilde\pi(\Pi)$ is two-dimensional and $\tilde\pi(\Pi)\in G(z;\overline{\pi(Z)})$.

\subsection{Perturbations and the Gauss map}
In this section we will prove a technical lemma that will be useful when we have to cut the surfaces $V_j(\RR)$ into pieces. In the next section we will be confronted with an irreducible two-dimensional variety $W$ that is a component of the intersection $Z_{\CC}(R_1)\cap Z_{\CC}(R_2)$. We will need to understand smooth points $z\in W_{\smooth}$ where the tangent plane $T_zW$ has certain properties. Ideally, $T_zW$ would be given by the two-dimensional vector space orthogonal to $\langle\nabla R_1(z),\nabla R_2(z)\rangle$. However, if $\langle\nabla R_1(z),\nabla R_2(z)\rangle$ is instead a zero or one-dimensional vector space, then this will not work. Instead, we will consider $\langle\nabla R_1(z^\prime),\nabla R_2(z^\prime)\rangle,$ where $z^\prime$ is a point close to $z$. If we set things up carefully, then we can recover information about $T_zW$ from $\langle\nabla R_1(z^\prime),\nabla R_2(z^\prime)\rangle$. This is made precise in Corollary \ref{perturbGaussMapLem}.

\begin{defn}
For $R\in \CC[x_1,\ldots,x_4]$ and $z_0\in\CC^4$, define $R^{z_0}=R(z)-R(z_0)$. In particular, $\BZ_{\CC}(R^{(z_0)})=\BZ_{\CC}(R(z)-R(z_0))$. This is the level set of $R$ passing through the point $z_0$.
\end{defn}

 The following lemma is rather technical; the reader may wish to first look at Corollary \ref{perturbGaussMapLem}, which may provide some motivation.
\begin{lem}\label{projectedPerturbLem}
Let $R_1,R_2\in\CC[x_1,\ldots,x_4]$, and let $W$ be an irreducible component of $Z_{\CC}(R_1)\cap Z_{\CC}(R_2)$ and let $v,v_1\in\CC^4$ be generic vectors (see Remark \ref{explainGen}). Then there exists a curve $\zeta\subset W$ so that the following conditions hold.

\begin{itemize}
\item $\deg(\zeta)\lesssim (\deg R_1+\deg R_1)\deg W.$
 \item For all $z\in W_{\smooth}\backslash \zeta,$ $v_1\notin T_z(W)$.
 \item If $\tilde\pi=\tilde\pi_{v_1}\colon\grass(2,4;\CC)\to\grass(2,3;\CC)$ is the corresponding map on the Grassmannian, $z\in W_{\smooth}\backslash \zeta$, and if $U\subset\CC$ is a sufficiently small neighborhood of 0, then the map \begin{equation}
\begin{split}
\rho_{\pi,z}\colon U&\to \grass(2,3;\CC),\\
t&\mapsto \tilde\pi \Big(T_{z+tv}\big(Z_{\CC}(R_1^{(z+tv)})\cap Z_{\CC}(R_2^{(z+tv_1)})\big)\Big)
\end{split}
\end{equation}
is continuous on $U$.
\end{itemize}
\end{lem}
\begin{rem}\label{explainGen}
When we say that $v$ is generic with respect to $R_1,R_2,W$, we mean that given any $R_1,R_2,W$, there is a Zariski open set $O\subset \CC^4$ so that the lemma holds for any $v, v_1\in O$.
\end{rem}
\begin{rem}
Heuristically, Lemma \ref{projectedPerturbLem} says that the tangent plane to the level set of $R_1$ and $R_2$ passing through $z\in W$ is similar to that of the level set of $R_1$ and $R_2$ passing through $z+tv$, provided $|t|$ is small and $z$ does not lie on a small bad set. More precisely, Lemma \ref{projectedPerturbLem} says that the (generic) projections of the two tangent planes into $\CC^3$ are similar. Corollary \ref{perturbGaussMapLem} will let us recover the result about tangent planes in $\CC^4$.
\end{rem}

\begin{rem}Let us understand the map $\rho_{\pi,z}$. For $t\in\CC$, let $z^\prime = z+tv$. Let
\begin{equation*}
W^\prime=Z_{\CC}(R_1^{(z+tv)})\cap Z_{\CC}(R_2^{(z+tv)}).
\end{equation*}
This is the intersection of the level sets of $R_1$ and $R_2$ that pass through $z^\prime$. Then the image of $t$ under the above map is the projection of $T_{z^\prime}(W^\prime)$ to $\CC^3$ (the projection is given by $\pi$).
\end{rem}
\begin{proof}[Proof of Lemma \ref{projectedPerturbLem}]
Let $L\subset\CC^4$ be a 3--plane orthogonal to $v$, and let
\begin{equation*}
B_1^\prime=\{z\in L\colon \operatorname{rank}\binom{\pi(\nabla R_1(z+tv))}{\pi(\nabla R_2(z+tv))}\leq 1\ \textrm{for all}\ t\in\CC\}.
\end{equation*}
First, consider the set
\begin{equation}\label{nablaSingSetProj}
\{z\in \CC^4\colon \operatorname{rank}\binom{\pi(\nabla R_1(z))}{\pi(\nabla  R_2(z))}\leq 1\}.
\end{equation}
If $v_1$ (and thus $\pi$) is chosen generically\footnote{More precisely, for every choice of $R_1$ and $R_2$, there is a dense Zariski open subset of $\CC^4$ so that if $v_1$ lies in this open subset then the desired property holds} with respect to $R_1$ and $R_2$, then \eqref{nablaSingSetProj} is not all of $\CC^4$. Indeed, it is a proper algebraic variety of dimension at most three and degree $O(\deg R_1+\deg R_2)$. In particular, the intersection of \eqref{nablaSingSetProj} with a generic (with respect to $R_1, R_2, v_1$ and $v$) translate of $L$ has dimension at most two. This implies that $B_1^\prime$ is contained in a two-dimensional variety $B_1^{\prime\prime}$ of degree $O(\deg R_1+\deg R_2)$ (to obtain such a variety, simply intersect the set \eqref{nablaSingSetProj} with a generic translate of $L$).

Let $B_1^{\prime\prime\prime}=\pi_v^{-1}(\pi_v(B_1^{\prime\prime}))$ be the extension of $B_1^{\prime\prime}$ in the direction $v$. So $B_1^{\prime\prime\prime}$ is a three dimensional variety of degree $O(\deg R_1+\deg R_2).$ Let $B_1=B_1^{\prime\prime\prime}\cap W$. We can assume that $\dim_{\CC}(B_1)\leq 1$. Indeed, if $\dim_{\CC}(B_1)=2$, then $B_1=W$, and this would imply that $\eqref{nablaSingSetProj}=\CC^4$, and we have already shown that this is not the case.

Let $z\in W_{\smooth}\backslash B_1$. Then for any $t\neq 0$ in a sufficiently small (Euclidean) neighborhood of $0$, we have $\pi(\nabla R_1(z+tv))\times \pi(\nabla R_2(z+tv))\neq 0$. Note that since $\pi(\nabla R_1(z+tv))$ and $\pi(\nabla R_2(z+tv))$ are vectors in $\CC^3$, the cross product is well-defined.

Thus, we can define
\begin{equation*}
\lambda(z,t) = \frac{\pi(\nabla R_1(z+tv))\times \pi(\nabla R_2(z+tv)) }{|\pi(\nabla R_1(z+tv))\times \pi(\nabla R_2(z+tv))|}.
\end{equation*}
Note that if $z \in W\backslash B_1$, then $|\pi(\nabla R_1(z+tv))\times \pi(\nabla R_2(z+tv))|$ does not vanish identically in $t$. Write
\begin{equation*}
\pi(\nabla R_1(z+tv)\times \pi(\nabla R_2(z+tv))=(v_1(z,t),v_2(z,t),v_3(z,t)).
\end{equation*}
We can expand $v_j(z,t)=\sum t^i\theta_{i,j}(z)$. For each $j=1,2,3,$ let $i_j$ be the minimum index so that $\theta_{i,j}(z)$ doesn't vanish identically on $W$. For notational convenience, we'll assume that $i_1=\min(i_1,i_2,i_3)$ (if not, then just permute the indices). Let $B_2=W\cap Z_{\CC}(\theta_{i_1})$. Note that $\deg(\theta_{i_1})\lesssim \deg R_1 + \deg R_2$.

Let $v_j^{\dagger}(z,t)=t^{-i_1}(z,t)$, and let
\begin{equation*}
\lambda(z,t)^{\dagger}=\frac{(v_1^{\dagger}(z,t),v_2^{\dagger}(z,t),v_3^{\dagger}(z,t))}{\big(|v_1^{\dagger}(z,t)|^2+|v_2^{\dagger}(z,t)|^2+|v_3^{\dagger}(z,t)|^2\big)^{1/2}}.
\end{equation*}

Define $\zeta=B_1\cup B_2$ and let $W^\prime = W_{\smooth}\backslash\zeta$. If $z\in W^\prime$, then the denominator of $\lambda^\dagger(z,t)$ does not vanish when $t=0$, so in particular $\lambda^\dagger(z,t)$ is a smooth function of $t$ in a neighborhood of $0$. Furthermore, for $t\neq 0$, $\lambda^{\dagger}(z,t)=\lambda(z,t)$, and for all $t$ (in a neighborhood of 0), $\lambda^\dagger(z,t)$ is the normal vector to the 2--plane $\tilde\pi(\rho_z(t))$. This implies that $\tilde\pi \rho_z(t)$ is continuous for $t$ in a neighborhood of $t=0$, as desired.
\end{proof}
\begin{cor}\label{perturbGaussMapLem}
Let $R_1,R_2\in\CC[x_1,\ldots,x_4]$, and let $W$ be an irreducible component of $\BZ_{\CC}(R_1)\cap \BZ_{\CC}(R_2)$. Let $v\in\CC^4$ be a generic vector (as described in Remark \ref{explainGen}). Then there exists a curve $\zeta$ (depending only on $W$ and $v$) with
\begin{equation*}
\deg(\zeta)\lesssim (\deg R_1+\deg R_2)\deg W
 \end{equation*}
 so that if $z\in W_{\smooth}\backslash \zeta$ and if $U\subset\CC$ is a sufficiently small neighborhood of 0, then the map
\begin{equation}\label{toLevelSet}
\rho_z\colon t\mapsto T_{z+tv}\big(\BZ_{\CC}(R_1^{(z+tv)})\cap \BZ_{\CC}(R_2^{(z+tv)})\big)
\end{equation}
is continuous on $U$.
\end{cor}

\begin{proof}
Let $v_1,v_2,v_3,v_4$ be generic vectors, and apply Lemma \ref{projectedPerturbLem} to the collection $\{R_1,R_2,W,$ $v,v_i\},\ i=1,2,3,4.$ Let $U_1,U_2,U_3,U_4\subset\CC$ be the resulting open neighborhoods of 0, and let $\zeta_1,\zeta_2,\zeta_3,\zeta_4\subset W$ be the resulting curves. Let $U=U_1\cap U_2\cap U_3\cap U_4$ and let $\zeta=\zeta_1\cup\zeta_2\cup\zeta_3\cup\zeta_4$.

By Lemma \ref{projectedPerturbLem} the maps
\begin{equation}
t\mapsto \tilde\pi_{v_i} \Big(T_{z+tv}\big(Z_{\CC}(R_1^{(z+tv)})\cap Z_{\CC}(R_2^{(z+tv_1)})\big)\Big),\ i=1,2,3,4
\end{equation}
are continuous for $t\in U$ and $z\in W\backslash \zeta$. However, the map
\begin{equation*}
\begin{split}
\psi\colon&\grass(2,4;\CC)\to (\grass(2,3;\CC))^4,\\
&\Pi\mapsto(\tilde\pi_{v_1}(\Pi),\ldots,\tilde\pi_{v_4}(\Pi))
\end{split}
\end{equation*}
has full rank at every point $\Pi_0$ for which $(\Pi_0+\langle v_1\rangle)\cap(\Pi_0+ \langle v_2\rangle)\cap(\Pi_0+\langle v_3\rangle)\cap(\Pi_0+\langle v_4\rangle)=\Pi_0$. Since $v_1,v_2,v_3,v_4$ were chosen generically, this condition will hold at every point. This implies that the map
\begin{equation}
t\mapsto T_{z+tv}\big(Z_{\CC}(R_1^{(z+tv)})\cap Z_{\CC}(R_2^{(z+tv_1)})\big)
\end{equation}
is continuous for $t\in U$ and $z\in W\backslash\zeta$.
\end{proof}
\section{Bounding $I_6$: Cutting a variety into open regions}\label{STOnAVarietySec}
In order to bound the incidences in $I_6$, we will cut each surface $V_i(\RR)$ and each curve $\{\gamma(\RR)\colon\gamma\in\Gamma_{S,i}^{(3)}\}$ into pieces. On each piece of $V_i(\RR)$, we will have an arrangement of points and curves. In later sections, we will apply the crossing lemma to each of these arrangements to bound the number of point-curve incidences in terms of the number of curve-curve crossings, plus an error term.

\subsection{Defining some bad points on the curves}
In this section we will define various bad points on the curves in $\Gamma_{S,i}^{(3)}$. After these points are removed, the real locus of $\gamma$ will consist of a collection of simple open curves, which will be amenable to crossing lemma type arguments. First, we must deal with a small technical annoyance.
\subsubsection{Incidences occurring on the bad sets $\zeta$}\label{dealingWithBadSetsSec}
Fix a generic vector $v\in\CC^4$. For each index $i=1,\ldots,\ell$ and each irreducible component $W\subset V_i$, let $\zeta_W$ be the bad set obtained by applying Corollary \ref{perturbGaussMapLem} to $W$, using the generic vector $v$. Let
\begin{equation}\label{defnZetaI}
\zeta_i = \bigcup_{W}\zeta_W,
\end{equation}
where the union is taken over all irreducible components $W\subset V_i$, and let
\begin{equation*}
I_7=\{(p,S)\colon \gamma\subset \zeta_i\}.
\end{equation*}

By Corollary \ref{perturbGaussMapLem}, $\zeta_i$ has degree $\sum_{W} O(D_i+E_i)\deg W=O(D_iE_i^2).$ Thus if we define
\begin{equation*}
\zeta= \bigcup_i\zeta_i,
\end{equation*}
then $\deg \zeta=O(\sum_i(D_iE_i^2))$. By Lemma \ref{incidencesOnACurveLem},
\begin{equation}\label{I8BoundEqn}
|I_7|\lesssim m^{\frac{k}{2k-1}}n^{\frac{2k-2}{2k-1}}+m.
\end{equation}

The idea is that we will decompose each variety $V_i(\RR)_{\smooth}$ into a disjoint collection of pieces, each of which is homeomorphic to an open subset of $\RR^2$. With a few exceptions, incidences between points on $V_i(\RR)_{\smooth}$ and curves lying in $\Gamma_{S,i}^{(3)}$ will be counted using the crossing lemma. By $V_i(\RR)_{\smooth},$ we will mean points of $V_i(\RR)$ that are smooth in dimension 2. If $\dim_{\RR}(V_i(\RR))<2$, then by Corollary \ref{realComplexDimDontAgreeCor}, $V_i(\RR)\subset (V_i)_{\sing}$, and all incidences on $V_i$ have already been counted.

Select a generic (real) 2--plane $\Pi^\prime\in\grass(2,4;\RR)$. For each index $i$, define
\begin{equation}\label{defnOfAj}
B_i = \{z\in (V_i)_{\smooth}\colon T_z(V_i)\in A_{\Pi^{\prime}} \}.
\end{equation}

For each $\gamma\in\Gamma_{S,i}^{(3)}$, we will define various types of bad points. For $S\in\sphrs_2,$ define
\begin{equation}
 \alpha_{S,i}=\bigcup_{\gamma\in\Gamma_{S,i}^{(3)}}\gamma.
\end{equation}

Define

\begin{align}
&\Xi_{\gamma, \sing}=\gamma\cap (\alpha_{S,i})_{\sing},\label{XiGammaSing}\\
&\Xi_{\gamma, \operatorname{shared}}=\{z\in\gamma\colon z\ \textrm{is an isolated point of}\ \gamma\cap V_{i^\prime}\ \textrm{for some}\ i^\prime\neq i\},\\
&\Xi_{\gamma,\operatorname{dir}}=\{z\in\gamma_{\smooth}\colon T_z(\gamma)\cdot v_1=0\}, \label{XiDirEqn}\\
&\Xi_{\gamma,\operatorname{singPt}}=\gamma_{\smooth}\cap (V_i)_{\sing},\\
&\Xi_{\gamma,\operatorname{vertPt}}=\gamma_{\smooth}\cap B_i.\label{XiVertPtEqn}
\end{align}
In \eqref{XiDirEqn}, $v_1$ is a generic unit vector. By generic, we mean that $v_1$ and $\Pi^\prime$ are generic with respect to the collection $\sphrs_2,$ the points $\pts$, and the polynomials $\{P_j\}$ and $\{Q_j\}$.

Finally, define
\begin{equation}\label{xiBadPtDefn}
\Xi_{\gamma,\operatorname{badPt}}=\eqref{XiGammaSing}\cup\ldots\cup \eqref{XiVertPtEqn}.
\end{equation}

\subsection{Bounding the sets $\eqref{XiGammaSing},\ldots,\eqref{XiVertPtEqn}$}
We first record the following corollary of Lemma \ref{branchesOfACurveLem}
\begin{cor}
For each $S\in\sphrs_2$ and each index $i$,
\begin{align}
&\sum_{z\in\alpha_{S,i}}G_z(\alpha_{S,i})\lesssim D_i^2,\\
&\sum_{\gamma\in\Gamma_{S,i}^{(3)}} \sum_{z\in\gamma_{\sing}}G_z(\gamma)\lesssim D_i^2.\label{boundingSumGammaSingBranches}
\end{align}
\end{cor}
%
\begin{lem}\label{boundingGammaShared}
For each $\gamma\in\Gamma_{S,i}^{(3)}$,
\begin{equation}\label{sharedPtsBd}
|\Xi_{\gamma, \operatorname{shared}}|\lesssim \deg\gamma\sum_i(D_i+E_i).
\end{equation}
\end{lem}

\begin{proof}
If $x\in \Xi_{\gamma, \operatorname{shared}},$ then $x$ is an isolated intersection point of $\gamma\cap\BZ_{\CC}(P_{i^\prime})$ or $\gamma\cap\BZ_{\CC}(Q_{i^\prime})$ for some $i^\prime\neq i$. By B\'ezout's theorem (Proposition \ref{Bezout}), the number of times this can occur is bounded by the RHS of \eqref{sharedPtsBd}.
\end{proof}

\begin{lem}\label{boundingGammaDir}
For each $S\in\sphrs_2$ and each $\gamma\in\Gamma_{S,i}^{(3)},$
\begin{equation}
|\Xi_{\gamma,\operatorname{dir}}|\lesssim\deg(\gamma)^2.
\end{equation}
\end{lem}
\begin{proof}
After a rotation, we can assume that $v_1=(1,0,0,0)$. Let $\gamma^\prime=\overline{\pi(\gamma)}$, where $\pi$ is the projection onto the $(x_1,x_2)$--plane. Note that $\deg\gamma^\prime=\deg\gamma$. Since $v_1$ was chosen generically, $z\in \Xi_{\gamma,\operatorname{dir}}$ if and only if $z\in\gamma^\prime_{\smooth}$ and $T_{\pi(z)}(\pi(\gamma^\prime))\cdot \pi(v_1)=0$. Let $f_{\gamma^\prime}$ be a square-free polynomial such that $Z(f_{\gamma^\prime})=\gamma^\prime$. We have $\deg f_{\gamma^\prime}\leq\deg\gamma^\prime=\deg\gamma$. Then
\begin{equation*}
\{z\in\gamma_{\smooth}\colon T_{\pi(z)}(\pi(\gamma))\cdot \pi(v_1)=0\}\subset \gamma^\prime\cap \BZ_{\CC}( \pi(v_1)\cdot\nabla f_{\gamma^\prime}).
\end{equation*}
The latter set has cardinality $O(\deg(\gamma)^2).$
\end{proof}
\begin{lem}\label{boundOnSingPtGamma}
For each $S\in\sphrs_2$ and each index $i$,
\begin{equation}
\sum_{\gamma\in\Gamma_{S,i}^{(3)}}|\Xi_{\gamma,\operatorname{singPt}}|\lesssim D_iE_i+\sum_{\gamma\in\Gamma_{S,i}^{(3)}}(\deg\gamma)^2.\\
\end{equation}
\end{lem}
\begin{proof}
Factor $V_i$ into irreducible components $W_{i,j}$. Recall that if $\gamma\in\Gamma_{S,i}^{(3)},$ then a generic point $x\in\gamma$ lies in $(V_i)_{\smooth},$ so in particular, $\gamma$ is contained in $W_{i,j}$ for precisely one index $j$. Furthermore, we have $T_x(S^*)\neq T_x(V_i)$ at a generic point $x\in\gamma$. If $x\in \Xi_{\gamma,\operatorname{singPt}}$ has not already been placed in $\Xi_{\gamma,\operatorname{sing}}$, then $x$ is a smooth point of $\gamma$. We can now apply the argument used to bound $|I_5|$ (Section \ref{I5Sec}) to conclude that for each index $j$,

\begin{equation}\label{singPtInWij}
\sum_{\substack{\gamma\in\Gamma_{S,i}^{(3)},\\ \gamma\subset W_j}}|\Xi_{\gamma,\operatorname{singPt}}|\lesssim \deg W_{i,j}+\sum_{\substack{\gamma\in\Gamma_{S,i}^{(3)},\\ \gamma\subset W_j}}(\deg\gamma)^2.\\
\end{equation}
Summing \eqref{singPtInWij} over all irreducible components of $V_i$ yields \eqref{boundOnSingPtGamma}.
\end{proof}
\begin{lem}\label{bdVeryPtsLem}
For each $\gamma\in\Gamma_{S,i}^{(3)}$,
\begin{equation}\label{countingVertPtsEqn}
|\Xi_{\gamma,\operatorname{vertPt}}|\lesssim \deg(\gamma)E_i.\\
\end{equation}
\end{lem}
\begin{proof}
Let $W\subset V_i$ be the (unique) irreducible component of $V_i$ that contains $\gamma$. For each $z_0\in \Xi_{\gamma,\operatorname{vertPt}}$, we can select a small interval $\beta_{z_0}\subset\gamma(\RR)$ that contains $z_0$, so that the intervals $\{\beta_{z_0}\}_{z_0\in \Xi_{\gamma,\operatorname{vertPt}}}$ are disjoint.

By Lemma \ref{changeSignAcrossCurve}, we can assume (after shrinking $\beta_{z_0}$ if necessary) that for each interval $\beta_{z_0}$, we have
\begin{equation}
\Big(\det(T_{z_0^\prime}(W(\RR)),\Pi^\prime)\Big)\Big(\det(T_{z_0^{\prime\prime}}(W(\RR)),\Pi^\prime)\Big)<-\epsilon_1,
\end{equation}
where $z_0^\prime$ and $z_0^{\prime\prime}$ are the two endpoints of the curve $\beta_{z_0}$, and $\Pi^\prime$ is the 2--plane from \eqref{defnOfAj}. Here $\epsilon_1>0$ is some sufficiently small constant, depending on $W,\ \gamma,$ and $\Pi^\prime$.

By Corollary \ref{perturbGaussMapLem}, we have that if we select $\epsilon_2>0$ sufficiently small depending on $\epsilon_1$, then if we let $\tilde\beta_{z_0} = \beta_{z_0}+ \epsilon_2 v$ (here $v$ is the vector from Section \ref{dealingWithBadSetsSec}), and define $\tilde z_0^\prime=z_0^\prime+\epsilon_2v,\ \tilde z_0^{\prime\prime}=z_0^{\prime\prime}+\epsilon_2 v$, then

\begin{equation}
\Big(\det(T_{\tilde z_0^\prime}(W(\RR)),\Pi^\prime)\Big)\Big(\det(T_{\tilde z_0^{\prime\prime}}(W(\RR)),\Pi^\prime)\Big)<0.
\end{equation}

Fix vectors $v_3,v_4$ so that $\Pi^\prime=\langle v_3,v_4\rangle.$ Define the function

\begin{equation*}
\Psi(z)=   \det\left[\begin{array}{c}\nabla P_j\\ \nabla Q_j\\ v_3 \\ v_4\end{array}\right](z).
\end{equation*}

Now, if $\epsilon_2>0$ is selected generically (and still selected sufficiently small, depending on $\epsilon_1$), the curve $\gamma+\epsilon_2 v$ does not lie in $\BZ_{\CC}(\Psi)$. This means that
\begin{equation}
| (\gamma+\epsilon_2v)\cap\BZ_{\CC}(\Psi)|\leq (\deg\gamma)(\deg\Psi)\lesssim(\deg\gamma)E_i.
\end{equation}
On the other hand, Lemma \ref{intermediateValueThmLem} implies that at least one intersection point of $(\gamma+\epsilon_2v)\cap\BZ_{\CC}(\Psi)$ must occur inside every interval of the form $\tilde\beta_{z_0},\ z_0\in \Xi_{\gamma,\operatorname{vertPt}}$. This gives us the bound \eqref{countingVertPtsEqn}.
\end{proof}


Combining the previous lemmas, we obtain the following result:

\begin{prop}\label{cutNumberBddProp}
For each $S\in\sphrs_2$ and index $i,$ we have the bound
\begin{equation}\label{sumBdPtsGamma}
\sum_{\gamma\in\Gamma_{S,i}^{(3)}}|\Xi_{\gamma,\operatorname{badPt}}|\lesssim D_iE_i.
\end{equation}
\end{prop}
\begin{proof}
 First, note that
 \begin{equation*}
 \sum_{\gamma\in\Gamma_{S,i}^{(3)}}\deg\gamma\lesssim D_i.
  \end{equation*}
 Now we combine the bounds from Corollary \ref{boundingSumGammaSingBranches} and Lemmas \ref{boundingGammaShared}, \ref{boundingGammaDir}, \ref{boundOnSingPtGamma}, and \ref{bdVeryPtsLem}, to obtain \eqref{sumBdPtsGamma}.
\end{proof}

Combing the bounds from this section and using \eqref{sumDjEj}, we obtain the following bounds, which we will record as a lemma.

\begin{lem}\label{sumOverEverythingCor}
We have the bounds
\begin{align}
&\sum_i\sum_{S\in\sphrs_2}\sum_{\gamma\in\Gamma_{S,i}^{(3)}}|\Xi_{\gamma,\operatorname{badPt}}|\lesssim m^{\frac{k}{2k-1}}n^{\frac{2k-2}{2k-1}},\label{boundBadPtsEqn}\\
&\sum_i\sum_{S\in\sphrs_2}\sum_{z\in (\alpha_{S,i})_{\sing}}G_z(\alpha)\lesssim m^{\frac{k}{2k-1}}n^{\frac{2k-2}{2k-1}}.\label{boundBadPtsMultEqn}
\end{align}
\end{lem}
\subsection{Cutting the curves in $\Gamma_{S,i}^{(3)}$}
Fix a surface $S\in\sphrs_2$ and an index $i$. For each $\gamma\in\Gamma_{S,i}^{(3)}$, consider the set
\begin{equation}\label{defnOfPiecesGamma}
\begin{split}
\operatorname{Pieces}_{\gamma}= \big\{\beta\subset\RR^4\colon\beta\ \textrm{is a connected}&\ \textrm{component}\\ &\textrm{of}\ \gamma(\RR)\backslash \Xi_{\gamma,\operatorname{badPt}}\big\}.
\end{split}
\end{equation}
\begin{lem}
If $\beta\in\operatorname{Pieces}_{\gamma},$ then $\beta$ is a point or a simple open curve (homeomorphic to $(0,1)$).
\end{lem}
\begin{proof}
Suppose $\beta\in\operatorname{Pieces}_{\gamma}$ is not a point. Since $\beta$ does not contain any singular points, $\beta$ is a smooth one-dimensional manifold. Thus $\beta$ is either a simple open curve or is homeomorphic to a circle. However, if $\beta$ is homeomorphic to a circle, then it must contain a point $z\in\beta$ where $T_z\cdot v_1=0$, where $v_1$ is the vector from \eqref{XiDirEqn}. Since we removed all points of this form, no curve $\beta\in\operatorname{Pieces}_{\gamma}$ may be homeomorphic to a circle.
\end{proof}

By Corollary \ref{surfacesEntering4DCellCor},
\begin{equation}\label{numberComponentsCount}
\begin{split}
\sum_{S\in\sphrs_2}\sum_{i}\sum_{\gamma\in\Gamma_{S,i}^{(3)}}b_0(\gamma(\RR))&\lesssim n \sum_i D_i^2\\
&\lesssim m^{\frac{k}{2k-1}}n^{\frac{2k-2}{2k-1}},
\end{split}
\end{equation}
where $b_0(\gamma(\RR))$ is the number of Euclidean connected components of $\gamma(\RR)$. We need to bound the size of $\operatorname{Pieces}_{\gamma}$. By Lemma \ref{removePtIncreaseConnectedComponents}, each time we remove a point $z\in \Xi_{\gamma, \sing}$, we increase the number of connected components by at most $2G_z(\gamma)$. Each time we remove a point, we are left with a (new) semialgebraic set (indeed, this is just the previous semialgebraic set with one point removed). Thus we can apply the lemma iteratively, removing one point from $\Xi_{\gamma, \sing}$ at a time. By \eqref{boundingSumGammaSingBranches}, removing all the points $\Xi_{\gamma, \sing}$ increases the number of connected components in \eqref{numberComponentsCount} by at most $O\big(m^{\frac{k}{2k-1}}n^{\frac{2k-2}{2k-1}}\big)$.

If we remove a point $z\notin \Xi_{\gamma, \sing}$ from a curve $\gamma(\RR)$, we increase the number of connected components by at most one. Thus if we remove all the points from $\Xi_{\gamma,\operatorname{badPt}}\backslash \Xi_{\gamma, \sing}$ as $\gamma$ ranges over all curves in $\bigcup_{S\in\sphrs_2}\bigcup_i \Gamma_{S,i}^{(3)}$, we increase the number of connected components in \eqref{numberComponentsCount} by at most $O\big(m^{\frac{k}{2k-1}}n^{\frac{2k-2}{2k-1}}\big)$.

We conclude that
\begin{equation}
\sum_{S\in\sphrs_2}\sum_j \sum_{\gamma\in\Gamma_{S,j}}|\operatorname{Pieces}_{\gamma}| \lesssim m^{\frac{k}{2k-1}}n^{\frac{2k-2}{2k-1}}.
\end{equation}
\subsection{Cutting the surfaces $V_i(\RR)$}
For each index $i$, let
\begin{equation}\label{defnOfCalY}
\begin{split}
\mathcal{Y}_i=\big\{A\subset \RR^4\colon A\ \textrm{is a} &\ \textrm{connected component} \\ & \textrm{of}\ V_i(\RR)\backslash\big((V_i)_{\sing}\cup B_i \cup \zeta_i\big) \big\},
\end{split}
\end{equation}
where $B_i$ is the set from \eqref{defnOfAj}, and $\zeta_i$ is the set from \eqref{defnZetaI}.

\begin{lem}
The sets $A\in \mathcal{Y}_i$ are two-dimensional smooth manifolds.
\end{lem}
\begin{proof}
By Corollary \ref{realComplexDimDontAgreeCor}, if $z\in V_i(\RR)\backslash (V_i)_{\sing},$ then $\dim_{\RR,z}(V_i(\RR))=2$. Since $z\in (V_i)_{\smooth}$, this also implies that $z$ is a smooth point of $V_i(\RR)$ (in dimension 2). Thus $V_i(\RR)\backslash(V_i)_{\sing}$ is a two-dimensional smooth manifold. Since $B_i(\RR) \cup \zeta_i(\RR)$ are algebraic curves (possibly with 0 dimensional components), $V_i(\RR)\backslash\big((V_i)_{\sing}\cup B_i \cup \zeta_i\big)$ is also a two-dimensional smooth manifold, and thus so are its connected components.
\end{proof}

\begin{lem}
Let $A\in \mathcal{Y}_i$, and let $\pi\colon\RR^4\to\RR^2$ be the projection in the direction $\Pi^\prime$ (i.e., the direction that maps $\Pi^\prime$ to the vector space 0). Here $\Pi^\prime$ is the (real) 2--plane from \eqref{defnOfAj}. Then the restriction of $\pi$ to $A$ is a diffeomorphism, and $\pi(A)$ is an open subset of $\RR^2$.
\end{lem}
\begin{proof}
The main thing to show is that $\pi$ is injective. Suppose there exists two points $x,x^\prime\in A$ such that $\pi(x)=\pi(x^\prime)$. Let $\eta\subset A$ be a smooth curve connecting $x$ and $x^\prime$, and let $\eta(t)$ be the parametrization of this curve by arclength, normalized so that $\eta(0)=x$ and $\eta(1)=x^\prime$. For each $t\in[0,1],$ let $r(t) = \operatorname{dist}(\eta(t),x+\Pi^\prime)^2$, where $x+\Pi^\prime$ is the affine 2--plane obtained by translating $\Pi$ by the (vector) $x\in\RR^4$, and $\operatorname{dist}(\eta(t),x+\Pi^\prime)$ is the (Euclidean) distance between the point $\eta(t)$ and the set $x+\Pi^\prime$.

$r(t)$ is smooth and $r(0)=r(1)=0$. Thus there exists some $t_0\in(0,1)$ so that $\frac{d}{dt}r(t)|_{t=t_0}=0$. This implies that the curve $\eta$ has tangent vector $w\in \Pi^\prime$ at the point $r(t_0)$. However, at every point $z\in A$ we have that $T_z(A)\cap\Pi^\prime = 0$. This is a contradiction. Thus $\pi$ is injective.

We can see that the map $\pi$ is a local diffeomorphism whenever the Jacobian matrix of $\pi$ has full rank. However, this occurs precisely at points $x\in A$ with $T_z(A)\cap\Pi^\prime = 0$. By the definition of $A$, this happens at every point.

Since $\pi$ is injective and is everywhere a local diffeomorphism, we conclude that $\pi$ is a diffeomorphism.
\end{proof}
\subsection{Combining $\mathcal{Y}_i$ and $\operatorname{Pieces}_{\gamma}$.}

\begin{lem}
If $\gamma\in \Gamma_{S,i}^{(3)}$, and if $\beta\in \operatorname{Pieces}_{\gamma},$ then $\beta$ is entirely contained in a single set $A\subset\mathcal{Y}_i$, and this set $A$ is unique.
\end{lem}
\begin{proof}
Recall that every set $\beta\in \operatorname{Pieces}_{\gamma}$ is connected (in the Euclidean topology), and each set $\beta$ is contained in some set $V_i$. Thus if $\beta$ meets two sets $A,A^\prime\in\mathcal{Y}_i$, then by \eqref{defnOfCalY}, $\beta$ must intersect a point from $(V_i)_{\sing}\cup B_i \cup \zeta_i$. However, every point from $\beta\cap (V_i)_{\sing}\cup B_i \cup \zeta_i$ also lies in $\Xi_{\gamma,\operatorname{badPt}}$, (where $\gamma$ is the algebraic curve associated to $\beta$). By definition, $\beta$ contains no points from this set.
\end{proof}

\begin{defn}\label{shrinkDef}
If $\beta\in \operatorname{Pieces}_{\gamma}$, define $\operatorname{shrink}(\beta)$ to be the curve obtained by \emph{shrinking} $\beta$ by a small amount. More precisely, since $\beta$ is a simple open curve, there is a homeomorphism $\eta\colon(0,1)\to\beta$. Define $\operatorname{shrink}(\beta)=\iota((\epsilon,1-\epsilon))$, where $\epsilon>0$ is a very small quantity. Specifically, we choose $\epsilon$ so that the following two properties hold:
\begin{itemize}
\item If $p\in\pts$ and $p$ is an interior point of $\beta$, then $p$ is an interior point of $\operatorname{shrink}(\beta)$.
\item If two curves $\beta,\beta^\prime$ are disjoint, then $\operatorname{shrink}(\beta)$ and $\operatorname{shrink}(\beta^\prime)$ have disjoint (Euclidean) closures.
\end{itemize}
\end{defn}

For each $i=1,\ldots,\ell$, and for each $A\in \mathcal{Y}_i$, define
\begin{equation}
\mathcal L_A = \bigcup_{S\in\sphrs_2}\bigcup_{\gamma\in\Gamma_{S,i}^{(3)}}\{\operatorname{shrink}(\beta)\colon \beta\in \operatorname{Pieces}_{\gamma},\ \textrm{and}\ \beta\subset A\},
\end{equation}

and define

\begin{equation}
\pts_A = \pts_i \cap A,
\end{equation}
where $V_i$ is the (unique) variety such that $V_i(\RR)$ contains $A$.

The sets $\{\pts_A\}$ are disjoint as $A$ ranges over the sets in $\mathcal{Y}_i$ and as $i=1,\ldots,\ell$. Furthermore, if $(p,S)\in I_7\backslash I_8$, then $p$ lies in some set $\pts_A$.

To bound the number of incidences in $I_7\backslash I_8$, we will need to use the crossing lemma. If $A\in\mathcal{Y}_i$, define
\begin{equation*}
\operatorname{crossings}(A)=\sum_{\substack{\beta,\beta^\prime\in\mathcal{L}_A\\ \beta\neq\beta^\prime}}|\beta\cap\beta^\prime|.
\end{equation*}
\begin{lem}[Bounding the number of crossings]\label{boundingCrossingNumLem}
\begin{equation}
\sum_{i=1}^\ell\sum_{A\in\mathcal{Y}_j}\operatorname{crossings}(A) \leq C_0n^2,
\end{equation}
where $C_0$ is the constant from Property \ref{noCommonComponents} of Definition \ref{goodSurfacesDefn}.
\end{lem}
\begin{proof}
First, note that $\sum_{S\neq S^\prime}|S\cap S^\prime|\leq C_0n^2$. The entire point is to show that if $z\in S\cap S^\prime$, then there is at most one pair $\beta,\beta^\prime$ with $\beta\subset S,\beta^\prime\subset S^\prime$ so that $z\in\beta\cap \beta^\prime$. We also need to observe that if both $\beta$ and $\beta^\prime$ are contained in the same surface $S$, then $\beta\cap\beta^\prime=\emptyset$. This is because any point $z\in\beta\cap\beta^\prime$ is a singular point of $S^*\cap V_i$ for some index $i$, so this point lies in $\Xi_{\gamma,\operatorname{badPt}}\cap\Xi_{\gamma^\prime,\operatorname{badPt}}$, where $\gamma,\gamma^\prime$ are the (not necessarily distinct) curves associated to $\beta$ and $\beta^\prime$, respectively. Thus points of this form were removed at a previous step.

For contradiction, suppose there existed some indices $i_1,i_2$, some $A_1\in\mathcal{Y}_{i_1},\ A_2\in\mathcal{Y}_{i_2}$, and some curve segments $\beta_1,\beta_1^\prime\in \lines_{A_1},\beta_2,\beta_2^\prime\in\lines_{A_2}$ so that $\beta_1,\beta_2\subset S$, $\beta_1^\prime,\beta_2^\prime\subset S^\prime$, and $(\beta_1\cap\beta_1^\prime) \cap (\beta_2\cap\beta_2^\prime)\neq\emptyset$.

First, we must have $i_1\neq i_2$. Indeed, if $i_1=i_2=i$, then $\beta_1\cap\beta_2$ is a singular point of $S^*\cap V_i,$ and by \eqref{defnOfPiecesGamma}, neither $\beta_1$ nor $\beta_2$ can contain any points of this type. Next, we can assume that $\beta\not\subset \big(\bigcup_{i}(V_i)_{\sing}\big),$ since all irreducible components of $S^*\cap \bigcup_i V_i$ that were contained in $\big(\bigcup_{i}(V_i)_{\sing}\big)$ were already removed. In particular, since $i_1\neq i_2$, we must have that $\beta_1\cap V_{i_2}$ is a discrete set, where $\gamma_1$ is the curve associated to $\beta_1$. But every point in this intersection was already removed when we removed the set $\Xi_{\gamma, \operatorname{shared}}$. Thus no points of this type may exist in any curve segment $\beta$.
\end{proof}

\begin{lem}
Fix an $A\in\mathcal{Y}_i$, and let $p_1,\ldots,p_k\in A$. Then at most $C_0$ curves $\beta\in\mathcal{L}_A$ can contain the points $p_1,\ldots,p_k$, where $C_0$ is the quantity from the statement of Theorem \ref{surfacesInR4Thm}.
\end{lem}
\begin{proof}
First, if two curves $\beta,\beta^\prime\in\mathcal{L}_A$ both contain $p_1$, then  $\beta,\beta^\prime$ must come from distinct surfaces $S,S^\prime$. Otherwise $p_1$ would lie in the sets $\Xi_{\gamma,\operatorname{sing}}, \Xi_{\gamma^\prime,\operatorname{sing}}$ (where $\gamma,\gamma^\prime$ are the curves associated to $\beta,\beta^\prime$, respectively). However, $\beta$ cannot contain any point from $\Xi_{\gamma,\operatorname{sing}}$, and similarly for $\beta^\prime$.

Since every curve that contains the points $p_1,\ldots,p_k$ must come from a distinct surface $S\in\sphrs_2$, by Property \ref{tangentSpacesTransverse} from Definition \ref{goodSurfacesDefn}, at most $C_0$ curves can contain the points $p_1,\ldots,p_k$.
\end{proof}

We must now develop the tools needed to apply the crossing lemma to the collections of curves and points on the open regions $A\in\mathcal{Y}_i$.

%
%
%
%
\section{The final interlude: Some graph theory}\label{graphTheorySection}
In \cite{Szekely}, Sz\'ekely provided a simple proof of the Szemer\'edi-Trotter theorem using the crossing lemma from topological graph theory. In brief, the crossing lemma states that a graph drawing either contains very few edges, or it must have many crossings (points where two edges cross). Sz\'ekely showed how a point-line arrangement could be converted into a graph drawing, where the number of point-line incidences was bounded by the number of edges. On the other hand, since every two lines cross at most once, Sz\'ekely was able to bound the number of crossings in the graph drawing. This led to a bound on the number of incidences.

We wish to do something similar, but in our case we do not have a single graph but many, and the crossings are spread out amongst all of the graphs. We need to obtain an incidence bound across all of these graphs. This will be done in Lemma \ref{planarIncidenceLem}.

\subsection{Graphs and graph drawings}
\begin{defn}
We define a \emph{generalized undirected graph drawing} to be a triple $H=(\pts,\Gamma,E)$. Here $\pts\subset\RR^2$ is a finite collection of points (also called vertices); $\Gamma$ is a finite set of bounded simple open curves, with $|\gamma\cap\gamma^\prime|$ finite for every pair of distinct curves $\gamma,\gamma^\prime$; and $E$ is a set of pairs of the form $(\{p,q\},\gamma)$, where $p,q$ are distinct points in $\pts$ and $\gamma\in\Gamma$. If $p,q$ are vertices of a graph drawing $H$, we define
\begin{equation*}
\edgemult(p,q)=|\{\gamma\in\Gamma\colon (\{p,q\},\gamma)\in E\}|.
\end{equation*}
Informally, this is the number of edges between $p$ and $q$.
\end{defn}

\begin{defn}
We say that the undirected drawing $H$ is \emph{proper} if the following properties hold
\begin{itemize}
\item No point of $\pts$ lies in the relative interior of any curve in $\Gamma$.
\item $(\{p,q\},\gamma)\in E$ if and only if the endpoints of $\gamma$ are the points $p$ and $q$.
\end{itemize}
Thus a proper undirected graph drawing is a special type of generalized undirected graph drawing.
\end{defn}
\begin{defn}
Let $G=(V^\prime,E^\prime)$ be an undirected multigraph. Thus $V^\prime$ is a set of vertices and $E^\prime$ is a multiset of pairs of distinct vertices from $V^\prime$. Let $H=(\pts,\Gamma,E)$ be a (generalized) undirected graph drawing. We say that $G$ is associated to $H$ (or $H$ is associated to $G$) if there is a bijection from $\pts$ to $V^\prime$ so that for every pair of vertices $p,q\in \pts$, $\edgemult(p,q)$ is equal to the number of edges between $p$ and $q$ in $G$. Given a graph drawing $H$, there is always a unique multigraph $G$ associated to $H$.
\end{defn}

\begin{defn}
We define a \emph{generalized directed graph drawing} to be a triple $H=(\pts,\Gamma, E).$ Here $\pts\subset\RR^2$ is a finite collection of points (also called vertices); $\Gamma$ is a finite set of bounded simple open curves, with $|\gamma\cap\gamma^\prime|$ finite for every pair of distinct curves $\gamma,\gamma^\prime$; and $E$ is a set of pairs of the form $(p,q,\gamma)$, where $p,q$ are distinct points in $\pts$ and $\gamma\in\Gamma$. If a triple $(p,q,\gamma)$ is in $E$, we say that $p\stackrel{\gamma}{\to} q$, i.e.~there is a directed edge from $p$ to $q$ along $\gamma$ (note that $p$ and $q$ need not be the endpoints of $\gamma$). The collection of all directed edges from $p$ to $q$ is denoted $p\to q$, and the number of edges is denoted by $\edgemult(p\to q)$.
\end{defn}

\begin{defn}\label{defnOfProperGraphDrawing}
We say that a directed graph drawing is \emph{proper} if the following properties hold:
\begin{itemize}
\item No point of $\pts$ lies in the relative interior of any curve in $\Gamma$.
\item If $p\stackrel{\gamma}{\to} q$ then $p$ and $q$ are the endpoints of $\gamma$. Conversely, if $p,q$ are the endpoints of $\gamma$, then precisely one of $p\stackrel{\gamma}{\to} q$ or $q\stackrel{\gamma}{\to} p$ must hold.
\end{itemize}
The intuition is that generalized graph drawings are allowed to have multiple edges stacked on top of each other, while proper graph drawings do not permit this.
\end{defn}

\subsection{Crossings and graph drawings}
\begin{defn}
If $H$ is a generalized (directed or undirected) graph drawing, we define the \emph{number of crossings} in $H$,
\begin{equation*}
\crossing(H)=\sum_{\substack{\gamma,\gamma^\prime\in\Gamma\\ \gamma\neq\gamma^\prime}}|\gamma\cap\gamma^\prime|.
\end{equation*}
Since the intersection of any two curves is a discrete set, $\crossing(H)$ is finite.
\end{defn}

Let $G$ be an undirected multigraph. We define $\vertices(G)$ to be the number of vertices of $G$ and $\edges(G)$ to be the number of edges. 

\begin{thm}[Ajtai, Chvatal, Newborn, Szemer\'edi \cite{Ajtai}; Leighton \cite{Leighton};  Sz\'ek\-ely \cite{Szekely}] \label{crossingNumberIneq}
Let $H$ be a proper undirected graph drawing and let $G$ be the multigraph associated to $H$. Suppose $G$ has maximum edge multiplicity $M$. If $\edges(G)\geq5\vertices(G)$, then
 \begin{equation}\label{crossingNumberIneqEqn}
 \crossing(H)\geq\frac{\edges(G)^3}{100M \vertices(G)^2}.
\end{equation}
 \end{thm}

\subsection{Bounding incidences by crossings}
\begin{thm}\label{planarIncidenceLem}
Let $U\subset\RR^2$ be open. Let $\pts\subset\RR^2$ be a set of points, and let $\Gamma\subset\RR^2$ be a set of simple open curves with $k$ degrees of freedom (relative to $\pts$), i.e. for any $k$ points of $\pts$, there are at most $C_0$ curves from $\Gamma$ that contain all $k$ points, and any two curves intersect in at most $C_0$ points. Then
\begin{equation}
\incidences(\pts,\Gamma) \lesssim |\pts|^{\frac{k}{2k-1}} \Big(\sum_{\substack{\gamma,\gamma^\prime\in\Gamma\\\gamma\neq\gamma^\prime}} |\gamma\cap\gamma^\prime| \Big)^{\frac{k-1}{2k-1}} + |\pts| + |\Gamma|.
\end{equation}
The implicit constant depends only on $k$ and $C_0$.
\end{thm}

The proof of Theorem \ref{planarIncidenceLem} is much easier for the $k=2$ case (it is a variant of Sz\'ekely's proof in \cite{Szekely}), so will provide a proof of this case first. The proof for general $k$ also works for $k=2$.
\begin{proof}[Proof of Lemma \ref{planarIncidenceLem}, $k=2$ case]
Replace each curve $\gamma\in \Gamma$ with a slightly shrunk curve $\gamma^\prime$ (in the sense of Definition \ref{shrinkDef}), so that $\bdry(\gamma^\prime)$ does not meet any point from $\pts$ nor any curve from $\Gamma$. If $\Gamma^\prime$ denotes the set of shrunk curves, then $|\incidences(\pts,\Gamma^\prime)|\geq |\incidences(\pts,\Gamma)|-2|\Gamma|.$ Delete from $\Gamma^\prime$ those curves that are incident to fewer than 2 points from $\pts$, and denote the resulting set of curves $\Gamma^{\prime\prime}$. Then  $|\incidences(\pts,\Gamma^{\prime\prime})|\geq |\incidences(\pts,\Gamma)|-4|\Gamma|.$

Let $G$ be the undirected multigraph whose vertex set is $\pts$, and where two vertices are connected by an edge if the two corresponding points are joined by a curve from $\Gamma^{\prime\prime}$, and the two vertices are consecutive on this curve. Let $H$ be the (proper, undirected) drawing of $G$ given by the points $\pts\subset\RR^2$ and the curve segments joining consecutive edges from curves $\gamma\in\Gamma^{\prime\prime}$.

The multigraph $G$ need not be a graph, since two vertices can be connected by several edges. However, the maximum edge multiplicity of $G$ is bounded by the constant $C_0$ from the statement of Lemma \ref{planarIncidenceLem}. Furthermore, $\edges(G)\geq\frac{1}{2}|\incidences(\pts,\Gamma^{\prime\prime})|$, so by Theorem \ref{crossingNumberIneq},
\begin{equation*}
\begin{split}
|\incidences(\pts,\Gamma)|&\lesssim |\pts|^{2/3}\mathcal{C}(H)^{1/3}+|\pts|+ 4|\Gamma|\\
&\lesssim |\pts|^{2/3} \Big(\sum_{\substack{\gamma,\gamma^\prime\in\Gamma\\\gamma\neq\gamma^\prime}} |\gamma\cap\gamma^\prime| \Big)^{1/3} + |\pts| + |\Gamma|.\qedhere
\end{split}
\end{equation*}
\end{proof}
We will now prove Lemma \ref{planarIncidenceLem} for general $k$. The proof is very similar to Pach and Sharir's proof in \cite{Pach} of a Szemer\'edi-Trotter type theorem for curves with $k$ degrees of freedom. However, the main term in Pach and Sharir's bound is $|\pts|^{\frac{k}{2k-1}}|\Gamma|^{\frac{2k-2}{2k-1}}$ rather than
\begin{equation*}
|\pts|^{\frac{k}{2k-1}}\Big(\sum|\gamma\cap\gamma^\prime|\Big)^{\frac{k-1}{2k-1}},
\end{equation*}
and the former could potentially be much larger. This fact forces us to modify Pach and Sharir's proof.

\begin{proof}[Proof of Lemma \ref{planarIncidenceLem}, general case]

First, either
\begin{equation}\label{GammaControlsNumIncidences}
\incidences(\pts,\Gamma)\leq 100k |\Gamma|,
\end{equation}
or
\begin{equation}\label{incidencesBiggerN}
\incidences(\pts,\Gamma)>100k |\Gamma|.
\end{equation}
If \eqref{GammaControlsNumIncidences} holds, then the theorem follows immediately. Thus for the remainder of the proof we will assume that \eqref{incidencesBiggerN} holds

Let $\Gamma^\prime\subset\Gamma$ be the set of curves that are incident to $\geq 2k$ points from $\pts$. By \eqref{incidencesBiggerN},
\begin{equation*}
\incidences(\pts,\Gamma^\prime)>\frac{1}{2}\incidences(\pts,\Gamma),
\end{equation*}
so it suffices to consider curves in $\Gamma^\prime$. If $p\in \pts,$ let $d_p=|\{\gamma\in \Gamma^\prime\colon p\in\gamma\}|.$ We will call this the \emph{degree} of $p$. Let
\begin{equation*}
\pts^\prime=\Big\{p\in \pts\colon d_p\geq \frac{\mathcal{I}(\pts,\Gamma^\prime)}{2|\pts|}\Big\}.
\end{equation*}
Then
\begin{equation}
\incidences(\pts^\prime,\Gamma^\prime)\geq\frac{1}{2}\incidences(\pts,\Gamma^\prime)\geq\frac{1}{4}\incidences(\pts,\Gamma).
\end{equation}
For $p\in \pts^\prime,$ $\gamma\in \Gamma^\prime,$ and $p\in\gamma$, let
\begin{equation*}
S_{p,\gamma}=\{q\in \pts^\prime\colon q\in\gamma,\ q\neq p,\ d_q\geq d_p\}.
\end{equation*}

Let $H=(\pts^\prime,\Gamma^\prime,E)$ be a generalized directed graph drawing, where the triple $(p,q,\gamma)$ is in $E$ if the following conditions hold:
\begin{itemize}
\item $p\in\pts^\prime,\gamma\in\Gamma^\prime,\ p\in\gamma,\ q\in S_{p,\gamma}.$
\item $|S_{p,\gamma}|\geq k.$
\item $q$ is one of the $k$ closest points to $p$ of the point set $S_{p,\gamma}$ (i.e.~the curve segment of $\gamma$ connecting $p$ and $q$ passes through at most $k-1$ points from $S_{p,\gamma}$).
\end{itemize}

Note that $H$ might not be a proper directed graph drawing since several edges may be drawn over the same curve segment.

\begin{lem}\label{85Lem}
\begin{equation}\label{edgesCompIncidencesMinusLines}
\edges(H)\geq \incidences(\pts^\prime,\Gamma^\prime)-k|\Gamma^\prime|.
\end{equation}
\end{lem}
\begin{proof}
Let $p\in\pts^\prime,\gamma\in\Gamma^\prime$ with $p\in\gamma$. Then either there is an edge $p\overset{\gamma}{\to} q$ for some $q\in S_{p,\gamma}$, or $p\in X_\gamma,$ where $X_\gamma$ is the set of the $k$ highest degree points on $\gamma$. The lemma now follows from the observation that $|X_\gamma|\leq k$.
\end{proof}
Lemma \ref{85Lem} and \eqref{incidencesBiggerN} imply that
\begin{equation}\label{edgesHSimIncidences}
\edges(H)\geq \frac{1}{2}\incidences(\pts^\prime,\Gamma^\prime)\geq\frac{1}{8}\incidences(\pts,\Gamma).
\end{equation}

Note that a given segment of a curve $\gamma\in \Gamma$ may be part of several distinct edges, i.e.~our graph drawing $H$ may not be proper (in the sense of Definition \ref{defnOfProperGraphDrawing}). However, the following lemma controls the extent to which this occurs.
\begin{lem}\label{notTooMuchStacking}
Let $\gamma\in \Gamma^\prime,$ and let $x$ be a point on $\gamma$. Then the number of pairs $(p,q)\in(\pts^\prime)^2$ such that the arc $p\stackrel{\gamma}{\to} q$ contains $x$ is at most $10k^2$.
\end{lem}
\begin{proof}
Since $\gamma$ is a simple open curve, $\gamma\backslash x$ consists of two connected pieces, which we will call the right and left pieces. This establishes a global notion of right and left on the curve $\gamma$. We will now prove the lemma. Suppose there were more than $10k^2$ pairs $(p,q)\in(\pts^\prime)^2$ with $x$ contained in the arc $p\stackrel{\gamma}{\to} q$. Then without loss of generality, there are more than $5k^2$ arcs of the form $p \overset{\gamma}{\to}q$ where $p$ is right of $x$ and $q$ is left of $x$. Since each point $p\in\pts^\prime \cap\gamma$ has at most $k$ curves of the form $p \overset{\gamma}{\to}q$ that exit it, there exists a set of $\geq 5k$ distinct points to the right of $x$, so that each of these points contains at least one arc of the form $p \overset{\gamma}{\to}q$ that contains $x$. Denote this set of points by $\pts_1$. Let $\pts_2\subset \pts_1$ be the $2k$ right-most points from this collection, and let $p^*\in \pts_2$ be the point with lowest degree. Then the arc from $p^*$ to
$x$ passes over at least $3k$ points of $\pts^\prime$, but there are at least $2k-1>k$ points of $\pts^\prime$ on the arc $\gamma$ with distance $\leq 2k$. Each of these points lies in $S_{p^*,\gamma}$. This is a contradiction, since by definition $p^*$ is connected to the $k$ \emph{closest} points on $\gamma$ with degree $\geq d_{p^*}$.
\end{proof}

Let $H^\prime$ be the generalized directed graph drawing obtained by starting with $H$ and deleting all edges of the form $p\overset{\gamma}{\to}q$ where $p\in X_\gamma$ (recall from above that $X_\gamma$ is the set of $k$ points on $\gamma$ that have the highest multiplicity). Then since every curve in $\Gamma^\prime$ is incident to at least $2k$ edges, $\mathcal{E}(H^\prime)\geq \frac{1}{2}\mathcal{E}(H)\geq\frac{1}{16}\incidences(\pts,\Gamma)$.

Now, let $H^{\prime\prime}$ be the generalized directed graph drawing obtained by starting with $H^\prime$ and deleting all edges of the form $p\overset{\gamma}{\to}q$ whenever
\begin{equation}\label{edgeMultThreshold}
\edgemult(p\to q)>Ad_p^{\frac{k-2}{k-1}}.
\end{equation}
Here $A$ is a large constant (depending only on $k$) to be determined later. We will call $H^{\prime\prime}$ the \emph{pruned} version of $H^\prime$. If an edge $p\overset{\gamma}{\to}q$ is present in $H^\prime$ but not in $H^{\prime\prime}$, we will say the edge $p\overset{\gamma}{\to}q$ has been \emph{pruned}.
\begin{lem}[Pach-Sharir]\label{edgesHandHPrimeComparable}
\begin{equation}\label{edgesHandHPrimeComparableEqn}
\edges(H^{\prime\prime})\geq \frac{1}{2k}\edges(H^\prime)\geq\frac{1}{32k}\incidences(\pts,\Gamma).
\end{equation}
\end{lem}
\begin{proof}
The proof of this lemma is nearly identical to the arguments of Pach and Sharir in \cite[p124]{Pach}. For the reader's convenience, we reproduce it here. For $p,q\in \pts^\prime,$ let $E_p(q)$ be the set of all edges of $H^\prime$ that connect $p$ to $q$, i.e.~all edges of the form $p\overset{\gamma}{\to}q,$ for $\gamma\in \Gamma^\prime$. By the definition of $H$, we have $\sum_{q\in\pts^\prime}|E_p(q)|\leq 2(k-1)d_p$.

Let $E_{p,q}$ be the set of edges of the form $p\overset{\gamma}{\to}r$, where $\gamma$ is a curve for which $p\overset{\gamma}{\to}q$ is an edge of $H^\prime$.

Let
\begin{equation*}
R_p = \{q \in \pts^\prime \colon |E_p(q)|>Ad_p^{\frac{k-2}{k-1}}\},
\end{equation*}
so
\begin{equation*}
|R_p|\leq 2kd_p\Big(Ad_p^{\frac{k-2}{k-1}}\Big)^{-1}\leq 2kA^{-1}d_p^{\frac{1}{k-1}}.
\end{equation*}
If $R_p=\emptyset,$ there is nothing to prove. Otherwise, consider in turn each vertex $q\in R_p$ and each curve $\gamma$ that contains an edge $p\overset{\gamma}{\to}q$ from $E_p(q).$ By the definition of $H^{\prime}$, $\gamma$ must contain at least $k-1$ edges that lie in the set $E_{p,q}.$ We want to \emph{charge} $p\overset{\gamma}{\to}q$ to one of these edges; we can do this as long as one of these edges is still present in the set $H^{\prime\prime}$ (i.e.~we can do this as long as one of these edges has not been pruned).

We say that $p\overset{\gamma}{\to}q$ is \emph{good} if there exists at least one edge from $E_{p,q}$ in the generalized directed graph drawing $H^{\prime\prime}$ (i.e.~if at least one edge from $E_{p,q}$ survives the pruning process). If $p\overset{\gamma}{\to}q$ is not good, we say it is \emph{bad}.

If $p\overset{\gamma}{\to}q$ is bad, then the curve $\gamma$ passes through $p$ and through at least $k-1$ distinct points of $R_p$, and in this case $\gamma$ contains at most $2(k-1)$ bad edges. But, there are $\leq C_0$ curves passing through $p$ and any fixed set of $k-1$ points of $R_p$. Thus the number of bad edges is at most
\begin{equation}\label{boundOnNumEdges}
\begin{split}
2(k-1)C_0\binom{|R_p|}{k-1} & < \frac{2C_0(k-1)|R_p|^{k-1}}{(k-1)!}\\
& < \frac{2C_0}{(k-1)!}\big(\frac{2k}{A}\big)^{k-1}d_p.
\end{split}
\end{equation}
If we select $A$ sufficiently large, then
\begin{equation*}
\frac{2C_0}{(k-1)!}\big(\frac{2k}{A}\big)^{k-1}d_p < \frac{1}{2}(k-1)d_p,
\end{equation*}
and thus more than half of the edges in $E_p$ are good, and each of them can charge one of the surviving edges in $H^{\prime\prime}$. This implies that at least $\frac{1}{2k}d_p$ of the edges exiting $p$ survive in $H^{\prime\prime}$. Since this holds true for all edges in $H^{\prime\prime},$ Lemma \ref{edgesHandHPrimeComparable} follows.
\end{proof}
For each triple $(p,q,\gamma)\in E$ in the pruned graph drawing $H^{\prime\prime}$, let $\gamma_{p,q}\subset\gamma$ be the simple open curve connecting $p$ to $q$. Define $\Gamma_0=\{\gamma_{p,q}\colon (p,q,\gamma)\in E\}$. Let $\Gamma_1$ be obtained by perturbing each curve in $\gamma_0$ slightly so that the endpoints remain unchanged, but every two curves in $\Gamma_0$ intersect in a finite set. Let $H^{\prime\prime\prime}=(\pts^\prime, \Gamma_1,E_0)$ be the directed graph drawing where $(p,q,\gamma)\in E_0$ if and only if $p$ and $q$ are the endpoints of $\gamma$. Then $H^{\prime\prime\prime}$ is a proper directed graph drawing (in the sense of Definition \ref{defnOfProperGraphDrawing}). For every pair of distinct points $p,q\in\pts^\prime$, $\edgemult(p\to q)$ in $H^{\prime\prime}$ is equal to $\edgemult(p\to q)$ in $H^{\prime\prime\prime}$. Furthermore, by Lemma \ref{notTooMuchStacking},
\begin{equation*}
\cross(H^{\prime\prime\prime})<100k^4\cross(H^{\prime\prime}).
\end{equation*}

We will now perform a diadic decomposition of vertices in the graph $H^{\prime\prime\prime}.$ For $j = 0,\ldots,\lceil \log_2 m\rceil ,$ let $H_j$ be the proper undirected graph drawing with vertex set $\big\{p \in \pts^\prime\colon d_p\geq 2^j\frac{\mathcal{I}(\pts,\Gamma)}{2m}\big\}.$ If $p\in \pts^\prime$ and $2^j\frac{\mathcal{I}(\pts,\Gamma)}{2m}\leq d_p < 2^{j+1}\frac{\mathcal{I}(\pts,\Gamma)}{2m}$, then all of the multi-edges $p\to q$ from $H^{\prime\prime\prime}$ are added $H_j,$ but we add them as undirected edges. These are the only edges of $H_j$. Let $m_j$ be the number of vertices of $H_j$. Since $2^j \frac{\mathcal{I}(\pts,\Gamma)}{2m} m_j \leq \mathcal{I}(\pts,\Gamma)$, we have
\begin{equation}
m_j \leq 2^{-j+1}m.
\end{equation}

We have:
\begin{itemize}
\item Each multi-edge of $H_j$ has edge multiplicity $\leq \big(2^{j+1}\frac{\mathcal{I}(\pts,\Gamma)}{2m}\big)^{\frac{k-2}{k-1}}$.
\item Each multi-edge $p\to q$ in $H^{\prime\prime\prime}$ appears as a multi-edge in some $H_j$, so
\begin{equation}\label{sumJEdgesJ}
\edges(H^{\prime\prime\prime})\leq\sum_j \edges(H_j).
\end{equation}
\item $\cross(H_j)\leq\cross(H^{\prime\prime\prime})\leq100k^4\cross(H).$
\end{itemize}
Let $G_j$ be the undirected multigraph associated to $H_j$. Let
\begin{equation*}
J_1 = \bigg\{j\in\{0,\ldots,\lceil \log_2 m\rceil\}\colon\ \edges(G_j)\leq 100 m_j \Big(2^{j+1}\frac{\mathcal{I}(\pts,\Gamma)}{2m}\Big)^{\frac{k-2}{k-1}}\bigg\},
\end{equation*}
and let $J_2=\{0,\ldots,\lceil \log_2 m\rceil\}\backslash J_1$. By \eqref{edgesHandHPrimeComparableEqn} and \eqref{sumJEdgesJ}, either
\begin{equation}\label{firstOption}
 \incidences(\pts^\prime,\Gamma)\leq\frac{1}{64k}\sum_{j\in J_1}\edges(G_j),
\end{equation}
or
\begin{equation}\label{secondOption}
\incidences(\pts^\prime,\Gamma)\leq\frac{1}{64k}\sum_{j\in J_2}\edges(G_j).
\end{equation}

If \eqref{firstOption} holds, then
\begin{equation*}
\begin{split}
\sum_{j\in J_1}\edges(G_j)&\lesssim m^{\frac{1}{k-1}}\incidences(\pts,\Gamma)^{\frac{k-2}{k-1}}  \sum_{j=0}^{\lceil \log_2 m\rceil} 2^{-j/k}\\
&\lesssim  m^{\frac{1}{k-1}}\incidences(\pts,\Gamma)^{\frac{k-2}{k-1}}
\end{split}
\end{equation*}
(recall that the $\lesssim$ notation hides an implicit constant that is allowed to depend on $k$). Thus if  \eqref{firstOption} holds, then
\begin{equation}\label{boundSumJ1}
|\incidences(\pts,\Gamma)| \lesssim  |\pts|,
\end{equation}
which proves Lemma \ref{planarIncidenceLem}.

Alternately, if \eqref{secondOption} holds, then we can apply the crossing lemma to each $j\in J_2$ to conclude
\begin{equation}
\cross(H_j)\gtrsim \frac{\edges(G_j)^3}{m_j^2 \big(2^{j+1}\frac{\incidences(\pts,\Gamma)}{2m}\big)^{\frac{k-2}{k-1}}},
\end{equation}
and thus
\begin{equation}
\begin{split}
\edges(G_j)&\lesssim  (\cross(H_j))^{1/3}m^{\frac{k}{3(k-1)}}\incidences(\pts,\Gamma)^{\frac{k-2}{3(k-1)}}2^{\frac{-jk}{3(k-1)}}\\
&\lesssim  (\cross(H_j))^{1/3}m^{\frac{k}{3(k-1)}}\incidences(\pts,\Gamma)^{\frac{k-2}{3(k-1)}},
\end{split}
\end{equation}
where the implicit constant does not depend on $j$ (i.e.~it is an absolute constant). Thus we have
\begin{equation}
\begin{split}
\incidences(\pts,\Gamma)&\lesssim \sum_{j\in J_2}\edges(G_j)\\
&\lesssim \sum_{j\in J_2} (\cross(H_j))^{1/3}m^{\frac{k}{3(k-1)}}\incidences(\pts,\Gamma)^{\frac{k-2}{3(k-1)}}\\
&\lesssim (\cross(H))^{1/3}m^{\frac{k}{3(k-1)}}\incidences(\pts,\Gamma)^{\frac{k-2}{3(k-1)}}.
\end{split}
\end{equation}
By Lemma \ref{notTooMuchStacking}, we have
\begin{equation*}
\cross(H)\lesssim \sum_{\substack{\gamma,\gamma^\prime\in\Gamma\\\gamma\neq\gamma^\prime}} |\gamma\cap\gamma^\prime|.
\end{equation*}
Thus if \eqref{secondOption} holds, then
\begin{equation}\label{boundJ2Case}
|\incidences(\pts,\Gamma)| \lesssim|\pts|^{\frac{k}{2k-1}} \Big(\sum_{\substack{\gamma,\gamma^\prime\in\Gamma\\\gamma\neq\gamma^\prime}} |\gamma\cap\gamma^\prime| \Big)^{\frac{k-1}{2k-1}}.
\end{equation}

Combining the bounds \eqref{GammaControlsNumIncidences}, \eqref{boundSumJ1}, and \eqref{boundJ2Case}, we conclude

\begin{equation*}
\incidences(\pts,\Gamma)\lesssim |\pts|^{\frac{k}{2k-1}} \Big(\sum_{\substack{\gamma,\gamma^\prime\in\Gamma\\\gamma\neq\gamma^\prime}} |\gamma\cap\gamma^\prime| \Big)^{\frac{k-1}{2k-1}} + |\pts| + |\Gamma|.\qedhere
\end{equation*}
\end{proof}

%
%
%
%
\section{Bounding $I_6$}\label{boundingI6Sec}
To bound $|I_6\backslash I_7|$, we will apply Theorem \ref{planarIncidenceLem} to each collection $(A,\pts_A,\lines_A)$ for each $A\in \bigcup_i \mathcal{Y}_i$. We conclude that
\begin{equation}\label{incidencesI7MinsI8}
\begin{split}
|I_6\backslash I_7| & \lesssim \sum_{i=1}^\ell \sum_{A\in\mathcal{Y}_i} |\{(p,\beta)\in\pts_A\cap\lines_A\colon p\in\beta\}|\\
&\lesssim \sum_{i=1}^\ell \sum_{A\in\mathcal{Y}_i} \big(|\pts_A|^{\frac{k}{2k-1}}\operatorname{crossings}(A)^{\frac{k-1}{2k-1}}+|\pts_A|+|\lines_A|\Big)\\
&\lesssim \bigg(\sum_{i=1}^\ell \sum_{A\in\mathcal{Y}_i} |\pts_A|\bigg)^{\frac{k}{2k-1}}\bigg(\sum_{i=1}^\ell \sum_{A\in\mathcal{Y}_i} \operatorname{crossings}(A)\bigg)^{\frac{k-1}{2k-1}}\\
&\qquad+|\pts|+\sum_{i=1}^\ell \sum_{A\in\mathcal{Y}_i} |\lines_A|\\
&\lesssim m^{\frac{k}{2k-1}}n^{\frac{2k-2}{2k-1}}+m,
\end{split}
\end{equation}
where on the second-last line we used Lemma \ref{boundingCrossingNumLem}.
Combining \eqref{I8BoundEqn}, \eqref{incidencesI7MinsI8}, and the bounds on $I_1,\ldots,I_6$ from Sections \ref{I1I2I3Sec}, \ref{I4Sec}, and \ref{I5Sec}, we obtain the bound
\begin{equation}
\sum_{i\in\mathcal A_1} |I\cap\mathcal{I}(\pts_i \cap W_i,\sphrs_2)|\lesssim m^{\frac{k}{2k-1}}n^{\frac{2k-2}{2k-1}}+m+n,
\end{equation}
where the implicit constant depends only on $C_0$ and $k$ (indeed, each implicit constant only depended on previously defined implicit constants, and ultimately these only depended on $C_0$ and $k$). This is precisely the second term in \eqref{intermediateIncidenceBd}, which we sought to control. All together, we conclude that
\begin{equation}
|I\cap \incidences(\pts\cap Z,\sphrs_2)| \leq \frac{C_1}{10} \Big(m^{\frac{k}{2k-1}}n^{\frac{2k-2}{2k-1}}+m+n\Big),
\end{equation}
provided we choose $C_1$ sufficiently large depending only on the constants $C_0$ and $k$ from the statement of Theorem \ref{surfacesInR4Thm}. This (at last!) concludes the proof of Theorem \ref{surfacesInR4Thm}.

\section{Open problems and future work}\label{openProblemsSec}
There are a number of natural extensions and generalizations of Theorem \ref{surfacesInR4Thm}.
\subsection{Removing the restriction on $m$ and $n$}\label{removingRestrictionSection}
The requirement that $m\leq n^{\frac{2k+2}{3k}}$ is likely not necessary; we pose the following conjecture.
\begin{conjecture}\label{allMNConjecture}
Theorem \ref{surfacesInR4Thm} holds for all values of $m$ and $n$.
\end{conjecture}
If $m\geq n^2$, then Theorem \ref{surfacesInR4Thm} follows from the K\H{o}vari-S\'os-Tur\'an theorem (Theorem \ref{turanBoundLemma}). Thus the critical range is $n^{\frac{2k+2}{3k}} < m < n^2$. The author believes that using the same techniques as in Section 3.1 of \cite{Sheffer}, it would be possible to obtain the following partial progress towards Conjecture \ref{allMNConjecture}.
\begin{conjecture}\label{allMNConjectureWeak}
Let $\pts\subset\RR^4$ be a collection of $m$ points. Let $\sphrs$ be a $C_0$-good collection of pseudoflats with $k$ degrees of freedom, with $|\sphrs|=n$, and suppose $m\leq n^{2-\epsilon}$. Let $I\subset\incidences(\pts,\sphrs)$ be a good collection of incidences. Then
\begin{equation}\label{surfacesInR4ThmIntmEps}
|I| \leq C_1\Big(m^{\frac{k}{2k-1}}n^{\frac{2k-2}{2k-1}}+m+n\Big).
\end{equation}
The constant $C_1$ depends only on $C_0,k$ and $\epsilon$.
\end{conjecture}

Roughly speaking, Conjecture \ref{allMNConjectureWeak} should be provable as follows. In proving Theorem \ref{surfacesInR4Thm}, we construct partitioning polynomials $\{P_i\}$, $\{Q_i\}$ of degrees $D_i,\ E_i$, respectively. As discussed in Remark \ref{whereWeUseAssumtionRemark}, we need the bound
\begin{equation}\label{neededBoundDiE}
\sum_i (D_iE_i)^4 \lesssim m^{\frac{k}{2k-1}}n^{\frac{2k-2}{2k-1}}.
\end{equation}
Here, the numbers $\{D_i\}$ are essentially arbitrary positive integers satisfying $\sum D_i=D$ ($D$ is specified in \eqref{defnOfD}), and $E_i$ is given by \eqref{defnEj}. If $m\leq n^{\frac{2k+2}{3k}}$ then \eqref{neededBoundDiE} holds, while if $m> n^{\frac{2k+2}{3k}}$ then \eqref{neededBoundDiE} may fail.

However, one can get around this problem by using partitioning polynomials of lower degree (i.e.~making $D_i$ and $E_i$ smaller), so \eqref{neededBoundDiE} holds even when $m> n^{\frac{2k+2}{3k}}$. Let $\alpha=\log m/\log n$. The idea is to prove the theorem by induction on $\alpha$, starting with the base case $\alpha\leq  \frac{2k+2}{3k},$ which has already been handled by Theorem \ref{surfacesInR4Thm}.

Now, suppose the theorem has already been proved for all $\alpha<\alpha_0$, and let $\pts,\sphrs$ be collections of points and surfaces with $|\pts|=m,\ |\sphrs|=n$. Suppose that $\log m/\log n\leq\alpha_0+f(\alpha_0).$ The function $f(t)$ will be determined later; the key property is that $f(t)$ is continuous on $[\frac{2k+2}{3k},2]$ and $f(t)>0$ for $x<2$.

Let $D^\prime$ be a small power of $D$ ($D^\prime=D^{1/10}$ say). Instead of performing a partition using a polynomial of degree $D$, use a polynomial of degree $D^\prime$. A certain number of points and surfaces will enter each of the cells. There will be too many points and surfaces to apply the K\H{o}vari-S\'os-Tur\'an theorem directly. Luckily, however, if $m^\prime$ points and $n^\prime$ surfaces enter the cell, then $\log m^\prime / \log n^\prime\leq\alpha_0$ (provided the function $f(t)$ is chosen appropriately) so the induction hypothesis can be applied to bound the number of incidences inside each cell.

We must now bound the number of incidences occurring on the boundary of the partition. Define $E_i^\prime$ to be a small power of $E_i$. The incidences inside the second-level cells can again be bounded using the induction hypothesis.

It remains to bound the incidences occurring on the boundary of the second partition. Here we exploit the fact that $D_i^\prime$ and $E_i^\prime$ are much smaller than $D_i$ and $E_i$. In particular, the analogue of \eqref{neededBoundDiE} will hold with $D_i^\prime$ and $E_i^\prime$ in place of $D_i$ and $E_i$. This allows us to close the induction.

Analyzing the induction, we see that for any $\epsilon>0$, if $\pts,\sphrs$ are collections of points and surfaces with $|\pts|=m,\ |\sphrs|=n$, and if $m\leq n^{2-\epsilon},$ then we only apply the induction step $O_\epsilon(1)$ times before we are reduced to the base case $\alpha\leq\frac{2k+2}{3}$. Each time we apply the induction step we obtain an additional multiplicative constant in our bound. However, since we only perform this induction $O_\epsilon(1)$ times, the total contribution is still (a multiplicative) constant.

However, proving the above result would lengthen the exposition significantly and does not introduce any new ideas, so we prefer to state it as a conjecture rather than include the argument in this manuscript.

\subsection{Higher dimensions}
Extending Theorem \ref{surfacesInR4Thm} to dimensions higher than 4 appears to require some significant new ideas. In particular, if one tried to follow a similar proof strategy to prove an incidence theorem for 3--flats in $\RR^6$, one would need some sort of analogue of the crossing lemma for two-dimensional surfaces in $\RR^4$. The author is not aware of any statement of this type. It seems reasonable to conjecture that any proof of a Szem\'eredi-Trotter type theorem for 3--flats in $\RR^6$ will require a different proof strategy.

\bibliographystyle{amsplain}

\begin{thebibliography}{99}
%
\bibitem{Ajtai} M.~Ajtai, V.~Chv\'atal, M.~Newborn, E.~Szemer\'edi. Crossing-free subgraphs. \emph{Ann.~Discrete Math}. 12:9--12. 1982.
%
%
\bibitem{Arnov} B.~Aronov, V.~Koltun, M.~Sharir. Incidences between points and circles in three and higher dimensions. \emph{Discrete Comput.~Geom}. 33(2):185--206. 2005.
%
%
\bibitem{Barone} S.~Barone, S.~Basu. Refined bounds on the number of connected components of sign conditions on a variety. \emph{Discrete Comput.~Geom}. 47(3): 577--597. 2012.
%
%
\bibitem{Barone2} S.~Barone, S.~Basu. On a real analogue of Bezout inequality and the number of connected of connected components of sign conditions. arXiv:1303.1577v2. 2013.
%
%
\bibitem{Basu} S.~Basu, R.~Pollack, M.~Roy. \emph{Algorithms in real algebraic geometry}. Springer, Berlin. 2006.
%
\bibitem{Bochnak} J.~Bochnak, M.~Coste, M.~Roy. \emph{Real algebraic geometry}. Springer-Verlag, Berlin. 1998.
%
%
\bibitem{Clarkson} K.~Clarkson, H.~Edelsbrunner, L.~Guibas, M.~Sharir, E.~Welzl. Combinatorial complexity bounds for arrangements of curves and spheres. \emph{Discrete Comput.~Geom}. 5(1):99--160. 1990.
%
%
\bibitem{Edelsbrunner} H.~Edelsbrunner, M.~Sharir: A hyperplane incidence problem with applications to counting distances. \emph{The Victor Klee Festschrift, DIMACS Ser. Discret. Math. Theor. Comput. Sci.} 4:253--263. 1991.
%
%
\bibitem{Elekes} G.~Elekes, C.~T\'oth. Incidences of not-too-degenerate hyperplanes. \emph{Computational geometry (SCG'05)}. ACM, New York: 16--21. 2005.
%
%
\bibitem{ElKahoui}M.~El Kahoui. Topology of real algebraic space curves. \emph{J. Symbolic Comput}. 43(4): 235--258. 2008.
%
%
\bibitem{Eisenbud2} D.~Eisenbud \emph{Commutative Algebra: with a View Toward Algebraic Geometry }. Springer, New York NY. 1999.
%
%
\bibitem{Eisenbud} D.~Eisenbud, J.~Harris. \emph{The geometry of schemes}. Springer, New York NY. 2000.
%
%
\bibitem{Erdos}P.~Erd\H{o}s. Problems and results in combinatorial geometry. \emph{Ann.~New York Acad.~Sci.}, 440, New York Acad. Sci., New York, 1985.
%
%
\bibitem{Fulton} W.~Fulton. \emph{Intersection theory}. Ergebnisse der Mathematik und ihrer Grenzgebiete. 3. Folge. Springer-Verlag, Berlin, second edition, 1998.
%
%
\bibitem{Guth} L.~Guth, N.~Katz. On the Erd\H{o}s distinct distance problem in the plane. \emph{Ann. Math.},  181(1): 155--190. 2015.
%
%
%
\bibitem{harris} J.~Harris. \emph{Algebraic geometry: a first course}. Springer, New York, NY. 1995.

\bibitem{Kaplan2} H.~Kaplan, J.~Matou\v{s}ek, Z.~Safernova, M.~Sharir. Unit distances in three dimensions. Combin. Probab. Comput. 21(4): 597--610. 2012.
%
%
%
%
\bibitem{Turan} T.~K\H{o}vari, V.~S\'os, P.~Tur\'an. On a problem of K.~Zarankiewicz. \emph{Colloquium Mathematicum}. 3:50--57. 1954.
%
%
\bibitem{Laba} I.~\L{}aba, J.~Solymosi. Incidence theorems for pseudoflats. \emph{Discrete Comput.~Geom.}  37(2):163--174. 2007.
%
%
%
\bibitem{Leighton} F.~Leighton. \emph{Complexity Issues in VLSI, Foundations of Computing Series}. MIT Press, Cambridge, MA. 1983.
%
%
%
\bibitem{milnor} J.~Milnor. \emph{Singular points of complex hypersurfaces}. Princeton University Press, Princeton NJ. 1969.
%
%
\bibitem{mumford} D.~Mumford. \emph{Algebraic geometry I: complex projective varieties.} Springer, New York, NY. Corr. 2nd printing, 1981.
%
%
\bibitem{Pach} J.~Pach, M.~Sharir. On the number of incidences between points and curves. \emph{Combin. Probab. Comput}. 7(1):121--127. 1998.
%
%
\bibitem{Pach2}  J.~Pach, M.~Sharir. Repeated angles in the plane and related problems. \emph{J. Combin. Theo. Ser. A} 59:12--22. 1992.
%
%
\bibitem{Sheffer} M.~Sharir, A.~Sheffer, J.~Zahl. Improved bounds for incidences between points and circles. \emph{Combin. Probab. Comput.}, 24(3): 490--520, 2015.
%
%
\bibitem{SR}J.G.~Semple, L.~Roth. \emph{Introduction to Algebraic Geometry}. Oxford University Press, 1985.
%
%
\bibitem{Solymosi} J.~Solymosi, T.~Tao. An incidence theorem in higher dimensions. Discrete Comput. Geom. 48(2):255--280. 2012.
%
%
\bibitem{Solymosi3} J.~Solymosi, G.~Tardos, On the number of $k$--rich transformations. \emph{Proceedings of the 23th Annual Symposium on Computational Geometry}, (SoCG 2007), ACM, New York: 227--231. 2007.
%
%
\bibitem{Solymosi2} J.~Solymosi, C.~T\'oth. On distinct distances in homogeneous sets in the Euclidean space. \emph{Discrete Comput.~Geom.} 35:537--549. 2005.
%
%
\bibitem{Szekely} L.~Sz\'ekely. Crossing numbers and hard Erd\H{o}s problems in discrete geometry. \emph{Combin. Probab Comput}. 6(3):353--358. 1997.
%
%
\bibitem{Szemeredi} E.~Szemer\'edi, W.~Trotter. Extremal problems in discrete geometry. \emph{Combinatorica}. 3(3):381--392. 1983.
%
%
\bibitem{Toth} C.~T\'oth. The Szemer\'edi-Trotter theorem in the complex plane. \emph{Combinatorica}, 35(1): 95--126. 2015.
%
%
\bibitem{TothV1} C.~T\'oth. The Szemer\'edi-Trotter theorem in the complex plane. arXiv version 1. arXiv:math/0305283v1. 2003.
%
%
\bibitem{Whitney} H.~Whitney. Elementary structure of real algebraic varieties. \emph{Annals of Math}. 66:545--556. 1957.
%
%
\bibitem{Zahl} J.~Zahl. An improved bound on the number of point-surface incidences in three dimensions. Contrib. Discrete Math. 8(1):100--121. 2013.

\end{thebibliography}

\end{document}